\documentclass[12pt, reqno]{amsart}
\usepackage{amsmath, amsfonts, amssymb}
\newcommand{\dal}{\square}
\newcommand{\R}{{\mathbb R}}

\newcommand{\C}{{\mathbb C}}
\newcommand{\ve}{\varepsilon}
\newcommand{\pa}{\partial}
\newcommand{\jb}[1]{\left\langle #1 \right\rangle}
\newcommand{\ip}[2]{\left\langle #1, #2 \right\rangle}
\newcommand{\tens}[1]{{\bf #1}}
\DeclareMathOperator{\supp}{supp}
\DeclareMathOperator{\real}{Re}
\DeclareMathOperator{\ima}{Im}
\DeclareMathOperator{\sgn}{sgn}
\newtheorem{theorem}{Theorem}
\newtheorem{lemma}[theorem]{Lemma}
\newtheorem{corollary}[theorem]{Corollary}
\newtheorem{proposition}[theorem]{Proposition}
\newtheorem{definition}[theorem]{Definition}
\newtheorem{condition}[theorem]{Condition}
\newtheorem{question}[theorem]{Question}
\newtheorem{remark}[theorem]{Remark}
\newtheorem{example}[theorem]{Example}
\numberwithin{equation}{section}
\numberwithin{theorem}{section}
\begin{document}
\title[Asymptotic Pointwise Behavior for Wave Equations]{%
Asymptotic Pointwise  Behavior  
for Systems  of Semilinear Wave Equations in Three Space Dimensions}
\author{Soichiro Katayama}
\date{}
%\institute{S. Katayama \at Department of Mathematics, Wakayama University, 930 Sakaedani, Wakayama 640-8510, Japan
%\\ Tel. +81-73-457-7343, Fax: +81-73-457-7489
%\\ \email{katayama@center.wakayama-u.ac.jp}
%}
%\titlerunning{Systems of Nonlinear Wave and Klein-Gordon Equations}
%%%%%%
%%%%%%%%%%%%
\begin{abstract}
In connection with the weak null condition,
Alinhac introduced a sufficient condition
for global existence
of small amplitude solutions to systems of semilinear wave equations in three space dimensions.
We introduce a slightly weaker
sufficient condition
for the small data global existence, 
and we investigate the asymptotic pointwise behavior of global solutions 
for systems satisfying 
this condition. 
As an application, the asymptotic behavior
of global solutions 
under the Alinhac condition is also derived. 
\end{abstract}
%%%%%%%%%%%%
% \keywords{Null condition \and Wave equation \and  Global existence \and Asymptotic behavior}
% \subjclass{35L70}
%%%%%%%%%%%%
\maketitle
%%%%%%%%%%%%
%%%%%%%%%%%%%%%%%%%%%%%%%%%%%%%%%%%%%%%%%%%%%%
\baselineskip=0.57cm
%%%%%%%%%%%%%%%%%%%%%%%%%%%%%%%%%%%%%%%%%%%%%%
\section{Introduction}
%%%%%%%%%%%%%%%%%%%%%%%%%%%%%%%%%%%%%%%%%%%%%%
%%%%%%%%%%%%%%%%%%%%%%%%%%%%%%%%%%%%%%%%%%%%%%
This paper is devoted to the study of the Cauchy problem for systems of
semilinear wave equations in three space dimensions.
Throughout this paper, for the variables $t\in \R$ and $x=(x_1,x_2,x_3)\in \R^3$,
we use the notation
$$
 \pa_0=\pa_t=\frac{\pa}{\pa t}, \quad \pa_k=\frac{\pa}{\pa {x_k}},\ k=1,2,3.
$$
$\Delta_x$ and $\dal$ denote the Laplacian and the d'Alembertian, respectively; 
namely we define $\Delta_x=\sum_{k=1}^3 \pa_{k}^2$ and $\dal=\pa_t^2-\Delta_x$.
For a matrix (or vector) $\tens{B}$, its transpose is written as $\tens{B}^{\rm T}$.

We consider the Cauchy problem for systems of semilinear wave equations of the type
\begin{equation}
\label{OurSys}
\dal u_j=F_j({u}, \pa {u}) 
\quad \text{in $(0,\infty)\times \R^3$}, \quad j=1,2, \ldots, N
\end{equation}
with initial data
\begin{equation}
\label{Data}
{u}(0,x)=\ve {f}(x),\ (\pa_t {u})(0, x)=\ve {g}(x) \quad \text{for $x\in \R^3$},
\end{equation}
where ${u}=(u_1, \ldots, u_N)^{\rm T}$ and 
$\pa {u}=(\pa_0 {u}, \pa_1 {u}, \pa_2 {u}, \pa_3 {u})$.
Here each $u_j$ is supposed to be a real-valued function of $(t,x)\in [0, \infty)\times \R^3$.
%%%%%%%%%%
We assume that ${f}=(f_j)_{1\le j\le N}^{\rm T}$, ${g}=(g_j)_{1\le j\le N}^{\rm T}\in
C^\infty_0(\R^3; \R^N)$ in \eqref{Data}, and that $\ve$ is a small and positive parameter.
We refer to $({f}, {g})$ as the {\it initial profile} in what follows.
Throughout this paper, we always assume for simplicity that 
each $F_j$ with $1\le j\le N$ is a homogeneous polynomial of degree $2$ in its arguments $(u,\pa u)$.

We say that the {\it small data global existence} (or SDGE
in short) holds
for \eqref{OurSys} if for any
${f}, {g}\in C^\infty_0(\R^3; \R^N)$, there exists a positive constant $\ve_0$ such that for any
$\ve\in (0, \ve_0]$ the Cauchy problem \eqref{OurSys}-\eqref{Data} possesses
a global classical solution ${u}\in C^\infty([0,\infty)\times \R^3; \R^N)$. 
It is known that SDGE 
does not hold for 
general quadratic nonlinearity 
(see John \cite{Joh79} and \cite{Joh81}).
Klainerman \cite{Kla86} introduced the {\it null condition}, and
proved SDGE for \eqref{OurSys}
under the null condition (see also Christodoulou \cite{Chr86}). To state the null condition, we introduce the
{\it reduced nonlinearity}
$$
F_j^{\rm red} (\omega, {X}, {Y})
:=F_j\left( {X}, \left(
                -{Y}, \omega_1 {Y}, \omega_2 {Y}, \omega_3 {Y}
               \right)
               \right),\quad 1\le j\le N
$$
for $\omega=(\omega_1, \omega_2, \omega_3)\in S^2$,
${X}=(X_1, \ldots, X_N)^{\rm T}\in \R^N$, and ${Y}=(Y_1, \ldots, Y_N)^{\rm T}
\in \R^N$, where $S^2$ denotes the unit sphere in $\R^3$.
In other words, $F_j^{\rm red}$ is obtained by
substituting $X_k$ and $\omega_a Y_k$ (with $1\le k\le N$ and $0\le a \le 3$)
for $u_k$ and $\pa_a u_k$ in $F_j({u},\pa {u})$, respectively, where $\omega_0=-1$.
We say that ${F}=(F_j)_{1\le j\le N}^{\rm T}$ satisfies the null condition if
\begin{equation}
\label{NC00}
F_j^{\rm red} (\omega, X, Y)=0, \quad 1\le j\le N, \ X,Y\in \R^N,\ \omega\in S^2.
\end{equation}
%for any $X, Y\in\R^N$, and any $\omega\in S^2$.
We introduce the null forms
\begin{align}
Q_0(\varphi, \psi):=& (\pa_t\varphi)(\pa_t\psi)-\sum_{k=1}^3
(\pa_k \varphi)(\pa_k \psi),
\label{QT01}\\
Q_{ab}(\varphi, \psi):=&(\pa_a\varphi)(\pa_b \psi)
{}-(\pa_b\varphi)(\pa_a\psi),\quad
0\le a<b\le 3.
\label{QT02}
\end{align}
When each $F_j$ 
is a homogeneous polynomial of degree $2$ in $(u, \pa u)$, as is assumed, we can show that 
$F=(F_j)_{1\le j\le N}^{\rm T}$ satisfies the null condition if and only if $F$ 
can be written as 
\begin{equation}
\label{Channie}
F_j(u,\pa u)=\sum_{k, l=1}^{N} \biggl(r_j^{0,kl}Q_0(u_k, u_l)
{}+\sum_{a,b=0}^3 r_j^{ab, kl} Q_{ab}(u_k, u_l)\biggr),\quad 1\le j\le N
\end{equation}
with some constants $r_j^{0, kl}$ and $r_j^{ab, kl}$
(see \cite{Kla86} for instance).
We can also show that if the null condition is satisfied,
then the global solution $u$ to
\eqref{OurSys}-\eqref{Data} for small $\ve$ is asymptotically free
in the energy sense, that is to say,
there exists a solution $\widetilde{u}$ to 
the free wave equation
$\dal \widetilde{u}=0$ with some data in $\dot{H}^1(\R^3)\times L^2(\R^3)$ such that
\begin{equation}
\label{AsympFreeE-0}
\lim_{t\to\infty}
 \left\|\left(u-\widetilde{u}\right)(t,\cdot)\right\|_E=0,
\end{equation}
where $\dot{H}^1(\R^3)$ is the homogeneous Sobolev space,
and the energy norm $\|\cdot\|_E$ is given by
\begin{equation}
\label{DefEnergyNorm}
 \|\varphi(t,\cdot)\|_E:=
\left(\frac{1}{2}
 \int_{\R^3} \left(|\pa_t \varphi(t,x)|^2+|\nabla_x \varphi(t,x)|^2 \right)dx
 \right)^{1/2}.
\end{equation}
Since we have the conservation of the energy $\left\|\widetilde{u}(t,\cdot)\right\|_E=\left\|\widetilde{u}(0,\cdot)\right\|_E$
for the free solution $\widetilde{u}$,
\eqref{AsympFreeE-0} implies
\begin{equation}
\label{AsympFreeE}
\lim_{t\to\infty}
 \left\|u(t,\cdot)\right\|_E=\left\|\widetilde{u}(0,\cdot)\right\|_E.
\end{equation}

In order to understand the null condition,
we introduce the {\it reduced system}.
For this purpose, first we recall the asymptotic behavior of the free solution:
Let $\varphi, \psi \in C^\infty_0(\R^3; \R)$, and
consider the Cauchy problem for the single wave equation
\begin{align}
\label{LinearWave}
&\dal u_0(t,x)=0 & & \text{for $(t,x)\in (0,\infty) \times \R^3$,}\\
\label{LinearData}
& u_0(0,x)=\varphi(x),\ (\pa_t u_0)(0,x)=\psi(x)
& & \text{for $x\in \R^3$}.
\end{align}
It is well known that the solution $u_0$ to \eqref{LinearWave}-\eqref{LinearData}
can be written as
\begin{equation}
\label{ExpressFreeSol02}
u_0(t,x)
=\frac{1}{4\pi t}\int_{|y-x|=t} \psi (y) dS_y
{}+\pa_t\left(\frac{1}{4\pi t} \int_{|y-x|=t} \varphi(y) dS_y\right), 
\end{equation}
where $dS_y$ is the area element of the sphere with radius $t$ centered at $x$.
Suppose that
$$
\supp \varphi \cup \supp \psi \subset B_R:=\{x\in \R^3;\ |x|\le R\}
$$
with a positive constant $R$.
Then  
\eqref{ExpressFreeSol02}
implies that $u_0(t,x)=0$ for any $(t,x)\in [0,\infty)\times \R^3$ satisfying
$|r-t|\ge R$ with $r=|x|$ (this property is called the {\it (strong) Huygens principle}). Hence it is reasonable
to consider the asymptotic behavior of $u_0$ for large $t$
(or equivalently large $r$) with $r-t$ being fixed;
then we can expect that the integrals over the sphere of radius $t$
centered at $x=r\omega$ (with $\omega=|x|^{-1}x$) in \eqref{ExpressFreeSol02}
tend to those over the tangential plane of the sphere at the point $(r-t)\omega$.
This is the motivation to introduce 
the {\it Friedlander radiation field}
\begin{equation}
\label{FriedRad}
{\mathcal F}_0\left[\varphi, \psi\right](\sigma, \omega)
=\frac{1}{4\pi}
\left({\mathcal R}\left[\psi\right](\sigma, \omega)
{}-\bigl(\pa_\sigma {\mathcal R}\left[\varphi\right]\bigr)(\sigma, \omega)
\right)
\end{equation}
for $(\sigma, \omega)\in \R\times S^2$, where
${\mathcal R}[h]$ is the Radon transform of a function $h$
given by
\begin{equation}
\label{RadTrans}
{\mathcal R}[h](\sigma, \omega)=
\int_{y\cdot \omega=\sigma} h(y) dS_y'.
\end{equation}
Here $dS_y'$ is the area element of the plane
$\{y\in \R^3;\ y\cdot\omega=\sigma\}$.
It is known that for the solution $u_0$ to \eqref{LinearWave}-\eqref{LinearData} we have 
\begin{align}
& \left|ru_0(t,r\omega)-{\mathcal F}_0\left[\varphi, \psi\right](r-t, \omega)
   \right| \nonumber\\
& \qquad {}+\sum_{a=0}^3 \left|
                                    r(\pa_a u_0)(t, r\omega)-\omega_a 
                                         \left(\pa_\sigma 
                                              {\mathcal F}_0
                                                \left[ \varphi, \psi\right]
                                         \right)(r-t,\omega)
                                    \right|
\le C(1+t+r)^{-1}
\label{FriedAsymp}
\end{align}
for any $\omega=(\omega_1, \omega_2, \omega_3)\in S^2$ and any $(t,r)$ 
satisfying $r\ge t/2\ge 1$, 
where we have set $\omega_0=-1$ as before,
and $C$ is a positive constant determined by $\varphi$ and $\psi$
(see H\"ormander \cite[Theorem 6.2.1]{Hoe97} for instance; see also Katayama-Kubo \cite{Kat-Kub09a}).

Let $u$ be the solution to \eqref{OurSys}-\eqref{Data}.
Motivated by the Friedlander radiation field, and taking account of the lifespan of the solution
for the case of general quadratic nonlinearity, we seek an approximation of the solution $u$ of the form
$$
%u(t, r\omega) \sim 
\ve r^{-1} U\left(\ve \log t, r-t, \omega\right),
$$
which approximates $u(t, r\omega)$ as $\ve$ tends to $0$, with $\tau=\ve \log t$, 
$\sigma=r-t$, and $\omega(\in S^2)$ being fixed. 
Formal calculations show that $U=U(\tau, \sigma, \omega)
=\left(U_j(\tau,\sigma, \omega)\right)_{1\le j\le N}^{\rm T}$
should be determined by
\begin{equation}
\label{RedSys}
-2 \pa_\tau\pa_\sigma U_j(\tau, \sigma, \omega)= 
F_j^{\rm red} \left(\omega, U(\tau,\sigma,\omega), \pa_\sigma U(\tau, \sigma, \omega)\right),\quad 1\le j\le N,
\end{equation}
which is called the {\it reduced system} 
(the quasi-linear version of this system was successfully used in the
detailed study of the lifespan of the solution; see John \cite{Joh87} and 
H\"ormander \cite{Hoe87}).
Now we understand that the null condition \eqref{NC00} says that the right-hand side of \eqref{RedSys} vanishes identically, and we can solve \eqref{RedSys} globally. From this observation,
Lindblad-Rodnianski \cite{Lin-Rod03} introduced the notion of the {\it weak null condition},
which means that the reduced system
always has global solutions with at most exponential growth in $\tau$
(see also \cite{Lin-Rod05});
it is conjectured that the weak null condition implies SDGE, 
but this conjecture is still open.

In connection with the weak null condition,
Alinhac \cite{Ali06} considered the case of $F=F(\pa u)$, 
and introduced a sufficient condition for SDGE,
which is stronger than the weak null condition, but still weaker than the null condition. 
For simplicity of exposition, when $F_j=F_j(\pa u)$,
we omit $X$ in $F_j^{\rm red}(\omega, X, Y)$ and write $F_j^{\rm red}=F_j^{\rm red} (\omega, Y)$ for $1\le j\le N$ in what follows.
If we write $\omega_0=-1$, $\omega=(\omega_1, \omega_2, \omega_3)\in S^2$,
and $Y=(Y_j)_{1\le j\le N}^{\rm T}\in \R^N$,
then his condition
 can be read as follows:
\begin{condition}[Alinhac~\cite{Ali06}]\label{AlinhacCond}
\normalfont
There exist a real valued function $M=M(\omega, Y)$, 
an $\R^N$-valued function $\beta=\beta(\omega)=\left(\beta_j(\omega)\right)_{1\le j\le N}^{\rm T}$,
a positive integer $N_0$,
and linear forms $g_{jl}=g_{jl}(\omega, Y)$ and $h_l=h_l(\omega,Y)$ in $Y$,
which can be written as
\begin{align}
\label{Yone}
g_{jl}(\omega, Y)= & \sum_{k=1}^N g_{jl, k}(\omega) Y_k, 
  & & 1\le j\le N, \ 1\le l\le N_0,\\
\label{Mil}
h_l(\omega, Y)= & \sum_{a=0}^3\sum_{k=1}^N h_{l, ka} \omega_a Y_k, 
  & & 1\le l\le N_0
\end{align}
with smooth coefficients $g_{jl, k}=g_{jl, k}(\omega)$ and real constants
$h_{l, ka}$,
such that
\begin{align}
\label{AA-1}
& F_j^{\rm red} (\omega, Y)=M(\omega, Y) \beta_j(\omega),
& & 1\le j\le N,\ (\omega, Y)\in S^2\times \R^N,\\
\label{AAB-1}
& F_j^{\rm red} (\omega, Y)= \sum_{l=1}^{N_0}
g_{jl}(\omega, Y) h_l(\omega, Y), 
& & 1\le j\le N,\ (\omega, Y)\in S^2\times \R^N,
\\
\label{AAB-2}
& h_l \left(\omega, \beta(\omega)\right) =0,
& & 1\le l\le N_0,\ \omega\in S^2.
\end{align}
\end{condition}
%%%%%%%%%%%%%%%%%%%%%%%%
\begin{remark}
\normalfont
\eqref{AA-1}, \eqref{AAB-1}, and \eqref{AAB-2} yield
\begin{equation}
\label{AA-2}
M\left(\omega, \beta(\omega)\right)=0, \quad \omega\in S^2.
\end{equation}
The weak null condition follows from \eqref{AA-1} and \eqref{AA-2},
but it is not known if these two conditions \eqref{AA-1} and \eqref{AA-2}
are sufficient for SDGE (see Alinhac~\cite{Ali06}).
\end{remark}
%%%%%%%%%%%%%%%%%%
It is easy to see that the null condition implies Condition~\ref{AlinhacCond}.
A simple example satisfying 
Condition~\ref{AlinhacCond} but not the null condition is
\begin{equation}
\label{simplestEx}
\begin{cases}
\dal u_1=(\pa_1u_1)(\pa_1u_2-\pa_2u_1),\\
\dal u_2=(\pa_2u_1)(\pa_1u_2-\pa_2u_1)
\end{cases}
\end{equation}
in $(0,\infty)\times \R^3$. We have 
$F_1^{\rm red}=\omega_1 Y_1 (\omega_1Y_2-\omega_2 Y_1)$
and $F_2^{\rm red}=\omega_2 Y_1 (\omega_1Y_2-\omega_2 Y_1)$, and we find that
Condition~\ref{AlinhacCond} is satisfied with
\begin{align*}
& M(\omega, Y)=Y_1 (\omega_1Y_2-\omega_2 Y_1), \ \beta(\omega)=(\omega_1, \omega_2)^{\rm T}, \\
& N_0=1,\ g_{11}(\omega, Y)=\omega_1Y_1, \ g_{21}(\omega, Y)=\omega_2 Y_1,
\ h_1(\omega, Y)=\omega_1Y_2-\omega_2 Y_1.
\end{align*}
If we put $w=\pa_1 u_2-\pa_2 u_1$, which corresponds to $h_1$ above, then we obtain
\begin{equation}
\label{simplestExR}
\begin{cases}
\dal u_1=w(\pa_1u_1),\\
\dal u_2=w(\pa_2u_1),\\
\dal w=\pa_1\dal u_2-\pa_2 \dal u_1=Q_{12}(w, u_1).
\end{cases}
\end{equation}
As we will see later, this hidden structure plays an important role in deriving global solutions. 
Concerning the asymptotic behavior,
Katayama-Kubo \cite{Kat-Kub08:02} showed that global solutions 
under Condition~\ref{AlinhacCond} 
may not be asymptotically free in the energy sense:
For example, it is shown that for some initial profile,
there exists a positive constant $C$ such that 
\begin{equation}
\label{EnergyGrowth}
\|u(t)\|_E\ge C\ve (1+t)^{C \ve} \quad \text{ for small $\ve$},
\end{equation}
where $u=(u_1, u_2)^{\rm T}$ is the global solution to
\eqref{simplestEx} (or equivalently $(u_1, u_2, w)^{\rm T}$ is the global solution to \eqref{simplestExR}).
The estimate \eqref{EnergyGrowth} makes a sharp contrast to \eqref{AsympFreeE},
and the solution $u$ cannot be asymptotically free in the energy sense for such data.

Now the following questions arise:
\begin{question}\label{FirstQ}
\normalfont
We know that sometimes the
global solution to \eqref{OurSys} is asymptotically free, 
and sometimes not so with increasing energy. 
Do we have other kind of the asymptotic behavior?
Especially, is there some nonlinearity $F$ such that the global solution 
to \eqref{OurSys} for some small data behaves differently from free solutions,
although its energy stays bounded from above and below by positive constants?
\end{question}
%%%%%%%%%%%%%%%%%%
\begin{question}\label{SecondQ}
\normalfont
In addition to 
Condition~\ref{AlinhacCond} from Alinhac~\cite{Ali06},
what condition do we need
in order to ensure that the global solutions with small data behave like free solutions?
\end{question}
In this paper, 
motivated by the Friedlander radiation field,
we will investigate the asymptotic pointwise behavior of global solutions 
for large $t$ with $r-t$ and $\omega$ being fixed, under 
a certain condition which is related to 
Condition~\ref{AlinhacCond}, and we will answer the two questions above from this point of view.

Throughout this paper, as usual, various positive constants will
be indicated just by the same letter $C$, and its actual value 
may change line by line.

%%%%%%%%%%%%%%%%%%%%%%%%%%%%%%%%%%%%%%%%%%%%%%
\section{The Main Results}
%%%%%%%%%%%%%%%%%%%%%%%%%%%%%%%%%%%%%%%%%%%%%%
\subsection{Notation}\label{notation}
First we introduce some notation.
We define the vector fields
\begin{align}
S := & t\pa_t+x\cdot \nabla_x, \\
L = & (L_1,L_2,L_3):=t\nabla_x+x\pa_t=(t\pa_j+x_j\pa_t)_{1\le j\le 3}, \\
\Omega = & (\Omega_1, \Omega_2, \Omega_3):=x\times \nabla_x
=(x_2\pa_3-x_3\pa_2, x_3\pa_1-x_1\pa_3, x_1\pa_2-x_2\pa_1), 
\label{Rotation}\\
\pa := & 
(\pa_0, \pa_1, \pa_2, \pa_3),
\end{align}
where $\nabla_x=(\pa_1, \pa_2, \pa_3)$. Here the symbols ``\,$\cdot$\,'' and
``$\times$'' denote the inner and exterior products in $\R^3$,
respectively.
We put
\begin{equation}
\Gamma=(\Gamma_0, \Gamma_1, \ldots, \Gamma_{10})
=(S, L, \Omega, \pa)=\left(S, (L_j)_{1\le j\le 3}, (\Omega_j)_{1\le j\le 3}, (\pa_a)_{0\le a\le 3}\right),
\end{equation}
and we write $\Gamma^\alpha=\Gamma_0^{\alpha_0}\Gamma_1^{\alpha_1}\cdots \Gamma_{10}^{\alpha_{10}}$ with a multi-index $\alpha=(\alpha_0, \alpha_1, \ldots, \alpha_{10})$.
For a nonnegative integer $s$ and a smooth function $\varphi=\varphi(t,x)$, we define
\begin{equation}
\label{InvariantNorm}
|\varphi(t,x)|_s=\sum_{|\alpha|\le s} |\Gamma^\alpha \varphi(t,x)|,
\text { and }
\|\varphi(t,\cdot)\|_s=\left(\sum_{|\alpha|\le s} \|\Gamma^\alpha \varphi(t,\cdot)\|_{L^2(\R^3)}^2\right)^{1/2}.
\end{equation}

For $(\varphi, \psi)\in \dot{H}^1(\R^3)\times L^2(\R^3)$,
we define
\begin{equation}
{\mathcal U}_0[\varphi, \psi](t,x):=u_0(t,x),
\end{equation}
where $u_0$ is the solution
to \eqref{LinearWave}-\eqref{LinearData} with 
$$
(u_0, \pa_t u_0)\in C\bigl([0,\infty); \dot{H}^1(\R^3)\bigr)
\times C\bigl([0,\infty); L^2(\R^3)\bigr).
$$ 
Here $\dot{H}^1(\R^3)$ denotes the homogeneous Sobolev space which is the completion of $C^\infty_0(\R^3)$ with respect to the norm $\|\varphi\|_{\dot{H}^1(\R^3)}=\|\nabla_x \varphi\|_{L^2(\R^3)}$.

For $z\in \R^d$ with a positive integer $d$, the notation $\jb{z}=\sqrt{1+|z|^2}$ will be used throughout this paper. 
As in the introduction, we always put 
$$
\omega_0=-1
$$
in what follows.

For the later convenience we allow $f$ and $g$ in \eqref{Data} to
depend on $\ve$. More precisely,
let ${\mathcal X}_N$ be the set of all mappings 
$$
(f,g): [0,1]\ni \ve \mapsto \bigl(f(\cdot; \ve), g(\cdot; \ve)\bigr)\in C^\infty_0(\R^3; \R^N)
 \times C^\infty_0(\R^3;\R^N)
$$
having the following two properties:
\begin{enumerate} 
\item 
There is some $R>0$, depending on $(f,g)$, such that
$\bigl(f(x;\ve), g(x;\ve)\bigr)=(0,0)$ for $|x|\ge R$ and $\ve\in [0,1]$.
\item 
For any nonnegative integer $s$, we have
$$
\sup_{\ve\in (0,1], \, x\in \R^3}\ve^{-1} \sum_{|\alpha|\le s} %\sup_{x\in \R^3}
\left(\left|\pa_x^\alpha \bigl(f(x;\ve)-f(x;0)\bigr)\right|
{}+\left|\pa_x^\alpha \bigl(g(x;\ve)-g(x;0)\bigr)\right|\right)<\infty.
$$
\end{enumerate}
Here $\pa_x=(\pa_1,\pa_2,\pa_3)$, and we have used the standard notation of multi-indices.
We replace the initial condition \eqref{Data} with
\begin{equation}
\label{Data0}
u(0,x)=\ve f(x; \ve),\ (\pa_t u)(0,x)=\ve g(x; \ve) \quad \text{for $x\in \R^3$},
\end{equation}
where $(f,g)\in {\mathcal X}_N$, and $\ve$ is positive and small.
%%%%%%%%%%%%%%%%%%
\subsection{Basic assumption and examples}
To state the main condition for our results, we introduce the following equivalence relation, which is motivated by the enhanced decay 
estimate for the null forms (see Lemma~\ref{NullNull} below):
\begin{definition}
\label{EquiNL}
Let $F$ in \eqref{OurSys} be given, and let $D\subset {\mathcal X}_N$.

For $\Phi(\omega, u, \pa u)$ and $\Psi(\omega, u, \pa u)$,
which are homogeneous polynomials of degree $2$ in $(u, \pa u)$ 
with smooth coefficients depending on $\omega\in S^2$,
we write 
$$
 \Phi(\omega, u, \pa u) \stackrel{D}{\sim} \Psi(\omega, u, \pa u),
$$
if for any nonnegative integer $s$ there exists a positive constant $C_s$
such that the following property holds:
If $u=u(t,x)$ satisfies \eqref{OurSys}-\eqref{Data0} for $0\le t<T$ with some $(f,g)\in D$, $\ve\in (0,1]$, and $T>0$, then
it holds that
\begin{align}
& \left|\Phi\left(|x|^{-1}x, u(t,x), \pa u(t, x)\right)-\Psi\left(|x|^{-1}x, u(t, x), \pa u(t, x)\right)\right|_s
\nonumber\\
& \quad \le C_s \jb{t+|x|}^{-1}\left(|u(t, x)|_{[s/2]+1}
|\pa u(t,x)|_s+|\pa u(t, x)|_{[s/2]} |u(t,x)|_{s+1}\right)
\label{Tama}
\end{align}
for any $(t,x)\in [0, T)\times \R^3$ satisfying $|x| \ge t/2\ge 1$.
\end{definition}
%%%
Now, motivated by \eqref{simplestExR}, 
we introduce the following condition
which we put on
the nonlinearity $F(u,\pa u)=\left(F_j(u, \pa u)\right)_{1\le j\le N}^{\rm T}$
in our theorems:
%%%%%%%%%%%%%%%%%%%%%%%%%%%%%%%%%%%%%%%%%%%%%%
\begin{condition}\label{OurCond}
\normalfont
Setting 
\begin{align}
\left(v^{\rm T},w^{\rm T}\right)=&(v_1, \ldots, v_{N'}, w_1, \ldots, w_{N''})
\nonumber\\
            :=& (u_1, \ldots, u_{N'}, u_{N'+1}, \ldots u_N)
=u^{\rm T}
\label{DU02}
\end{align}
with $N'+N''=N$,
we can write $F$ as
\begin{equation}
\label{Form01}
F_j(u, \pa u)=
\begin{cases}
F_j^1(\pa u)+F_j^2(u, \pa u),
& 1\le j \le N', \\
F_j^1(\pa u), & N'+1\le j\le N,
\end{cases}
\end{equation}
where each $F_j^1$ is a homogeneous polynomial of degree $2$ in $\pa u$, and 
$$ 
F_j^2(u, \pa u)=\sum_{a=0}^3 \sum_{k=1}^{N'}\sum_{l=1}^{N''} p_j^{ka, l}w_l(\pa_av_k)
$$
with some constants $p_j^{ka, l}$;
furthermore, there exist some subset $D$ of ${\mathcal X}_N$,
and some homogeneous polynomials $G_j(\omega, u, \pa u)$ ($1\le j\le N'$) of degree $2$
in $(u, \pa u)$, which have the form
\begin{align}
\label{AssA}
 G_j(\omega, u, \pa u)=& \sum_{a=0}^3
 \sum_{k=1}^{N'}\sum_{l=1}^{N''}
  \left(c_{jk}^{a, l}(\omega) w_{l}+\sum_{b=0}^3d_{jk}^{\, a, lb}(\omega)
 (\pa_b w_{l})\right) 
(\pa_a v_k)
\end{align}
with smooth coefficients $c_{jk}^{a, l}$ and $d_{jk}^{\, a, l b}$ on $S^2$,
such that
\begin{equation}
\label{Form02}
F_j(u, \pa u)\stackrel{D}{\sim}
\begin{cases}
G_j(\omega, u, \pa u), & 1\le j\le N',\\
0, & N'+1\le j\le N.
\end{cases}
\end{equation}
%%%%%%%%%%%%%%%%%%%%%%%%%%%%%%%%%%%%%%%%%%%%%%
\end{condition}
%%%%%%%%%%%%%%%%%%%%%%%%%%%%%%%%%%%%%%%%%%%%%

It may seem difficult to check Condition~\ref{OurCond} because
it contains the relation given by the inequality \eqref{Tama}.
Here we give two kinds of algebraic conditions to ensure 
Condition~\ref{OurCond}:
\begin{proposition}\label{Example}
{$(1)$} 
Let 
$u=(u_1, \ldots, u_N)^{\rm T}$, and
let $v$ and $w$ be given by \eqref{DU02}.
Suppose that $F=(F_j)_{1\le j\le N}$ has the form
\begin{equation}
\label{Ulrich}
F_j(u, \pa u)
=
\begin{cases}
F_j^{0}(\pa u)+\widetilde{G}_j(u, \pa u),
& 1\le j\le N', \\
\displaystyle F_j^{0}(\pa u), & N'+1\le j\le N,
\end{cases}
\end{equation}
where $\widetilde{G}_j$ is given by
$$
\widetilde{G}_j(u,\pa u)=\sum_{a=0}^3\sum_{k=1}^{N'}\sum_{l=1}^{N''} 
 \left(\widetilde{c}_{jk}^{\, a, l} w_{l}+\sum_{b=0}^3\widetilde{d}_{jk}^{\, a, lb} (\pa_b w_{l}) \right)
 (\pa_a v_k), \quad 1\le j\le N'
$$
with some constants $\widetilde{c}_{jk}^{\, a, l}$ 
and $\widetilde{d}_{jk}^{\, a, lb}$,
while $F^{0}=(F_j^{0})_{1\le j\le N}^{\rm T}$ satisfies the null condition.
Then Condition~$\ref{OurCond}$ is satisfied 
with $D={\mathcal X}_N$ and
$G_j(\omega, u, \pa u)=\widetilde{G}_j(u, \pa u)$ for $1\le j\le N'$.
\smallskip\\
{$(2)$} Let 
 $F_j=F_j(\pa u)$ be a homogeneous 
polynomial of degree $2$ in its arguments for $1\le j\le N$.
Then the system \eqref{OurSys} satisfying Condition~$\ref{AlinhacCond}$
with the initial condition
\eqref{Data0} for $(f,g)\in {\mathcal X}_N$ can be reduced to another system of some size $N^*(\ge N)$,
for which Condition~$\ref{OurCond}$ is satisfied 
with appropriately chosen $D (\subset {\mathcal X}_{N^*})$.
\end{proposition}

We will prove Proposition \ref{Example} in Section \ref{Ex3}.
Note that the null condition implies Condition~\ref{OurCond}
with $D={\mathcal X}_N$, because \eqref{Ulrich} holds with $\widetilde{G}_j\equiv 0$.
%%%%%%%%%%%%%%%
\begin{remark}
\normalfont
Our condition can be slightly weakened: We can add
such terms as $w_l(\pa_a w_{m})$ to $F_j^2$ for $1\le j\le N'$ in \eqref{Form01},
and such ones as $w_l(\pa_a w_{m})$ and $(\pa_b w_l)(\pa_a w_{m})$,
with smooth coefficients on $S^2$, to $G_j$ for $1\le j\le N'$ in \eqref{AssA}
(here $l$ and $m$ run from $1$ to $N''$;
$a$ and $b$ from $0$ to $3$).
These terms have been omitted just for simplicity of exposition. 
Indeed, in order to treat them, we only need to duplicate the equations for
$w$, add them 
to the original system for $u$, and regard {original} $u$ and {duplicated} $w$ 
as {new} $v$ and $w$, respectively (namely we put $v_{N'+m}=w_m$ for $1\le m\le N''$,
so that we have $w_l(\pa_a w_m)=w_l(\pa_a v_{N'+m})$ and so on);
we can see that Condition~\ref{OurCond}
is satisfied for this extended system with $N+N''$ components. 
The same is true for $(1)$ of Proposition~\ref{Example}.
\end{remark}
%%%%%%%%%%%%%%%%%%
\subsection{Global existence}
We define $\R_+:=[0,\infty)$.
Concerning the existence of global solutions, we have the following:
\begin{theorem}\label{GE}
We fix a positive integer $m\ge 5$. We also fix two positive constants
$\lambda$ and $\rho$ satisfying
$0<\lambda<1/20$ and $1/2<\rho\le 1-6\lambda$.
Suppose that Condition~$\ref{OurCond}$ 
is fulfilled.
Let $D(\subset {\mathcal X}_N)$ be from Condition~$\ref{OurCond}$.
%Then, given $(f, g)\in D$, 
If $(f,g)\in D$, then
there exists a positive constant $\ve_0(\le 1)$
such that for every $\ve \in (0, \ve_0]$ the Cauchy problem \eqref{OurSys}-\eqref{Data0}
admits a unique global solution 
$u\in C^\infty(\R_+\times\R^3; \R^N)$.
Moreover there exists a positive constant $C$ such that
\begin{align}
& 
\sup_{(t,x)\in \R_+\times \R^3} \jb{t+|x|}
\Bigl\{\jb{t+|x|}^{-\lambda}|v(t,x)|_{m+1}
\nonumber\\
& \qquad\qquad\qquad\qquad\qquad\qquad\qquad
{}+\jb{t-|x|}^\rho|w(t,x)|_{m+2}\Bigr\}\le C\ve,
\label{Fukkie}
\\
\label{Chato}
& 
\sup_{t\in \R_+} \left\{(1+t)^{-\lambda} \|\pa v(t, \cdot)\|_{2m}
{}+\|\pa w(t,\cdot)\|_{2m}\right\}\le C\ve 
\end{align}
for $\ve\in (0, \ve_0]$, where $v=(v_1, \ldots, v_{N'})^{\rm T}$ and $w=(w_1, \ldots, w_{N''})^{\rm T}$
are given by \eqref{DU02}.

If we put 
$$
\widetilde{v}_j=v_j-\ve\, {\mathcal U}_0[f_j, g_j]
\text{ and }
\widetilde{w}_k=w_k-\ve\, {\mathcal U}_0[f_{N'+k}, g_{N'+k}]$$ 
for $1\le j\le N'$ and $1\le k\le N''$, then
we also have 
\begin{align}
& \sup_{(t,x)\in \R_+\times \R^3} \jb{t+|x|}
\Bigl\{\jb{t+|x|}^{-\lambda}|\widetilde{v}(t,x)|_{m+1}
\nonumber\\
& \qquad\qquad\qquad\qquad\qquad\qquad\qquad
{}+\jb{t-|x|}^\rho|\widetilde{w}(t,x)|_{m+2}\Bigr\}\le C\ve^2,
\label{Fukkie-t}
\\
\label{Chato-t}
& \sup_{t\in \R_+} \left\{(1+t)^{-\lambda} \|\pa \widetilde{v}(t, \cdot)\|_{2m}
{}+\|\pa \widetilde{w}(t,\cdot)\|_{2m}\right\}\le C\ve^2. 
\end{align}
Here ${\mathcal U}_0[f_j,g_j]$ means ${\mathcal U}_0[f_j(\cdot;\ve), g_j(\cdot;\ve)]$
for $1\le j\le N$.
\end{theorem}
%%%%%%%%%%%%%%%%%%%%%%%%%%%%%%%%%%%%%%%%%%%%%%
This result can be proved by some modification of the arguments
in Alinhac \cite{Ali06} (see also \cite{Kat-Kub08:02}). 
We will give the proof of Theorem~\ref{GE} in Section \ref{ProofGlobal}.
The estimates \eqref{Fukkie}--\eqref{Chato-t} play important roles in the 
proof of our next theorem on the pointwise behavior of the solutions.
%%%%%%%%%%%%%%%%%%%%%%%%%%%%
\subsection{Asymptotic pointwise behavior}\label{SAPB}
Suppose that Condition~\ref{OurCond} is fulfilled. 
Let an $\R^{N''}$-valued function
$$
\zeta=\zeta(\sigma, \omega)=\left(\zeta_1(\sigma, \omega),
\ldots, \zeta_{N''}(\sigma, \omega) \right)^{\rm T}
$$
of $(\sigma,\omega)\in \R\times S^2$ be given. We define an $N'\times N'$ matrix-valued
function $\tens{A}[\zeta]$ 
by
\begin{align}
\label{Weiss}
\tens{A}[\zeta](\sigma, \omega)=&
\left(A_{jk}[\zeta](\sigma, \omega)\right)_{1\le j, k\le N'},\\
{A}_{jk}[\zeta](\sigma, \omega)= &
-\frac{1}{2}\sum_{a=0}^3 \omega_a
\sum_{l=1}^{N''} 
\biggl(c_{jk}^{a, l} (\omega) \zeta_l(\sigma, \omega)
{}+\sum_{b=0}^3
d_{jk}^{a, lb}(\omega) \omega_b (\pa_\sigma \zeta_l)(\sigma, \omega)
\biggr)
\label{AssB}
\end{align}
for $(\sigma, \omega)\in \R\times S^2$,
where the functions $c_{jk}^{a, l}$ and $d_{jk}^{a, lb}$
are from \eqref{AssA} in Condition~\ref{OurCond}.
For a matrix $\tens{B}$ we define $e^{\tens{B}}(=\exp \tens{B})$ in the standard way of 
$$
 e^{\tens{B}}=\tens{I}+\sum_{k=1}^\infty \frac{1}{k!} \tens{B}^k,
$$ 
where $\tens{I}$ is the identity matrix.

Now we are in a position to state our main result on the asymptotic pointwise behavior:
%%%%%%%%%%%%%%%%%%%%%%%%%%%%%%%%%%%%%%%%%%%%%%%%
\begin{theorem}\label{PointwiseAsymptotics}
Assume that 
Condition~$\ref{OurCond}$
is fulfilled,
and let $\lambda$, $\rho$ and $\ve_0$ be from Theorem $\ref{GE}$.
Let $D(\subset {\mathcal X}_N)$ be from Condition~$\ref{OurCond}$.

Suppose that we have $0<\ve\le \ve_0$.
Let 
$u=(v^{\rm T},w^{\rm T})^{\rm T}$ be the global solution to \eqref{OurSys}-\eqref{Data0}
with $(f,g)\in D$.
Then there exist
$V=V(\sigma, \omega)=\left(V_j(\sigma, \omega)\right)_{1\le j\le N'}^{\rm T}$,
$W=W(\sigma, \omega)=\left(W_k(\sigma, \omega)\right)_{1\le k\le N''}^{\rm T}$, and a positive constant $C$ such that
we have
\begin{align}
& \sum_{a=0}^3 \left|r(\pa_a v)(t,r\omega)-
 \ve \omega_a e^{(\ve \log t) \tens{A}[W](r-t, \omega)}
 (\pa_\sigma V)(r-t, \omega)\right|\nonumber\\
& \qquad\qquad\qquad\qquad\qquad\qquad\qquad\qquad\qquad\qquad\quad
\le C\ve \jb{t+r}^{3\lambda+C\ve-1}, 
\label{Ruh} \\
\label{Piyoko}
& \left|r w(t, r\omega)-\ve W(r-t, \omega)\right|\le C\ve \jb{t+r}^{2\lambda-1}\jb{t-r}^{1-\rho},\\
\label{Chada}
& \sum_{a=0}^3\left| r(\pa_a w)(t,r\omega)
{}-\ve \omega_a (\pa_\sigma W)(r-t, \omega)\right| \le C\ve \jb{t+r}^{2\lambda-1}\jb{t-r}^{-\rho}
\end{align}
%%%%%%%%%%%%%%%%%%%%%
for any $(t,r)\in \R_+ \times \R_+$ with $r \ge t/2\ge 1$,
and any  $\omega=(\omega_1,\omega_2,\omega_3) \in S^2$,
where $\tens{A}[W](\sigma, \omega)$ is
given by \eqref{Weiss}.
Here $V$ and $W$ may depend on $\ve$, but the constant $C$ is independent of $\ve$.

Moreover, there exists a positive constant $C$, being independent of $\ve$, 
such that we have
\begin{align}
\label{As01}
& \sum_{j=1}^{N'}|\pa_\sigma V_j(\sigma, \omega)-\pa_\sigma {\mathcal F}_0[f_j, g_j](\sigma, \omega)|\le 
C\ve (1+|\sigma|)^{3\lambda+C\ve-1}, \\ 
\label{As02}
& \sum_{k=1}^{N''}|\pa_\sigma^l W_k(\sigma, \omega)-\pa_\sigma^l {\mathcal F}_0[f_{k+N'}, g_{k+N'}](\sigma, \omega)|
    \le C \ve (1+|\sigma|)^{-\rho-l}
\end{align}
for any $(\sigma, \omega)\in \R\times S^2$ and $l=0,1$, where ${\mathcal F}_0$ is defined by \eqref{FriedRad}.
Here ${\mathcal F}_0[f_j, g_j]$ means 
${\mathcal F}_0[f_j(\cdot;\ve), g_j(\cdot;\ve)]$ for $1\le j\le N$.
\end{theorem}
%%%%%%%%%%%%%%%%%%%%%%%%%%%%
In what follows, we refer to $V$ in the above as the {\it modified asymptotic profile}
for $v$, and $W$ as the {\it standard asymptotic profile} for $w$.
%%%%%%%%%%%%%%%%%%%%%%%%%%%%%%%%%%%%%%%%%%%%%%%
%%%%%%%%%%%%%%%%%%%%%%%%%%%%%%%%%%%%%%%%%%%%%%%
\begin{remark}\label{ZeroAsymp}
\normalfont
(i) Let $(\varphi, \psi) \in C^\infty_0\times C^\infty_0$,
and suppose that $(\varphi, \psi)\not \equiv (0,0)$.
Then we have 
$\pa_\sigma {\mathcal F}_0\left[\varphi,\psi\right]\not \equiv 0$
(see Section \ref{translation} below).
Hence, for each $j\in \{1,\ldots, N'\}$ (resp.~$k\in\{1,\ldots, N''\}$),
it follows from \eqref{As01} (resp.~\eqref{As02})
that 
$$
\text{$\pa_\sigma V_j\not \equiv 0$ (resp.~$\pa_\sigma W_k\not \equiv 0$) for $0<\ve\ll 1$, }
$$
unless 
$$
\text{$\bigl(f_j(\cdot;0),g_j(\cdot;0)\bigr) \equiv (0,0)$
(resp.~$\bigl(f_{k+N'}(\cdot;0),g_{k+N'}(\cdot;0)\bigr) \equiv (0,0)$)}
$$
in Theorem~\ref{PointwiseAsymptotics}.

(ii) If $\varphi, \psi\in C^\infty_0(\R^3)$, then 
${\mathcal F}_0[\varphi,\psi](\sigma, \omega)=0$ for large $|\sigma|$ (see \eqref{RadSupport} below), and ${\mathcal F}_0[\varphi, \psi]$ is bounded in $\R\times S^2$. 
Therefore \eqref{As01} and \eqref{As02} yield
\begin{align}
 |\pa_\sigma V(\sigma, \omega)| \le & C (1+|\sigma|)^{3\lambda+C\ve-1},
 \label{As01a}\\
 |\pa_\sigma^l W(\sigma, \omega)| \le & C (1+|\sigma|)^{-\rho-l}, \quad l=0,1
 \label{As02a}
\end{align}
for $(\sigma,\omega)\in \R\times S^2$ and $0<\ve\le \ve_0$,
respectively.
\end{remark}
%%%%%%%%%%%%%%%%%%%%%%%%%%%%%%%%%%%%%%%%%%%%%%%
%%%%%%%%%%%%%%%%%%%%%%

%%%%%%%%%%%%%%%%%%%%%%
The asymptotic behavior in the energy sense also follows from
Theorem~\ref{PointwiseAsymptotics}.
%%%%%%%%%%%%%%%%%%%%%%
\begin{corollary}\label{EnergyAsymptotics}
Suppose that all the assumptions in Theorem~$\ref{PointwiseAsymptotics}$ are fulfilled,
and let $u=(v^{\rm T},w^{\rm T})^{\rm T}$, $W$, and $D$ be as 
in Theorem~$\ref{PointwiseAsymptotics}$. 
Then, for $(f,g)\in D$ and sufficiently small $\ve>0$, there exist some functions 
$(f^+_{j}, g^+_{j})\in \dot{H}^1(\R^3)\times L^2(\R^3)$ $(j=1,\ldots, N)$
such that we have
\begin{align}
\label{PiyokoE}
& \lim_{t\to \infty}
\left(\frac{1}{2}\sum_{a=0}^3 \|e^{-\tens{\Theta}_\ve^+(t,\cdot)}\pa_a v(t,\cdot)-\pa_a v^+(t, \cdot)\|_{L^2(\R^3)}^2\right)^{1/2}=0,\\
\label{ChadaE}
& \lim_{t\to \infty} \|w(t,\cdot)-w^+(t,\cdot)\|_E=0,
\end{align}
where 
\begin{align*}
((v^+)^{\rm T}, (w^+)^{\rm T})=&\left(v^+_{1}, \ldots, v^+_{N'}, w^+_{1},\ldots, w^+_{N''}\right)
=\left(\ve\, {\mathcal U}_0[f^+_{1}, g^+_{1}], \ldots,
        \ve\, {\mathcal U}_0[f^+_{N}, g^+_{N}]\right),\\
\tens{\Theta}_\ve^+(t,x)=&
\begin{cases}
(\ve \log t) \tens{A}[W](|x|-t, x/|x|), & (t,x)\in [2, \infty)\times (\R^3\setminus\{0\}), \\
\tens{O}, & \text{otherwise}.
\end{cases}
\end{align*}
Here $\|\cdot\|_E$ is given by \eqref{DefEnergyNorm},
and $\tens{O}$ is the zero matrix.
\end{corollary}
%%%%%%%%
Theorem~\ref{PointwiseAsymptotics} and Corollary~\ref{EnergyAsymptotics}
will be proved in Sections~\ref{ProofPointwise} and \ref{EneBehavior}, 
respectively.

%%%%%%%%%%%%%%%%%%%%%%
Comparing \eqref{Piyoko} and \eqref{Chada} with \eqref{FriedAsymp}, we see that
$w$ behaves similarly to the free solutions in the pointwise sense
(and \eqref{ChadaE} says that $w$ is asymptotically free in the energy sense), 
but $v$ may behave quite differently from the free solutions
because of the exponential factor in \eqref{Ruh} (and also in \eqref{PiyokoE}).
Here we give some applications of Theorem \ref{PointwiseAsymptotics}
to answer Question~$\ref{FirstQ}$ in the introduction.
To simplify the exposition, we introduce the following notation:
For functions $\varphi=\varphi(t, r, \omega)$ and $\psi=\psi(t,r,\omega)$
of $(t,r,\omega)\in \R_+ \times \R_+ \times S^2$,   
we write $\varphi(t,r,\omega)\approx \psi(t, r, \omega)$,
if we have
$$
 \lim_{t\to \infty} |(\varphi-\psi)(t, t+\sigma, \omega)|=0, \quad (\sigma, \omega)\in \R\times S^2.
$$
%%%%%%%%%%%%%%%%%%%%%%%%%
\begin{example}\label{FirstE}
\normalfont
Let $(u_1, u_2, w)^{\rm T}$ be the global solution to \eqref{simplestExR}
with initial data
\begin{equation}
\label{Data1}
\bigl(u_1, u_2, w\bigr)^{\rm T}=\ve (f_1, f_2, f_3)^{\rm T},
\ \bigl(\pa_t u_1, \pa_t u_2, \pa_t w\bigr)^{\rm T}=\ve (g_1, g_2, g_3)^{\rm T}
\text{ at $t=0$},
\end{equation}
and set $v=(v_1, v_2)^{\rm T}=(u_1, u_2)^{\rm T}$, where $\ve$ is assumed to be sufficiently small.
We suppose $(f, g)\in {\mathcal X}_3$ with $f=(f_j)_{1\le j\le 3}^{\rm T}$ and $g=(g_j)_{1\le j\le 3}^{\rm T}$.
Note that if we want to treat \eqref{simplestEx}, then we only have to impose restrictions
\begin{equation}
\label{Data2}
f_3=\pa_1 f_2-\pa_2 f_1,\ g_3=\pa_1 g_2-\pa_2 g_1.
\end{equation}
Concerning \eqref{simplestExR}, we get
$$
\tens{A}[\zeta](\sigma, \omega)=
-\frac{\zeta(\sigma, \omega)}{2}\left(
\begin{matrix}
\omega_1  & 0 \\
\omega_2  & 0
\end{matrix}
\right),
\text{ and } e^{\tau \tens{A}[\zeta](\sigma, \omega)}=
\left(
\begin{matrix}
e^{-\tau\omega_1\zeta(\sigma, \omega)/2} & 0 \\
\displaystyle\frac{\omega_2}{\omega_1}(e^{-\tau\omega_1\zeta(\sigma, \omega)/2}-1) & 1
\end{matrix}
\right)
$$
for $(\sigma, \omega)\in \R\times S^2$ and $\tau\in \R$,
where $\left. \left((\omega_2/\omega_1)(e^{-\tau\omega_1\zeta(\sigma, \omega)/2}-1)\right)\right|_{\omega_1=0}$
is regarded as $\left. -2^{-1}\tau\omega_2\zeta(\sigma, \omega)\right|_{\omega_1=0}$.
By Theorem \ref{PointwiseAsymptotics}, 
there exist an $\R^2$-valued function $V=V(\sigma, \omega)=\bigl(V_1(\sigma,\omega), V_2(\sigma, \omega)\bigr)^{\rm T}$
and a real-valued function $W=W(\sigma, \omega)$ such that
\begin{align}
\label{Example0101}
r(\pa_a v_1)(t,r\omega) \approx{}  & \ve\omega_a 
t^{-\ve \omega_1W(r-t, \omega)/2} 
(\pa_\sigma V_1)(r-t,\omega),\\
r(\pa_a v_2)(t, r\omega) \approx{}  & \ve \omega_a \frac{\omega_2}{\omega_1} 
\left(t^{-\ve \omega_1W(r-t, \omega)/2}-1\right)
 (\pa_\sigma V_1)(r-t,\omega)
\nonumber\\
&+\ve\omega_a (\pa_\sigma V_2)(r-t, \omega), 
\label{Example0102}\\
\label{Example0103}
r w(t, r\omega)\approx{}  & \ve W(r-t, \omega),\quad r(\pa_a w)(t, r\omega)\approx \ve \omega_a(\pa_\sigma W)(r-t,\omega)
\end{align}
for $0\le a\le 3$.
From this asymptotic pointwise 
behavior, we find that $\pa v_1$ decays slower than the free solutions 
along the line $\left\{\bigl(t, (t+\sigma)\omega\bigr); t>0\right\}$
for fixed $(\sigma, \omega) \in \R\times S^2$ if $\omega_1W(\sigma,\omega)<0$,
and faster if $\omega_1W(\sigma,\omega)>0$.
We can also recover the previous result in \cite{Kat-Kub08:02}
from \eqref{Example0101}--\eqref{Example0103};
namely we can show that if we choose appropriate initial profile
(satisfying \eqref{Data2}), then we have
\begin{equation}
\label{EneGrow01}
C_1\ve (1+t)^{C_1\ve} \le \|u(t)\|_E \le C_2\ve (1+t)^{C_2\ve}
\end{equation}
with some positive constants $C_1$ and $C_2$.
\eqref{EneGrow01}, as well as \eqref{EneGrow02} and \eqref{EneGrow03} below, will be
proved in Proposition~\ref{EneGrow11}.
\end{example}
%%%%%%%%%%%%%%%%%%%%%%%%%%%%%%%%%
\begin{example}\label{SecondE}
\normalfont
Let $u=(u_1, u_2, u_3)^{\rm T}=(v_1, v_2, w)^{\rm T}$, and consider
\begin{equation}
\label{LogEx}
\begin{cases}
\dal v_1 = -2(\pa_t w)(\pa_t v_2)+F_1^{0}(\pa u),\\
\dal v_2 = F_2^{0}(\pa u),\\
\dal w = F_3^{0}(\pa u)
\end{cases}
\end{equation}
in $(0,\infty)\times \R^3$ with the initial condition \eqref{Data0} for small $\ve$, where 
$F^{0}=(F_j^{0})^{\rm T}_{1\le j\le 3}$ 
satisfies the null condition
(note that a similar and simpler example is mentioned 
in \cite{Lin-Rod03} and \cite{Lin-Rod05} as an example
for the weak null condition).
Then we have
$$
\tens{A}[\zeta](\sigma, \omega)=
\left(
\begin{matrix}
0 & (\pa_\sigma \zeta)(\sigma, \omega) \\
0 & 0
\end{matrix}
\right), \text{ and }
e^{\tau \tens{A}[\zeta](\sigma, \omega)}=
\left(
\begin{matrix}
1 & \tau (\pa_\sigma\zeta)(\sigma, \omega)\\
0 & 1
\end{matrix}
\right)
$$
for $(\sigma, \omega)\in \R\times S^2$ and $\tau\in \R$.
By Theorem \ref{PointwiseAsymptotics},
for any $(f,g)\in {\mathcal X}_3$ 
there exist $V=V(\sigma, \omega)$
and $W=W(\sigma, \omega)$ such that
we have \eqref{Example0103} and
\begin{align}
r(\pa_a v_1)(t,r\omega) \approx{}  & \ve \omega_a (\pa_\sigma V_1)(r-t,\omega)
\nonumber\\
& {}+\ve^2 \omega_a\left(\log t\right) (\pa_\sigma W)(r-t, \omega)(\pa_\sigma V_2)(r-t,\omega),
\label{Example0201}\\
\label{Example0202}
r(\pa_a v_2)(t, r\omega) \approx{}  & \ve \omega_a (\pa_\sigma V_2)(r-t, \omega)
\end{align}
for $0\le a\le 3$.
This asymptotic pointwise behavior results in the slower growth of the energy
than \eqref{EneGrow01}; 
if we choose appropriate initial profile, then we have
\begin{equation} 
\label{EneGrow02}
C_1\left(\ve+\ve^2\log (1+t) \right)\le \|u(t)\|_E \le C_2\left(\ve+\ve^2\log (1+t) \right)
\end{equation}
with some positive constants $C_1$ and $C_2$.
We remark that this kind of logarithmic growth of the energy was observed by
Sunagawa~\cite{Sun04} for a system of semilinear Klein-Gordon equations with different masses
(see also \cite{Sun05} for the pointwise behavior).
\end{example}
%%%%%%%%%%%%%%%%%%%%%%
\begin{example}\label{ThirdE}
\normalfont
Let $u=(u_1, u_2, u_3)^{\rm T}=(v_1, v_2, w)^{\rm T}$, and we consider
\begin{equation}
\label{RotEx}
\begin{cases}
\dal v_1 = -2(\pa_t w)(\pa_t v_2){}+F_1^{0}(\pa u), \\
\dal v_2 = 2(\pa_t w)(\pa_t v_1){}+F_2^{0}(\pa u), \\
\dal w = F_3^{0}(\pa u)
\end{cases}
\end{equation}
in $(0,\infty)\times \R^3$ with the initial condition \eqref{Data0}, where 
$F^{0}=\bigl(F_j^{0}\bigr)_{1\le j\le 3}^{\rm T}$
satisfies the null condition as before, and $\ve$ is assumed to be sufficiently small.
For this example, we have
$$
\tens{A}[\zeta](\sigma, \omega)=
(\pa_\sigma \zeta)(\sigma, \omega)
\left(
\begin{matrix}
0   & 1 \\
-1 & 0
\end{matrix}
\right).
$$
Hence, by Theorem~\ref{PointwiseAsymptotics}, for any $(f,g)\in {\mathcal X}_3$ 
there exist $V=V(\sigma, \omega)$
and $W=W(\sigma, \omega)$ such that we have
\eqref{Example0103} and 
$$
\left(
\begin{matrix}
r(\pa_a v_1)(t,r\omega)\\
r(\pa_a v_2)(t, r\omega)
\end{matrix}
\right)
\approx \ve\omega_a e^{(\ve\log t)\tens{A}[W](r-t,\omega)}
\left(
\begin{matrix}
(\pa_\sigma V_1)(r-t, \omega)\\
(\pa_\sigma V_2)(r-t, \omega)
\end{matrix}
\right)
$$ 
for $0\le a\le 3$,
where
\begin{equation}
e^{\tau \tens{A}[W](\sigma, \omega)}=
\left(
\begin{matrix}
\cos \left(\tau(\pa_\sigma W)(\sigma, \omega)\right)  & \sin \left(\tau(\pa_\sigma W)(\sigma, \omega)\right)\\
-\sin  \left(\tau(\pa_\sigma W)(\sigma, \omega)\right) &  \cos\left(\tau(\pa_\sigma W)(\sigma, \omega)\right)
\end{matrix}
\right).
\label{Example0302}
\end{equation}
There is no growth or decay factor in $e^{(\ve\log t) \tens{A}[W]}$,
and we can show that
\begin{equation}
\label{EneGrow03}
C_1\ve\le \|u(t)\|_E\le C_2\ve
\end{equation}
for small $\ve>0$
with some positive constants $C_1$ and $C_2$, unless $\bigl(f(\cdot;0), g(\cdot;0)\bigr)\equiv (0,0)$.
Thus the energy of this system stays bounded from above and below
by positive constants for any initial profile not vanishing at $\ve=0$
(actually the energy is preserved for any initial profile 
if $F^0\equiv 0$).
However the solution behaves quite differently from the free solution
in the pointwise sense, as far as
$(\pa_\sigma W)(\sigma, \omega)\not \equiv 0$.
Furthermore we can also show that there exists some initial profile $(f,g)\in {\mathcal X}_3$
such that the global solution to \eqref{RotEx} is not asymptotically free in the energy sense for $0<\ve\ll 1$ (see Proposition~\ref{EneGrow11} below for the proof).
This example gives a positive answer to Question~$\ref{FirstQ}$. 
\end{example}
%%%%%%%%%%
\begin{remark}
\normalfont
Observe that Condition~$\ref{AlinhacCond}$
is not satisfied for \eqref{RotEx}, while it is fulfilled for
$\eqref{simplestEx}$ and $\eqref{LogEx}$. This 
is the reason why we have introduced Condition~$\ref{OurCond}$. 
\end{remark}
%%%%%%%%%%%%%
\subsection{Asymptotically free solutions and the null condition} 
%%%%%%%%%%%%%%%%%%%%%%%%%%%%%%%%%%%%%%%%%%%%%%%%%%%%%%%%%%%%%%%%%%%
Our aim here is 
to answer Question~$\ref{SecondQ}$ in the introduction.
%%%%%%%%%%%%%%%%%%%%%%%%
%%%%%%%%%%%%%%%%%%%%%%%%
\begin{definition}\label{wafp}
We say that a global solution $u=u(t,x)$ 
to \eqref{OurSys} is asymptotically free in the pointwise sense,
if there exists some function $U=U(\sigma, \omega)$ such that $r\pa_a u(t,r\omega)\approx\omega_a\pa_\sigma U(r-t,\omega)$
for $0\le a\le 3$, namely
$$
 \lim_{t\to\infty} \sum_{a=0}^3
\left|\left.\left\{r\pa_a u(t, r\omega)-\omega_a (\pa_\sigma U)(r-t, \omega)\right\}\right|_{r=t+\sigma}\right|=0,\quad 
(\sigma, \omega)\in {\R\times S^2}.
$$
\end{definition}
%%%%%%%%%%%%%%%
In view of \eqref{FriedAsymp}, by choosing the Friedlander radiation field as $U$,
we see that the solution to the free wave equation $\dal u=0$ 
with $C^\infty_0$-data is
asymptotically free in the pointwise sense.

We say that {(AFP)} (resp.~{(AFE)}) holds
if for any $(f,g) %=\left(f(x;\ve), g(x;\ve)\right)
\in {\mathcal X}_N$
there exists a positive constant $\ve_1$
such that, for any $\ve \in (0, \ve_1]$,
the global solution $u$ to the Cauchy problem \eqref{OurSys}-\eqref{Data0}
is asymptotically free in the pointwise sense (resp.~in the energy sense).

\begin{theorem}
\label{Necessity}
For the system \eqref{OurSys}
with each $F_j=F_j(u,\pa u)$ being a homogeneous polynomial of degree $2$ in its arguments, the following two are equivalent:
\begin{enumerate}
\item[$(1)$] The null condition is satisfied.
\item[$(2)$] Condition~$\ref{OurCond}$
is satisfied with $D={\mathcal X}_N$,
and $({\rm AFP})$ holds.
\end{enumerate}
\end{theorem}
We conjecture that (1) above is also equivalent to (2) with
(AFP) being replaced by (AFE), but this conjecture is still open.

Now let us consider the case of Condition~\ref{AlinhacCond}.
Though we can rewrite the system to another system satisfying 
Condition~\ref{OurCond} (see Proposition~\ref{Example}), 
we cannot apply Theorem~\ref{Necessity} directly
because of the restriction on initial data for the rewritten system.
 Thus we need to consider this case separately.
Making use of the ``rank one'' structure \eqref{AA-1}, 
we can get a better result than Theorem~\ref{Necessity}.
%%%%%%%%%%%%%%%%%%%%%%%%%%%%%%%%%%%%%%%%%%%%%%%%%%%%%%%%%%%%%%%%%%%
\begin{theorem}
\label{Necessity02}
For the system \eqref{OurSys} with each $F_j=F_j(\pa u)$
being a homogeneous polynomial of degree $2$ in $\pa u$,
the following three are equivalent to each other: 
\begin{enumerate}
\item[$(1)$] The null condition is satisfied.
\item[$(2)$] 
Condition~$\ref{AlinhacCond}$
is satisfied, and $({\rm AFP})$ holds.
\item[$(3)$]
Condition~$\ref{AlinhacCond}$
is satisfied, and $({\rm AFE})$ holds.
\end{enumerate}
\end{theorem}
%%%%%%%%%%%%%%%%%%%%%%%%%%%%%%%%%%%

Theorems \ref{Necessity} and \ref{Necessity02} will be proved in Section \ref{ProofNecessity}.

%%%%%%%%%%%%%%%%%%%%%%%%%%%%%%%%%%%%%%%%%%%%%%
%%%%%%%%%%%%%%%%%%%%%%%%%%%%%%%%%%%%%%%%%%%%%%
\section{Preliminaries}
\subsection{Vector fields associated with wave equations}
First we recall some properties of the vector fields contained in $\Gamma$
(remember the notation in Section \ref{notation}).
We have $[S, \dal]=-2\dal$, and $[\Omega_j, \dal]=[L_j, \dal]=[\pa_a, \dal]=0$
for $1\le j\le 3$ and $0\le a\le 3$, where $[A,B]=AB-BA$ for operators $A$ and $B$.
Thus we get
\begin{equation}
\label{CommGamma}
\dal (\Gamma^\alpha \varphi)= 
(\Gamma_0+2)^{\alpha_0}\Gamma_1^{\alpha_1}\cdots \Gamma_{10}^{\alpha_{10}}
(\dal \varphi)
\end{equation}
for any smooth function $\varphi$ and any multi-index $\alpha=(\alpha_0,\alpha_1,\dots,\alpha_{10})$.
Since $[\Gamma_a,\Gamma_b]$ with $0\le a, b\le 10$ can be written in 
terms of $\Gamma$, it follows that
\begin{equation}
\label{MultiGamma}
\Gamma^\alpha\Gamma^\beta=\Gamma^{\alpha+\beta}+\sum_{|\gamma|\le |\alpha|+|\beta|-1} c_{\gamma}^{\alpha, \beta} \Gamma^\gamma
\end{equation}
for any multi-indices $\alpha$ and $\beta$ with appropriate constants $c_\gamma^{\alpha, \beta}$.
Especially $[\Gamma_a, \pa_b]$ for $0\le a\le 10$ and $0\le b\le 3$ can be written 
in terms of $\pa$. Therefore 
for any nonnegative integer $s$, there exists a positive constant $C_s$
such that
\begin{equation}
\label{DerivativeG}
C_s^{-1} |\pa \varphi(t, x)|_s
\le \sum_{|\alpha| \le s}\sum_{a=0}^3 |\pa_a \Gamma^\alpha\varphi(t,x)|
\le C_s |\pa \varphi(t,x)|_s 
\end{equation}
for any sufficiently smooth function $\varphi=\varphi(t,x)$,
where $|\cdot|_s$ is given by \eqref{InvariantNorm}.

The following Sobolev type inequality, known as the Klainerman inequality,
is quite useful in deriving decay estimates for wave equations (see Klainerman~\cite{Kla87} for the proof; see also \cite{Kla85} and \cite{Hoe97}).
%%%%%%%%%%%%%%%%%%%%%%%%%%%%%%%%%%%%%%%%%%%%%
\begin{lemma}\label{KlainermanIneq}
There exists a positive constant $C$ such that we have
\begin{equation}
\label{KlainermanInequality}
\sup_{x\in \R^3}\jb{t+|x|}\jb{t-|x|}^{1/2}|\varphi(t,x)|\le C \|\varphi(t,\cdot)\|_2
\end{equation}
for any $C^2$-function $\varphi$, provided that
the right-hand side of \eqref{KlainermanInequality} is finite.
Here $\|\cdot\|_s$ is given by \eqref{InvariantNorm}.
\end{lemma}
%%%%%%%%%%%%%%%%%%%%%%%%%%%%%%%%%%%%%%%%%%%%%

We set $r=|x|$, $\omega=r^{-1}x$, and $\pa_r=\sum_{j=1}^3 (x_j/|x|) \pa_j$. We also define
$L_r=r\pa_t+t\pa_r=\sum_{j=1}^3 (x_j/|x|) L_j$, and $\pa_\pm =\pa_t\pm \pa_r$.
Since we have $S=t\pa_t+r\pa_r$, we obtain
\begin{equation}
\label{blue00}
\pa_+=\frac{1}{t+r}(S+L_r).
\end{equation}
Note that \eqref{blue00} implies $\pa_+=(1+t+r)^{-1}(S+L_r+\pa_t+\pa_r)$.
Since $\pa_t=(\pa_-+\pa_+)/2$ and $\pa_r=(-\pa_-+\pa_+)/2$, 
we obtain from \eqref{blue00} that
\begin{equation}
\label{blue01}
\Bigl|\Bigl(\pa_t-\frac{1}{2}\pa_-\Bigr)\varphi(t,x)\Bigr|
+\Bigl|\Bigl(\pa_r+\frac{1}{2}\pa_-\Bigr)\varphi(t,x)\Bigr|\le C\jb{t+r}^{-1}|\varphi(t,x)|_1
\end{equation}
for any smooth function $\varphi$.
Since 
\begin{equation}
\label{ExpressNabla}
\nabla_x=\omega \pa_r- r^{-1}\omega\times \Omega
=\omega \pa_r-t^{-1}\omega\times(\omega\times L),
\end{equation}
we get
\begin{equation}
\label{blue02}
|\left(\pa_j-\omega_j\pa_r\right) \varphi (t,x)| \le C\jb{t+r}^{-1} |\varphi(t,x)|_1, \quad 
j=1,2,3.
\end{equation}
%%%%%%%%%%%%%%%%%%%%%%%%%%%%%%%%
Now \eqref{blue01} and \eqref{blue02} lead to the following estimate \eqref{blue03} and its immediate consequence
\eqref{blue03'}:
\begin{lemma}\label{frame01}
There exists a positive constant $C$ such that we have
\begin{align}
\label{blue03}
& \sum_{a=0}^3 \Bigl|
                      \Bigl(
                        \pa_a-\frac{-\omega_a}{2}\pa_-
                      \Bigr)\varphi(t,x)
                    \Bigr|\le C\jb{t+r}^{-1} |\varphi(t,x)|_1, \\
\label{blue03'}
& \sum_{a=0}^3 \Bigl|
                      r\left(
                         \pa_a \varphi
                       \right)(t,x)-\frac{-\omega_a}{2}\pa_-
                      \bigl(r\varphi(t,x) \bigr)
                       \Bigr|\le C|\varphi(t,x)|_1 %, \quad 0\le a\le 3
\end{align}
for any smooth function $\varphi=\varphi(t,x)$, where
$\omega=(\omega_1, \omega_2, \omega_3)=|x|^{-1} x$, and $r=|x|$.
\end{lemma}

Recalling the definition of the null condition \eqref{NC00}, 
we obtain from Lemma~\ref{frame01} the following estimate for the null forms
given by \eqref{QT01} and \eqref{QT02}
(see Klainerman \cite{Kla86} for the details of the proof).
\begin{lemma}
\label{NullNull}
Let $Q$ be one of the null forms $Q_0$ and $Q_{ab}$ with $0\le a<b\le 3$.
Then, for any nonnegative integer $s$, there exists a positive constant $C_s$
such that
$$
|Q(u_k, u_l)|_s\le C_s \jb{t+|x|}^{-1}\bigl(|u|_{[s/2]+1}|\pa u|_s+|\pa u|_{[s/2]}|u|_{s+1}\bigr),
\quad 1\le k, l\le N
$$
at $(t,x)\in (0,\infty) \times \R^3$ for any smooth function $u=u(t,x)=\left(u_j(t,x)\right)_{1\le j\le N}^{\rm T}$,
where $[m]$ denotes the largest integer not exceeding the number $m$.
\end{lemma}

Since $tS-rL_r=(t^2-r^2)\pa_t$ and $tL_r-rS=(t^2-r^2)\pa_r$,
we get
\begin{equation}
\label{blue04}
\jb{t-r}\left(|\pa_t \varphi(t,x)|+|\pa_r \varphi(t,x)|\right)\le C |\varphi(t,x)|_1.
\end{equation}
Equations \eqref{blue02} and \eqref{blue04} yield the following (see also Lindblad~\cite{Lin90}).
\begin{lemma}\label{Frame02}
There exists a positive constant $C$ such that
$$
\jb{t-r}|\pa \varphi(t,x)|\le C |\varphi(t,x)|_1
$$
holds for any smooth function $\varphi=\varphi(t,x)$.
\end{lemma}
%%%%%%%%%%%%%%%%%%%%%%%%%%%%%%

For $R>0$, we define
\begin{equation}
\label{DefBall}
B_R:=\{x\in \R^3; |x|\le R\}.
\end{equation}
The following Hardy type inequality
is due to Lindblad \cite{Lin90}.
\begin{lemma}\label{LindbladIneq}
Let $R>0$ be given. 
Then there exists a positive constant $C_R$,
which depends on $R$, such that
$$
\left(\int_{\R^3} \frac{|\varphi(t,x)|^2}{\jb{t-|x|}^2} dx \right)^{1/2}\le C_R \|\pa \varphi(t,\cdot)\|_{L^2(\R^3)}, \quad t\ge 0
$$
holds for any smooth function $\varphi=\varphi(t,x)$ satisfying 
$$
 \supp \varphi(t,\cdot) \subset B_{t+R}, \quad t\ge 0.
$$
\end{lemma}

%%%%%%%%%%%%%%%%%%%%%%%%%%%%%%
Now we introduce
\begin{equation}
\label{AlinhacField}
Z=(Z_1, Z_2, Z_3):=\omega \pa_t+\nabla_x=(\omega_j \pa_t+\pa_j)_{1\le j\le 3},
\end{equation}
where $\omega=(\omega_1,\omega_2,\omega_3)=|x|^{-1}x\in S^2$.
For a nonnegative integer $s$ and a smooth function $\varphi=\varphi(t,x)$, we define
\begin{equation}
|\varphi(t,x)|_{Z, s}=\sum_{j=1}^3 \sum_{|\alpha|\le s} \left|
Z_j\Gamma^\alpha \varphi(t,x)
\right|.
\label{InvariantAlinhac}
\end{equation}
Then we obtain the following inequality.
\begin{lemma}
\label{frame03}
For a nonnegative integer $s$, there exists a positive constant $C_s$ such that
$$
|\Gamma \varphi(t,x)|_s\le C_s\left(r|\varphi(t,x)|_{Z,s}+\jb{t-r}|\pa \varphi(t,x)|_s\right)
$$
for any smooth function $\varphi=\varphi(t,x)$, where $\Gamma\varphi=(\Gamma_a\varphi)_{0\le a\le 10}$.
\end{lemma}
\begin{proof}
It is easy to check that
$$
S=x\cdot Z+(t-r)\pa_t,\ L=rZ+(t-r)\nabla_x,\ \Omega=x\times Z.
$$
Hence we get
$$
|\Gamma \varphi|_s\le C\sum_{a=0}^{10}\sum_{|\beta|\le s}|\Gamma_a\Gamma^\beta\varphi|
\le Cr|\varphi|_{Z,s}+C\jb{t-r}\sum_{a=0}^3\sum_{|\beta|\le s} |\pa_a \Gamma^\beta\varphi|, 
$$
which implies the desired result because of \eqref{DerivativeG}.
\end{proof}
%%%%%%%%%%%%%
By \eqref{ExpressNabla} and \eqref{blue00}, we have
$$
Z=\omega \pa_+-r^{-1}\omega\times \Omega=\omega(t+r)^{-1}(S+\omega\cdot L)-r^{-1}\omega\times \Omega.
$$
By direct calculations, we have
$$
[\pa_t, Z_j]=0,\ [\pa_k, Z_j]=r^{-1}(\delta_{jk}-\omega_j\omega_k)\pa_t
$$
for $1\le j, k\le 3$, where $\delta_{jk}$ is the Kronecker delta.
For a nonnegative integer $s$, there is a positive constant $C_s$ such that
\begin{align*}
& \sum_{j=1}^3\left(|(t+r)^{-1}\omega_j|_s+|r^{-1}\omega_j|_s\right)\\
& \qquad {}+\sum_{j,k=1}^3
\left(|(t+r)^{-1}\omega_j\omega_k|_s+|r^{-1}(\delta_{jk}-\omega_j\omega_k)|_s\right)\le C_s\jb{t+r}^{-1}
\end{align*}
when $r\ge t/2\ge 1$.
Thus we obtain the following estimate.
\begin{lemma}\label{CommZG}
For any positive integer $s$, there exists a positive constant $C_s$
such that
\begin{align}
\sum_{j=1}^3 
|Z_j \varphi(t,x)|_s
 \le & C_s\jb{t+r}^{-1}|\varphi(t,x)|_{s+1}, 
\label{CommRZG03}\\
\sum_{a=0}^3\sum_{j=1}^3 
|\pa_a Z_j \varphi(t,x)|_s
 \le & C_s\jb{t+r}^{-1}|\pa \varphi(t,x)|_{s+1},
 \label{CommZG03}
\end{align}
for any smooth function $\varphi$ and any $(t,x)$ with $r\ge t/2\ge 1$.
\end{lemma}

%%%%%%%%%%%%%%%%%%%%%%%%%%%%%%%%%
\subsection{Basic estimates for wave equations}
Here we give two basic estimates for the solution to
the inhomogeneous wave equation
\begin{equation}
\label{LinWave}
\dal u(t,x)=F(t,x), \quad (t,x)\in (0,\infty)\times \R^3
\end{equation}
with initial data
\begin{equation}
\label{LinData}
u(0,x)=\varphi(x),\ (\pa_t u)(0,x)=\psi(x),
\quad x\in \R^3.
\end{equation}
We suppose that $\varphi$, $\psi$ and $F$ are %sufficiently
smooth functions.

The following improved energy inequality is due to Alinhac \cite{Ali04} 
and Lindblad-Rodnianski \cite{Lin-Rod05} (see also Alinhac \cite{Ali03} and \cite{Ali06}). 
%%%%%%%%%%%%%%%%%%%%%%%%%%%%%%%%%%%%
\begin{lemma}\label{AlinhacGhost}
Let $u$ be the solution to \eqref{LinWave}-\eqref{LinData}.
Then, for $\mu\ge 0$ and $\rho>0$, there exists a positive constant $C=C(\rho)$ such that
\begin{align*}
& \jb{t}^{-\mu}\|\pa u(t,\cdot)\|_{L^2}+\left(  
 \int_0^t \int_{\R^3}
 \frac{|Z u(\tau, x)|^2}{\jb{\tau}^{2\mu}\jb{\tau-|x|}^{1+\rho}} dx d\tau \right)^{1/2}\\
& \qquad\qquad \le C \left(\left\|\nabla_x\varphi\right\|_{L^2}+\left\|\psi
\right\|_{L^2}
{}+\int_0^t \jb{\tau}^{-\mu}\|F(\tau, \cdot)\|_{L^2}d\tau\right)
\end{align*}
for $t\ge 0$, where $Z$ is given by \eqref{AlinhacField}.
\end{lemma}
\begin{proof}[Outline of the proof] We multiply $\dal u$ by
$\jb{t}^{-2\mu}e^{\zeta_\rho(|x|-t)} (\pa_t u)$,
integrate it over $\R^3$, and perform the integration by parts as in the standard energy estimate, where
$$
\zeta_\rho(s)=\int_{-\infty}^s \jb{\tau}^{-(1+\rho)} d\tau.
$$
Then we obtain the desired result, since there exists a positive constant $C_\rho$
such that $1\le e^{\zeta_\rho(s)}\le C_\rho$ for $s\in \R$. 
\end{proof}

Now we turn our attention to the decay estimate.
We define
$$
{\mathcal W}_\rho(t, r)=
 \begin{cases}
  \jb{t+r}^{\rho} & \text{if $\rho<0$},\\
  \bigl\{\log \bigl(2+\jb{t+r}\jb{t-r}^{-1}\bigr)\bigr\}^{-1} 
& \text{if $\rho=0$},\\
  \jb{t-r}^{\rho} & \text{if $\rho>0$}.                  
 \end{cases}
$$
\begin{lemma}\label{Asakura}
Let $u$ be the solution to \eqref{LinWave}-\eqref{LinData}.
Suppose that $\mu>0$ and $\kappa>1$.
Then there exists a positive constant $C=C(\mu, \kappa)$ such that
\begin{align}
&
\jb{t+|x|} 
{\mathcal W}_{\mu-1}(t,|x|) |u(t,x)|
\nonumber\\
& \quad \le
C \sup_{|y-x|\le t} \jb{y}^\mu \biggl(\jb{y} \sum_{|\alpha|\le 1} 
 \left|\left(\pa_x^\alpha \varphi \right)(y) \right|
{}+|y| \left|\psi(y)\right|\biggr)
\nonumber\\
& \qquad\, {}+C\sup_{\tau\in [0,t]}\sup_{|y-x|\le t-\tau} |y|\jb{\tau+|y|}^{\mu}
\jb{\tau-|y|}^{\kappa}
|F(\tau,y)|
\label{Otto}
\end{align}
for $(t,x)\in \R_+\times \R^3$, provided that the right-hand side of \eqref{Otto}
is finite. Here $\pa_x=(\pa_1, \pa_2, \pa_3)$, and we have used the standard notation of
multi-indices.
\end{lemma}
\begin{proof}[Outline of the proof]
The case $F\equiv 0$ is essentially proved in Asakura~\cite[Proposition~1.1]{Asa86} (see also Katayama-Yokoyama~\cite[Lemma 3.1]{Kat-Yok06}). Thus we assume $\varphi=\psi\equiv 0$.
Then the case $\mu\ge 1$ is treated in Kubota-Yokoyama~\cite[Lemma 3.2]{Kub-Yok01},
and the other case can be treated similarly. Here we give the outline of the proof.

The solution $u$ with $\varphi=\psi \equiv 0$ can be written as
$$
u(t,x)=\frac{1}{4\pi r}\int_0^t \left(\int_{|r-(t-\tau)|}^{r+t-\tau}
\left(\int_0^{2\pi}F\left(\tau, \lambda \Theta(\lambda, \theta; t-\tau, x)\right)
d\theta\right) \lambda d\lambda \right) d\tau,
$$
where we put $r=|x|$, and $\Theta$ is a certain $S^2$-valued function
with $$
|\lambda \Theta(\lambda,\theta; s, x)-x|=s
$$
(see John \cite{Joh83}).
Therefore \eqref{Otto} with $\varphi=\psi \equiv 0$ is obtained if we can show
\begin{align}
J_{\mu, \kappa}(t,r):= & \frac{1}{r}\int_0^t \left(\int_{|r-(t-\tau)|}^{r+t-\tau}
(1+\tau+\lambda)^{-\mu}(1+|\tau-\lambda|)^{-\kappa} d\lambda\right)d\tau 
\nonumber\\
\le & C\jb{t+r}^{-1}{\mathcal W}_{\mu-1}(t,r)^{-1} 
\label{KubYok01}
\end{align}
with some positive constant $C$.
We put $p=\tau+\lambda$ and $q=\lambda-\tau$. Then we have
\begin{align}
J_{\mu, \kappa}(t,r)=&\frac{1}{2r}\int_{|t-r|}^{t+r}(1+p)^{-\mu}
\left(\int_{r-t}^p (1+|q|)^{-\kappa} dq\right) dp
\nonumber\\
\le & \frac{1}{(\kappa-1)r}\int_{|t-r|}^{t+r} (1+p)^{-\mu} dp, 
\label{KubYok02}
\end{align}
where we have used the assumption $\kappa>1$.
Now, if $r\ge (1+t)/2$,
then a direct calculation of the last integral in \eqref{KubYok02}
leads to \eqref{KubYok01},  since we have $r^{-1} \le 3(1+t+r)^{-1}$ for this case. 
If $r<(1+t)/2$, 
then we have $(1+|t-r|)^{-1}\le (1+t-r)^{-1} \le 3(1+t+r)^{-1}$, and we get
$$
\frac{1}{r}\int_{|t-r|}^{t+r} (1+p)^{-\mu} dp \le (1+|t-r|)^{-\mu}\frac{1}{r}\int_{|t-r|}^{t+r}dp\le C(1+t+r)^{-\mu},
$$
which implies the desired result.
\end{proof}
%%%%%%%%%%%%%%%%%%%%%%%%%%%%%%%%%%%%%%%%%%%%%

%%%%%%%%%%%%%%%%%%%%%
%%%%%%%%%%%%%%%%%%%%%
\subsection{The Friedlander radiation field and the translation representation}
\label{translation}
It is known that 
if $(\varphi,\psi)\in C^\infty_0(\R^3)\times C^\infty_0(\R^3)$, then
the radiation field ${\mathcal F}_0[\varphi, \psi]$, given by \eqref{FriedRad},
belongs to $C^\infty_0(\R\times S^2)$. In fact, if $\supp\varphi\cup\supp\psi\subset B_R$
for some $R>0$, then \eqref{FriedRad} and \eqref{RadTrans}
imply
\begin{equation}
{\mathcal F}_0[\varphi, \psi](\sigma, \omega)=0,\quad |\sigma|\ge R,\ \omega\in S^2.
\label{RadSupport}
\end{equation}

We define the mapping ${\mathcal T}:C^\infty_0(\R^3)\times C^\infty_0(\R^3)
\to C^\infty_0(\R\times S^2)$ by
$$
{\mathcal T}[\varphi, \psi] (\sigma, \omega)
:=\pa_\sigma {\mathcal F}_0[\varphi, \psi] (\sigma, \omega),\quad
(\sigma, \omega)\in \R\times S^2
$$
for $(\varphi, \psi)\in C^\infty_0(\R^3)\times C^\infty_0(\R^3)$.
${\mathcal T}[\varphi, \psi]$ is called the {\it translation representation}
of $(\varphi, \psi)$, because we have
$$
{\mathcal T}\bigl[{\mathcal U}_0[\varphi,\psi](t,\cdot), \pa_t {\mathcal U}_0[\varphi, \psi](t,\cdot)\bigr](\sigma, \omega)
={\mathcal T}[\varphi, \psi] (\sigma-t, \omega),
$$
where ${\mathcal U}_0[\varphi, \psi]$ is 
defined in Section~\ref{notation} 
(namely it is the solution to \eqref{LinearWave}-\eqref{LinearData}).

Let $H_0$ be the completion of $C^\infty_0(\R^3)\times C^\infty_0(\R^3)$ with respect to the norm
$$
\|(\varphi, \psi)\|_{H_0}:=\left(\frac{1}{2}
\int_{\R^3} \left(|\nabla_x\varphi(x)|^2
{}+|\psi(x)|^2\right)dx \right)^{1/2}.
$$ 
Then it is known that
${\mathcal T}$ is uniquely extended to an isometric isomorphism
from $H_0$ onto $L^2(\R\times S^2)$. We refer the readers to Lax-Phillips 
\cite[Chapter IV]{LaxPhi89} for the facts mentioned above. 
Note that $(\varphi,\psi)\in H_0$ if and only if
$\varphi\in \dot{H}^1(\R^3)$ and $\psi\in L^2(\R^3)$.
%%%%%%%%%%%%%%%%%%%%%%%%%%%%%%%%%%%%%
\begin{lemma}\label{Trans01}
Let $(\varphi, \psi)\in H_0$.
Then we have
$$
\lim_{t\to \infty} \left(\frac{1}{2}
\sum_{a=0}^3
\|\pa_a u_0(t,\cdot)-T_*^a(t,\cdot)\|_{L^2(\R^3)}^2
\right)^{1/2}=0,
$$
where 
\begin{align*}
u_0(t,x)=& \, {\mathcal U}_0[\varphi, \psi](t,x),\\
T_*^a(t, x)=&\left. \left(
\omega_a r^{-1}{\mathcal T}[\varphi, \psi](r-t,\omega)
\right)\right|_{r=|x|, \omega=(\omega_1, \omega_2, \omega_3)=x/|x|},
\quad 0\le a\le 3.
\end{align*}
\end{lemma}
\begin{proof}
Let $\ve>0$. Then there exists $(\widetilde{\varphi}, \widetilde{\psi})\in
C^\infty_0(\R^3)\times C^\infty_0(\R^3)$ such that 
\begin{equation}
\label{SmoothApproximation}
\|(\varphi, \psi)-(\widetilde{\varphi}, \widetilde{\psi})\|_{H_0}<\ve.
\end{equation}
Let $\widetilde{u}_0$ and $\widetilde{T}_{*}^a$ be defined similarly to $u_0$ and $T_*^a$
by replacing $(\varphi, \psi)$  
with $(\widetilde{\varphi}, \widetilde{\psi})$ in their definitions.
Suppose that $\supp \widetilde{\varphi} \cup \supp \widetilde{\psi}\subset B_R$ with $R>0$. Then the Huygens principle implies $\widetilde{u}_0(t,x)=0$ for $|r-t|\ge R$ with $r=|x|$. We also have $\widetilde{T}_*^a(t,x)=0$ for $|r-t|\ge R$ 
(see \eqref{RadSupport}).
Hence, for $t\ge \max\{2R,2\}$, \eqref{FriedAsymp} leads to
\begin{align}
& \|\pa_a\widetilde{u}_0(t,\cdot)-\widetilde{T}_{*}^a(t,\cdot)\|_{L^2(\R^3)}^2 \nonumber\\
& \qquad 
= \int_{\omega\in S^2}\left(\int_{t-R}^{t+R} |r\pa_a\widetilde{u}_0(t,r\omega)-\omega_a\pa_\sigma {\mathcal F}_0[\widetilde{\varphi}, \widetilde{\psi}](r-t, \omega)|^2 dr \right)dS_\omega \nonumber\\
& \qquad 
\le CR(1+t)^{-2},
\label{Tabyonse01}
\end{align}
since we have $t-R\ge t/2\ge 1$.
Here and hereafter $dS_\omega$ denotes the area element on $S^2$.
From the energy identity and \eqref{SmoothApproximation} we get
\begin{equation}
\|u_0(t,\cdot)-\widetilde{u}_0(t,\cdot)\|_{E}
=\|({\varphi}, {\psi})-(\widetilde{\varphi}, \widetilde{\psi})\|_{H_0}<\ve.
\label{Tabyonse02}
\end{equation}
Since ${\mathcal T}$ is an isometry from $H_0$ to $L^2(\R\times S^2)$,
by \eqref{SmoothApproximation} we obtain
\begin{align}
\frac{1}{2}
 \sum_{a=0}^3\|T_*^a(t,\cdot)-\widetilde{T}_{*}^a(t,\cdot)\|_{L^2(\R^3)}^2
= & \int_{\omega\in S^2}\left(\int_{0}^\infty 
|{\mathcal T}[\varphi-\widetilde{\varphi},\psi-\widetilde{\psi}](r-t,\omega)
|^2 dr\right)dS_\omega 
\nonumber\\
\le & 
\bigl\| 
  {\mathcal T}\bigl[\varphi-\widetilde{\varphi}, \psi-\widetilde{\psi}\,\bigr]
\bigr\|_{L^2(\R\times S^2)}^2< \ve^2.
\label{Tabyonse03}
\end{align}
From \eqref{Tabyonse01}, \eqref{Tabyonse02}, and \eqref{Tabyonse03}
we get
$$
\limsup_{t\to \infty} \left(\frac{1}{2}\sum_{a=0}^3 \|\pa_a u_0(t,\cdot)-T_{*}^a(t,\cdot)\|_{L^2(\R^3)}^2\right)^{1/2}\le 
       2\ve.
$$
Since $\ve>0$ can be chosen arbitrarily, we obtain the desired result.
\end{proof}

%%%%%%%%%%%%%%%%%%%%%%%%%%%%%%%%%%%%%%%%%%%%%%
\section{Proof of Proposition \ref{Example}}
\label{Ex3}
In this section, we are going to prove Proposition \ref{Example}.
%%%%%%%%%%%%%
\begin{proof}[Proof of {\rm (1)}] Let the assumptions in (1) be fulfilled.
Then, as in \eqref{Channie}, $F_j^0$ can be written in terms of the null forms. Hence it is clear by Lemma~\ref{NullNull} that we have
\begin{equation}
F_j^{0}(\pa u) \stackrel{{\mathcal X}_N}{\sim} 0,\quad 1\le j\le N,
\end{equation}
which implies the desired result immediately. 
\end{proof} 
%%%%%%%%%%%%%
%%%%%%%%%%%%%
\begin{proof}[Proof of {\rm (2)}] 
Suppose that $F$ depends only on $\pa u$, i.e., $F=F(\pa u)$, and that 
Condition~$\ref{AlinhacCond}$ 
is satisfied. As before, we write $F^{\rm red}=F^{\rm red}(\omega, Y)$
for the reduced nonlinearity.
Let $u=(u_j)_{1\le j\le N}$ be a (local) solution to \eqref{OurSys} with 
the initial data \eqref{Data0}.
Set 
\begin{equation}
\label{Boss}
 v=(v_j)_{1\le j\le 5N}^{\rm T}=\left(u^{\rm T}, \pa_0 u^{\rm T}, \pa_1 u^{\rm T}, \pa_2 u^{\rm T}, \pa_3 u^{\rm T}\right)^{\rm T}, 
w=(w_l)_{1\le l\le N_0}^{\rm T}
\end{equation}
with
\begin{equation}
\label{Ofuku}
w_l=\sum_{a=0}^3 \sum_{k=1}^N h_{l, ka} \left(\pa_a u_k \right), \quad 1\le l\le N_0,
\end{equation}
where the constants $h_{l, ka}$ are from \eqref{Mil}.
We put
$u^*=(u^*_j)_{1\le j\le N^*}^{\rm T}=\left( v^{\rm T}, w^{\rm T} \right)^{\rm T}$
with $N^*=5N+N_0$.
Given $(f,g)\in {\mathcal X}_N$, from \eqref{OurSys} and \eqref{Data0} we can determine
the initial profile $\left(\ve^{-1}u^*(0,x), \ve^{-1}\pa_t u^*(0,x)\right)$ for $\ve\in (0,1]$,
and we denote this initial profile by 
$\left(f^*(x;\ve), g^*(x;\ve) \right)=\left(f^*[f,g](x;\ve), g^*[f,g](x;\ve) \right)$.
Since we can see that $(f^*,g^*)$ converges to a $C^\infty_0\times C^\infty_0$-function as $\ve\to +0$, 
we define $\bigl(f^*(x;0), g^*(x;0)\bigr)=\lim_{\ve\to +0}\bigl(f^*(x;\ve), g^*(x;\ve)\bigr)$.
Then it is easy to show that $(f^*, g^*)\in {\mathcal X}_{N^*}$.
Note that the first $N$ components of $f^*$ (resp.~$g^*$) are nothing but
$f$ (resp.~$g$) by definition. 
Now we define a subset $D$ of ${\mathcal X}_{N^*}$ by
$$
D=\bigl\{ \left(f^*[f,g], g^*[f,g] \right);\ (f,g)\in {\mathcal X}_N \bigr\}.
$$
Since $F_j$ is a homogeneous polynomial of degree $2$, we can write 
\begin{equation}
\label{Tom}
F_j(\pa u)=\sum_{k, l=1}^N \sum_{b,c=0}^3 q_{j}^{kb, lc} (\pa_b u_k)(\pa_c u_l)
\end{equation}
with some constants $q_j^{kb, lc}$.
Recall that $v_j=u_j$ and $v_{j+(a+1)N}=\pa_a u_j$ for $1\le j\le N$ and $0\le a\le 3$.
We put
\begin{align*}
F_j^*(\pa u^*):=&\sum_{k, l=1}^N \sum_{b,c=0}^3 q_{j}^{kb, lc} (\pa_b v_k)(\pa_c v_l),\\
F_{j+(a+1)N}^*(\pa u^*)
:=& \sum_{k, l=1}^N\sum_{b,c=0}^3
q_{j}^{kb, lc}\left(
(\pa_b v_k)(\pa_c v_{l+(a+1)N})+(\pa_b v_{k+(a+1)N})(\pa_c v_l)
\right)
\end{align*}
for $1\le j\le N$ and $0\le a\le 3$, so that we have
\begin{equation}
\label{Lime}
\text{$F_j^*(\pa u^*)=F_j(\pa u)$ and $F_{j+(a+1)N}^*(\pa u^*)=\pa_a\left(F_j(\pa u)\right)$.}
\end{equation}
Then we find that $u^*$ satisfies the extended system
\begin{equation}
\label{HolTheGarasshi}
\begin{cases}
\dal v_j=F_j^*(\pa u^*),  & 1\le j\le N,\\
\dal v_{j+(a+1)N} =
F_{j+(a+1)N}^*(\pa u^*), & 1\le j\le N, \ 0\le a\le 3,\\
 \dal w_l=\sum_{a=0}^3 \sum_{k=1}^N h_{l, ka} 
F_{k+(a+1)N}^* (\pa u^*),
 & 1\le l \le N_0
\end{cases}
\end{equation}
with initial data
\begin{equation}
\label{Nisesshi00}
u^*(0,x)=\ve f^*[f,g](x; \ve),\ (\pa_t u^*)(0,x)=\ve g^*[f,g](x; \ve),\quad  x\in \R^3.
\end{equation}

Conversely, 
let $u^*=(u_j^*)_{1\le j\le N^*}^{\rm T}=\left(v^{\rm T}, w^{\rm T}\right)^{\rm T}$ be the classical solution to \eqref{HolTheGarasshi}-\eqref{Nisesshi00}.
If we put
\begin{equation}
u=(u_j)_{1\le j\le N}^{\rm T}:=(u_j^*)_{1\le j\le N}^{\rm T}
\left(=(v_j)_{1\le j\le N}^{\rm T}\right),
\label{Prim}
\end{equation}
then $u$ is apparently the solution to the original problem \eqref{OurSys}-\eqref{Data0}.
Moreover, from the uniqueness of the solution we find that
\eqref{Boss}, \eqref{Ofuku}, and \eqref{Lime} are valid for all $t\ge 0$.
To sum up, we have proved that solving \eqref{OurSys}-\eqref{Data0} with $(f,g)\in {\mathcal X}_N$
is equivalent to solving \eqref{HolTheGarasshi}-\eqref{Nisesshi00} with $(f^*, g^*)\in D(\subset {\mathcal X}_{N^*})$.

Now we are going to prove that
Condition~$\ref{OurCond}$
is satisfied for \eqref{HolTheGarasshi},
if it is viewed as a system of $u^*$.
It is trivial to check the condition \eqref{Form01}, and we concentrate on the condition \eqref{Form02}.
Let $u^*=(v^{\rm T}, w^{\rm T})^{\rm T}$ be the solution
to \eqref{HolTheGarasshi}-\eqref{Nisesshi00}, and let 
$u$ be given by \eqref{Prim}.
As before, we write $r=|x|$, and 
$\omega=|x|^{-1} x=(\omega_1, \omega_2, \omega_3)\in S^2$.
We always assume $r \ge t/2\ge 1$ in this subsection
from now on. Hence we have
$$
 r^{-1}\le 4(1+t+r)^{-1}.
$$
Note that $\Phi \stackrel{D}{\sim} \Psi$ in this proof means 
\eqref{Tama} with $u$ being replaced by $u^*$.

For $N\times 4$-matrix valued functions
${\varphi}=(\varphi_{j, a})
$ and ${\psi}=(\psi_{j, a})
$ with $1\le j\le N$ and $0\le a\le 3$, we define
$$
F_j^\dagger(\varphi, \psi):=\sum_{k, l=1}^N\sum_{b,c=0}^3 
q_{j}^{kb, lc} \varphi_{k, b} \psi_{l, c}, \quad 1\le j\le N
$$
with the coefficients $q_{j}^{kb, lc}$ from \eqref{Tom}, 
so that we have
$$
F_j^*(\pa u^*)=F_j(\pa u)=F_j^\dagger(\pa u, \pa u), \quad 1\le j\le N
$$
by regarding $\pa u=(\pa_a u_j)_{1\le j\le N, 0\le a\le 3}$ as an 
$N\times 4$-matrix valued function.
We put $Z_0=0(=\omega_0\pa_t+\pa_0)$. Then from \eqref{AlinhacField} we get
\begin{equation}
\pa_a=Z_a-\omega_a\pa_t,\quad 0\le a\le 3.
\label{Hah}
\end{equation}
We write $Z^\star u=(Z_a u_j)_{1\le j\le N, 0\le a\le 3}$, and $\omega^\star \pa_t u=(\omega_a \pa_t u_j)_{1\le j\le N, 0\le a\le 3}$.
Regarding them as $N\times 4$ matrix-valued functions, we obtain
\begin{equation}
\label{Est00}
F_j(\pa u)=F_j^{\rm red}(\omega, \pa_t u)-F_j^\dagger(Z^\star u, \omega^\star \pa_t u)
{}-F_j^\dagger(\omega^\star \pa_t u, Z^\star u)+F_j^\dagger(Z^\star u, Z^\star  u),
\end{equation}
where we have used 
$F_j^\dagger(-\omega^\star \pa_t u, -\omega^\star \pa_t u)=F_j^{\rm red}(\omega, \pa_t u)$.
Let $s$ be a nonnegative integer.
Since we have $\sup_{|x|\ge t/2\ge 1}|\omega|_s<\infty$,
it follows from \eqref{CommRZG03} in Lemma~\ref{CommZG} that
\begin{align}
|F_j^\dagger (Z^\star u, \omega^\star \pa_t u)|_s
\le & C\left(|\pa u|_{[s/2]}|Z^\star u|_s+|Z^\star u|_{[s/2]}|\pa u|_s\right)\nonumber\\
\le & C\jb{t+r}^{-1}
\left(|\pa u|_{[s/2]}|u|_{s+1}+|u|_{[s/2]+1}|\pa u|_{s}\right),
\label{Est01}
\end{align}
which implies
$$
F_j^\dagger (Z^\star u, \omega^\star \pa_t u)\stackrel{D}{\sim} 0,\quad 1\le j\le N.
$$
In a similar manner we can show that
$$
F_j^\dagger(\omega^\star \pa_t u, Z^\star u)\stackrel{D}{\sim} F_j^\dagger(Z^\star u, Z^\star u)
\stackrel{D}{\sim}0,
$$
and we get
\begin{equation}
\label{FourSixteen}
F_j^*(\pa u^*)=F_j(\pa u) \stackrel{D}{\sim} F_j^{\rm red}(\omega, \pa_t u),\quad 1\le j \le N.
\end{equation}
Similarly to \eqref{Est01}, using \eqref{CommZG03} in Lemma~\ref{CommZG}
as well as \eqref{CommRZG03}, and recalling \eqref{Boss}, 
we find from \eqref{Est00} that
\begin{align*}
& \bigl|\pa_a \bigl(F_{j}(\pa u)\bigr)-\pa_a \bigl(F_j^{\rm red}(\omega, \pa_t u)\bigr)\bigr|_s\\
& \qquad \le C\jb{t+r}^{-1}\left(|\pa u^*|_{[s/2]}|u^*|_{s+1}+|u^*|_{[s/2]+1}|\pa u^*|_{s}\right).
\end{align*} 
In other words, we obtain
\begin{equation}
\label{Est06}
\pa_a\bigl(F_j(\pa u)\bigr) \stackrel{D}{\sim} \pa_a \bigl(
                                  F_j^{\rm red}(\omega, \pa_t u) 
                                                     \bigr),
\quad 1\le j\le N, \ 0\le a\le 3.
\end{equation}

From \eqref{Mil} and \eqref{Ofuku}, we get
\begin{equation}
\label{FourEight}
h_l(\omega, \pa_t u)+w_l=\sum_{a=0}^3\sum_{k=1}^N h_{l, ka} 
     Z_a u_k,\quad 1\le l\le N_0.
\end{equation}
Now 
we define $G_j(\omega, u^*, \pa u^*)$ for $1\le j\le 5N$ by
\begin{equation}
\label{Nisesshi01}
 G_j:=-\sum_{l=1}^{N_0} \sum_{k=1}^N g_{jl, k}(\omega) w_l (\pa_t v_k)
\biggl(=-\sum_{l=1}^{N_0} g_{jl}(\omega, \pa_t u)w_l\biggr)
\end{equation}
for $1\le j\le N$, and
\begin{equation}
\label{Nisesshi02}
G_{j+(a+1)N}:= -\sum_{l=1}^{N_0} \sum_{k=1}^N g_{jl, k}(\omega) 
 \left\{(\pa_a w_l)(\pa_t v_k)+w_l \left(\pa_t v_{k+(a+1)N}\right) 
\right\}
\end{equation} 
for $1\le j\le N$ and $0\le a\le 3$,
where $g_{jl, k}(\omega)$ and $g_{jl}(\omega, Y)$ are from \eqref{Yone}.
Then, by \eqref{AAB-1} in Condition~$\ref{AlinhacCond}$ 
and \eqref{FourEight}
we obtain
\begin{equation}
\label{Exp01}
F_j^{\rm red}(\omega, \pa_t u)-G_j(\omega, u^*, \pa u^*)
=\sum_{a=0}^3 \sum_{l=1}^{N_0} \sum_{k=1}^N h_{l, ka} 
 g_{jl}(\omega, \pa_t u)(Z_a u_k)
\end{equation}
for $1\le j\le N$.
Going a similar way to \eqref{Est01}, and using \eqref{FourSixteen}, 
we obtain
\begin{equation}
F_j^*(\pa u^*)\stackrel{D}{\sim}
F_j^{\rm red}(\omega, \pa_t u)\stackrel{D}{\sim} G_j(\omega, u^*, \pa u^*),
\quad 1\le j\le N.
\end{equation}
Since we have 
$$
\sup_{|x|\ge t/2\ge 1} r\left|\pa_a \bigl(g_{jl, k}(\omega)\bigr)\right|_s \le C, 
$$
and
\begin{align*}
G_{j+(a+1)N}(\omega, u^*, \pa u^*)=\pa_a\bigl(G_j(\omega, u^*, \pa u^*)\bigr)
{}+\sum_{l=1}^{N_0}\sum_{k=1}^N
\left(\pa_a\bigl(g_{jl, k}(\omega)\bigr)\right)w_l(\pa_t u_k)
\end{align*}
for $1\le j\le N$ and $0\le a\le 3$ (cf.~\eqref{Boss}), 
it follows from Lemma~\ref{CommZG} and \eqref{Exp01} that
$$
 \pa_a \bigl(F_j^{\rm red} (\omega, \pa_t u)\bigr)
 \stackrel{D}{\sim} G_{j+(a+1)N} (\omega, u^*, \pa u^*),
 \quad 1\le j\le N, \ 0\le a\le 3.
$$
Combining this with \eqref{Est06}, and remembering \eqref{Lime}, we obtain
\begin{equation}
\label{Est04}
 F_{j+(a+1)N}^* (\pa u^*)=\pa_a\bigl(F_j(\pa u)\bigr)
 \stackrel{D}{\sim} G_{j+(a+1)N} (\omega, u^*, \pa u^*)
\end{equation}
for $1\le j\le N$ and $0\le a\le 3$.

Now what is left to prove is
\begin{equation}
\sum_{a=0}^3\sum_{k=1}^N h_{l, ka} F_{k+(a+1)N}^*(\pa u^*)
\stackrel{D}{\sim} 0, \quad 1\le l\le N_0.
\end{equation}
We define 
$ \widetilde{q}_{j}^{\, kl}(\omega)=\sum_{b,c=0}^3 \left(q_{j}^{kb, lc}+q_{j}^{lb, kc}\right)\omega_b\omega_c$
for $1\le k<l\le N$,
and $\widetilde{q}_{j}^{\, kk}(\omega)=\sum_{b,c=0}^3 q_{j}^{kb, k c} \omega_b\omega_c$ for $1\le k\le N$
so that we can write
$$
F_j^{\rm red}(\omega, Y)=\sum_{1\le k\le l\le N} \widetilde{q}_{j}^{\, kl}(\omega) Y_kY_l.
$$
Let $\beta=\left(\beta_j(\omega)\right)_{1\le j \le N}^{\rm T}$ and
$M=M(\omega, Y)$ be from the condition \eqref{AA-1}. Since $M$ is quadratic in $Y$, we can write
it as 
$$
M(\omega, Y)=\sum_{1\le k\le l\le N} m_{kl}(\omega) Y_kY_l
$$
with some coefficients $m_{kl}(\omega)$.
Then from \eqref{AA-1} we find
\begin{equation}
\label{Est07}
\widetilde{q}_{j}^{\, kl}(\omega)
=\beta_j(\omega)m_{kl}(\omega), \quad 1\le j\le N,\ 1\le k\le l \le N.
\end{equation}
Observing that we have $\bigl|\pa_a\bigl(\widetilde{q}_{j}^{\, kl}(\omega)\bigr)\bigr|_s \le Cr^{-1}$, and that Lemma~\ref{CommZG} yields
$$
|Z_a\pa_t u_k|_s=|Z_a v_{k+N}|_s\le C\jb{t+r}^{-1}|u^*|_{s+1},
$$
we obtain from 
\eqref{Hah} and \eqref{Est07} that
\begin{align}
\pa_a \bigl(F_j^{\rm red}(\omega, \pa_t u) \bigr)
=& \sum_{1\le k\le l \le N} \bigl(\pa_a \bigl(\widetilde{q}_{j}^{\, kl}(\omega)\bigr)\bigr)(\pa_t u_k)(\pa_t u_l)
\nonumber\\
& {}+\sum_{1\le k\le l\le N} \widetilde{q}_{j}^{\, kl}(\omega)
\left((\pa_a\pa_t u_k)(\pa_t u_l)+(\pa_t u_k)(\pa_a \pa_t u_l)\right)
\nonumber\\
\stackrel{D}{\sim} & -\omega_a\sum_{1\le k\le l \le N} \widetilde{q}_{j}^{\, kl}(\omega)
\left((\pa_t^2 u_k)(\pa_t u_l)+(\pa_t u_k)(\pa_t^2 u_l)\right)
\nonumber\\
= &-\omega_a\beta_j(\omega)
\widetilde{M}(\omega, \pa_t u, \pa_t^2 u),
\label{FourTwentyNine}
\end{align}
where we have set
$$
\widetilde{M}(\omega, \pa_t u, \pa_t^2 u)=\sum_{1\le k\le l\le N} 
m_{kl}(\omega)\left((\pa_t^2 u_k)(\pa_t u_l)+(\pa_t u_k)(\pa_t^2 u_l)\right).
$$
Hence from \eqref{Lime}, \eqref{Est06}, and \eqref{FourTwentyNine}, we get
\begin{align*}
\sum_{a=0}^3\sum_{k=1}^N h_{l, ka} F_{k+(a+1)N}^*(\pa u^*)
\stackrel{D}{\sim}  &  
\sum_{a=0}^3 \sum_{k=1}^N h_{l, ka} \pa_a\bigl( F_k^{\rm red}(\omega, \pa_t u)\bigr)\\
\stackrel{D}{\sim} & -\sum_{a=0}^{3}\sum_{k=1}^N h_{l, k a} \omega_a \beta_k(\omega)
\widetilde{M}(\omega, \pa_tu, \pa_t^2 u)\\
=& -h_l\left(\omega, \beta(\omega)\right)\widetilde{M}(\omega, \pa_t u, \pa_t^2 u)
=0, \quad 1\le l \le N_0,
\end{align*}
where the last identity comes from \eqref{AAB-2}.
This completes the proof. 
\end{proof}
%%%%%%%%%%%%%%%%%%%%%%%%%%%%%%%%%%%%%%%%%%%%%%
\section{Proof of Theorem \ref{GE}}
\label{ProofGlobal}
Suppose that the assumptions in Theorem \ref{GE} are fulfilled.
Let $(f,g)\in D(\subset {\mathcal X}_N)$, and let $u=(v^{\rm T}, w^{\rm T})^{\rm T}$ be a (local) solution to \eqref{OurSys}-\eqref{Data0}
in $[0, T_0)\times \R^3$ with some $T_0>0$.
Since $u$ depends on the parameter $\ve$ from \eqref{Data0}, we sometimes write
$u=u(t,x; \ve)$ if we want to indicate the dependence of $u$ on $\ve$ explicitly,
but we sometimes omit $\ve$ and only write $u=u(t,x)$ for simplicity of expression. 
Similar notation will be used for functions depending on the parameter $\ve$ in what follows.
Since $(f,g)\in {\mathcal X}_N$, there is a positive constant $R$ such that
\begin{equation}
\label{SupportCond02}
\supp f(\cdot;\ve)\cup \supp g(\cdot;\ve)\subset B_R,\quad \ve\in (0,1],
\end{equation}
where $B_R$ is defined by \eqref{DefBall}. By the finite propagation property,
\eqref{SupportCond02} implies
\begin{equation}
\label{SupportCond03}
\supp u(t, \cdot;\ve)\subset B_{t+R},\quad t\in [0, T_0).
\end{equation}

We set
\begin{align*}
u_j^{0}(t,x; \ve):=  {\mathcal U}_0[f_j(\cdot; \ve), g_j(\cdot; \ve)](t,x), \quad
 \widetilde{u}_j(t,x;\ve):=   u_j(t,x;\ve)-\ve u_j^{0}(t,x;\ve)
\end{align*}
for $1\le j\le N$.
In other words, $u_j^{0}(t,x;\ve)$ is the solution to the free wave equation $\dal u_j^0=0$
with the initial condition 
$$
 u_j^{0}(0, x;\ve)=f_j(x;\ve),\ (\pa_t u_j^{0})(0, x; \ve)=g_j(x;\ve),\quad 1\le j\le N.
$$
We put $u^{0}=(u_j^{0})_{1\le j\le N}^{\rm T}$ and $\widetilde{u}=(\widetilde{u}_j)_{1\le j\le N}^{\rm T}$.
We also define
\begin{align*}
\bigl((v^0)^{\rm T}, (w^0)^{\rm T}\bigr)=&(v^0_1,\ldots,v^0_{N'}, w^0_1, \ldots, w^0_{N''}):=(u_1^0,\ldots, u^0_{N'},u_{N'+1}^0,\ldots, u^0_N),\\
\bigl(\widetilde{v}^{\rm T}, \widetilde{w}^{\rm T}\bigr)=&(\widetilde{v}_1,\ldots, \widetilde{v}_{N'}, \widetilde{w}_1, \ldots, \widetilde{w}_{N''}):=(\widetilde{u}_1,\ldots, \widetilde{u}_{N'}, \widetilde{u}_{N'+1}, \ldots, \widetilde{u}_N).
\end{align*}

Recall the definition
of $|\cdot|_s$, $\|\cdot\|_s$, and $|\cdot|_{Z,s}$ given in \eqref{InvariantNorm} and \eqref{InvariantAlinhac}.
Let $m$, $\lambda$, and $\rho$ be the fixed constants 
from the assumptions of Theorem \ref{GE}.
For $0<T\le T_0$, we define
\begin{align}
e_\ve[\widetilde{u}](T):=&
\sup_{(t,x)\in [0, T)\times \R^3}
\jb{t+|x|}^{-\lambda+1}|\widetilde{v}(t,x; \ve)|_{m+1}
\nonumber\\
&{}+\sup_{(t,x)\in [0, T)\times \R^3} \jb{t+|x|}\jb{t-|x|}^{\rho}|\widetilde{w}(t,x; \ve)|_{m+2} \nonumber\\
&{}+\sup_{t\in [0, T)} \left((1+t)^{-\lambda}\|\pa \widetilde{v}(t, \cdot; \ve)\|_{2m}+\|\pa \widetilde{w}(t,\cdot; \ve)\|_{2m}\right)
\nonumber\\
&{}+\left(
      \int_0^T\int_{\R^3} \frac{|\widetilde{u}(t, x; \ve)|_{Z, 2m}^2}{\jb{t}^{4\lambda}\jb{t-|x|}^2}
      dx dt
      \right)^{1/2}.
      \label{Amount}
\end{align}

For a smooth function $\varphi=\varphi(x)$ and a nonnegative integer $s$, we define
$$
|\varphi(x)|_{s,*}:=\sum_{|\alpha|\le s} \jb{x}^{|\alpha|}|\pa_x^\alpha \varphi(x)|,
\quad \|\varphi\|_{s,*}:=\left(\sum_{|\alpha|\le s} \bigl\|\jb{\,\cdot\,}^{|\alpha|}
\pa_x^\alpha \varphi\|_{L^2}^2\right)^{1/2},
$$
and we put 
\begin{align*}
M_0:=& \sup_{\ve\in [0,1]} (\|f(\cdot;\ve)\|_{2m+1,*}+\|g(\cdot;\ve)\|_{2m,*})
\\
& {}+\sup_{\ve\in [0,1], x\in \R^3}\jb{x}^3\left(|f(x;\ve)|_{2m-1,*}+
   |g(x,\ve)|_{2m-2,*}\right).
\end{align*}
Lemmas \ref{AlinhacGhost} and \ref{Asakura} yield
\begin{align}
& \sup_{(t,x)\in \R_+\times \R^3}\jb{t+|x|}\jb{t-|x|} 
\bigl|u^{0}(t,x;\ve)\bigr|_{2m-2}
\nonumber\\
& \ 
{}+\sup_{t\in \R_+} \bigl\|\pa u^{0}(t, \cdot;\ve)\bigr\|_{2m}
{}+\left(
 \int_0^\infty\int_{\R^3} \frac{\bigl|u^{0}(t,x;\ve)\bigr|_{Z, 2m}^2}{\jb{t-|x|}^2} dx dt 
\right)^{1/2}  \le C_0
\label{FreeEst}
\end{align}
for $\ve\in (0,1]$, where $C_0$ is a positive constant depending only on $M_0$. 

We are going to prove the following:
\begin{proposition}\label{bootstrap}
In the situation above, 
there exists a positive constant $A_0=A_0(M_0, R)$, which is independent of $T_0$,
such that
\begin{equation}
\label{AprioriStart}
 e_\ve[\widetilde{u}](T)\le A\ve^2
\end{equation}
implies
\begin{equation}
\label{AprioriEnd}
e_\ve[\widetilde{u}](T)\le \frac{A}{2}\ve^2,
\end{equation}
provided that $A\ge A_0$, $0<\ve\le \min\{1, 1/A\}$, and $0<T< T_0$.
\end{proposition}
Once we have the proposition above, we can easily obtain Theorem \ref{GE} in the following way.
%%%
\begin{proof}[Proof of Theorem~\ref{GE}]
If $A(\ge A_0)$ is sufficiently large, then
\eqref{AprioriStart} is true for some small $T>0$. Let $T_*$ be the supremum of all $T\in (0, T_0)$
such that \eqref{AprioriStart} is valid. Note that we have $T_*>0$.
Let $\ve\in (0, \ve_0]$ with $\ve_0:=\min\{1, 1/A\}$. 
Suppose that we have $T_*<T_0$; then, since we have $e_\ve[\widetilde{u}](T_*)\le A\ve^2$, Proposition \ref{bootstrap} implies that $e_\ve[\widetilde{u}](T_*)\le A\ve^2/2$, and from the continuity of $e_\ve$ in $T$, 
we see that \eqref{AprioriStart} is valid for some $T>T_*$; 
this contradicts the definition of $T_*$. 
Thus we conclude that $T_*=T_0$. 
In other words, we find that \eqref{AprioriStart}
is valid as long as the local solution exists, provided that $\ve\in (0, \ve_0]$.
Now, with the help of \eqref{FreeEst}, we 
see that $\sum_{|\alpha|\le 2}\|\pa^\alpha u(t,\cdot)\|_{L^\infty(\R^3)}$
does not tend to infinity in finite time,
and the local existence theorem implies the global existence of the solution
for $\ve\in (0,\ve_0]$ (see H\"ormander~\cite[Theorem~6.4.11]{Hoe97} for instance).
Since \eqref{AprioriStart} holds for any $T>0$, we immediately get
\eqref{Fukkie-t} and \eqref{Chato-t}, and we also obtain \eqref{Fukkie} and \eqref{Chato}
with the help of  \eqref{FreeEst}. 
This completes the proof. 
\end{proof}

The rest of this section is devoted to the proof of Proposition \ref{bootstrap}.
\begin{proof}[Proof of Proposition~\ref{bootstrap}]
We always assume $0<\ve\le \min\{1, 1/A\}$ in what follows,
so that we have $0<A\ve \le 1$.
The letter $C$ in this proof indicates various positive constants
which may depend on $M_0$ and $R$, but are independent of $A$, $\ve$, $T$, and $T_0$. 
The proof will be divided into several steps.

{\it Step 1: Basic estimates.}
Assume that \eqref{AprioriStart} is valid. Then \eqref{FreeEst} implies
\begin{equation}
\label{Boot01}
e_\ve[u](T)\le C_0\ve+A\ve^2\le C\ve,
\end{equation}
where $e_\ve[u](T)$ is defined by replacing 
$\widetilde{u}=(\widetilde{v}^{\rm T}, \widetilde{w}^{\rm T})^{\rm T}$
in \eqref{Amount} with $u=(v^{\rm T}, w^{\rm T})^{\rm T}$.
As before, we write $r=|x|$ and $\omega=|x|^{-1}x$.
Taking \eqref{MultiGamma} and \eqref{DerivativeG} into account, we obtain from \eqref{Boot01} and Lemma~\ref{Frame02} that
\begin{align}
\label{Boot02}
|\pa v(t,x)|_m\le C\ve \jb{t+r}^{\lambda-1}\jb{t-r}^{-1},\\
\label{Boot03}
|\pa w(t,x)|_{m+1}\le C\ve \jb{t+r}^{-1}\jb{t-r}^{-1-\rho}
\end{align}
for $(t,x)\in [0, T)\times \R^3$. Thus we have
\begin{equation}
\label{Boot02a}
|\pa u(t,x)|_m\le C\ve \jb{t+r}^{\lambda-1}\jb{t-r}^{-1}.
\end{equation}

Let $s$ be a nonnegative integer with $s\le 2m$.
From \eqref{Form01}, we have
 \begin{align}
 \label{Basic01}
|F_j(u, \pa u)|_{s} \le & C\left(|\pa u|_{[s/2]}|\pa u|_s+|w|_{[s/2]} |\pa v|_{s}+|\pa v|_{[s/2]} |w|_{s}
\right),\ 1\le j\le N', \\
 \label{Basic02}
|F_j(u,\pa u)|_s\le & C|\pa u|_{[s/2]}|\pa u|_s,\ N'+1\le j \le N.
 \end{align}
We set
$$
%\begin{equation}
%\label{InnerSet}
\Lambda:=\left\{(t,x)\in \R_+ \times \R^3; r \ge %\frac{t}{2}
t/2
\ge 1 \right\},
%\end{equation}
$$
and $\Lambda^{\rm c}:=(\R_+\times \R^3)\setminus \Lambda$.
If $(t,x)\in \Lambda^{\rm c}$, then we have either $r<t/2$ or $t<2$, and we obtain
$$
\jb{t-r}^{-1}\le C\jb{t+r}^{-1},\quad (t,x)\in \Lambda^{\rm c}.
$$
Hence, using \eqref{Boot01} and \eqref{Boot02a},
we see from \eqref{Basic01} and \eqref{Basic02} that
\begin{align}
\label{Basic03}
\sum_{j=1}^{N'}|F_j(u, \pa u)|_{s} 
\le & C\ve\left(\jb{t+r}^{-1-\rho}|\pa u|_{s}+\jb{t+r}^{\lambda-1} 
\frac{|w|_{s}}{\jb{t-r}} \right),
\\
 \label{Basic04}
\sum_{j=N'+1}^N
  |F_j(u,\pa u)|_{s}\le & C\ve \jb{t+r}^{\lambda-2} |\pa u|_{s}
\end{align}
at $(t,x)\in \Lambda^{\rm c}$, where we have used $1+\rho<2-\lambda$.

For any nonnegative integer $s$ and 
any smooth function $\varphi=\varphi(\omega)$ on $S^2$,
we have
$$
 \sup_{(t,x)\in \Lambda} \sum_{|\alpha|\le s} 
\left| \Gamma^\alpha \bigl(\varphi(r^{-1}x)\bigr)\right| <\infty.
$$
Hence, from \eqref{AssA} we get
\begin{equation}
\label{Basic06}
|G_j(\omega, u, \pa u)|_s 
\le C\left(\left(|w|_{[s/2]}+|\pa w|_{[s/2]}\right)|\pa v|_{s}
 {}+|\pa v|_{[s/2]}\left(|w|_s+|\pa w|_s\right)\right)
\end{equation}
in $\Lambda$ for $1\le j\le N'$.
It follows from \eqref{Boot01}, \eqref{Boot02}, \eqref{Boot03}, and \eqref{Basic06} that
\begin{align}
\label{Basic07}
\sum_{j=1}^{N'}|G_j
|_{s}
\le  C\ve\jb{t+r}^{-1}\left(\jb{t-r}^{-\rho}|\pa v|_{s}
 {}+\jb{t+r}^{\lambda}\left(\frac{|w|_{s}}{\jb{t-r}}+|\pa w|_{s}\right)\right)
\end{align}
at $(t,x)\in \Lambda$. 
Recalling \eqref{Tama} in Definition~\ref{EquiNL}, 
and observing that
$$
|\pa u|_{[s/2]}|u|_{s+1}\le C\left(
|\pa u|_{[s/2]}|\Gamma u|_s+|\pa u|_s|u|_{[s/2]+1}\right),
$$
we obtain from \eqref{Form02} that
\begin{align}
\label{Basic08}
\sum_{j=1}^{N} |F_j-G_j|_{s}\le C \ve \jb{t+r}^{\lambda-2}
\left(|\pa u|_{s}+\jb{t-r}^{-1}|\Gamma u|_{s}\right)
\end{align}
at $(t,x)\in \Lambda$, where we have set $G_j\equiv0$ for $N'+1\le j\le N$.

From \eqref{Basic03}, \eqref{Basic07} and \eqref{Basic08}, we obtain
\begin{align}
\sum_{j=1}^{N'}|F_j|_{s}\le & C\ve \left(\jb{t+r}^{-1}\jb{t-r}^{-\rho}|\pa u|_{s}+\jb{t+r}^{\lambda-1}
\left(\frac{|w|_{s}}{\jb{t-r}}+|\pa w|_{s}\right)\right)
\nonumber\\
& {}+C\ve 
\jb{t+r}^{\lambda-2}\jb{t-r}^{-1} |\Gamma u|_{s} 
\label{Basic09}
\end{align}
in $[0,T)\times \R^3$.
By \eqref{Basic04} and \eqref{Basic08}, we also obtain
\begin{align}
\label{Basic10}
\sum_{j=N'+1}^N |F_j|_{s} \le C \ve \jb{t+r}^{\lambda-2} 
\left(|\pa u|_{s}
{}+\jb{t-r}^{-1}|\Gamma u|_{s}
\right)
\end{align}
in $[0,T)\times \R^3$.

{\it Step 2: The energy estimates.} We put
$$
I_\ve (T)=\int_0^T \jb{t}^{\lambda-1} \left(\int_{\R^3} \frac{|u(t,x; \ve)|_{Z, 2m}^2}{\jb{t-r}^{2}}
 dx\right)^{1/2} dt.
$$
Writing $\jb{t}^{\lambda-1}=\jb{t}^{3\lambda-1}\jb{t}^{-2\lambda}$, and using the Schwarz inequality, we obtain from \eqref{Boot01}
that
\begin{align}
\label{SpaceTime}
I_\ve (T) 
\le & \left(\int_0^\infty \jb{t}^{6\lambda-2}dt\right)^{1/2}
\left(\int_0^T \int_{\R^3}\frac{|u(t,x)|_{Z, 2m}^2}{\jb{t}^{4\lambda}\jb{t-r}^{2}}
 dx dt \right)^{1/2} \le C \ve, 
\end{align}
where we have used $6\lambda-2<-1$.
Because of \eqref{SupportCond03}, we can use Lemma~\ref{LindbladIneq} to
obtain
\begin{equation}
\left(\int_{\R^3}\frac{|w(t, x)|_s^2}{\jb{t-r}^2}dx\right)^{1/2}
\le C\|\pa w(t,\cdot)\|_s.
\label{LinTech}
\end{equation}
Lemma~\ref{frame03} implies
\begin{align}
& \jb{t+r}^{\lambda-2}\jb{t-r}^{-1}|\Gamma u|_s
\nonumber\\
& \qquad\qquad\qquad
\le C\jb{t+r}^{\lambda-1}\left(\jb{t-r}^{-1}|u|_{Z,s}+\jb{t+r}^{-1}|\pa u|_s\right).
\label{MMM}
\end{align}

Let $\mu\ge 0$.
In view of 
\eqref{Basic09}, \eqref{LinTech}, and \eqref{MMM} with $s=2m$, 
we obtain from \eqref{Boot01} and \eqref{SpaceTime} that
\begin{align}
\sum_{j=1}^{N'}\int_0^t \jb{\tau}^{-\mu} \|F_j(\tau)\|_{2m} d\tau\le &
C\ve\int_0^t \jb{\tau}^{-1-\mu}\left(\|\pa u(\tau)\|_{2m}+\jb{\tau}^{\lambda}\|\pa w(\tau)\|_{2m}
\right)d\tau\nonumber\\
& {}+C\ve I_\ve (T)\nonumber\\
\le & C\ve^2 \int_0^t \jb{\tau}^{\lambda-1-\mu} d\tau+C\ve^2, \quad 0\le t\le T.
\label{Basic11}
\end{align}
Similarly, using \eqref{Basic10} with $s=2m$ instead of \eqref{Basic09}, 
we obtain
\begin{align}
\sum_{j=N'+1}^N \int_0^t \|F_j(\tau)\|_{2m} d\tau
\le & C\ve \left(\int_0^t \jb{\tau}^{\lambda-2}\|\pa u(\tau)\|_{2m} d\tau+I_\ve(T)\right)
\nonumber\\
\label{Basic12}
\le & C\ve^2 \int_0^t \jb{\tau}^{2\lambda-2} d\tau+C\ve^2
\le C \ve^2, \quad 0\le t<T.
\end{align}

From \eqref{OurSys}, we have
\begin{align}
\label{DEq}
\dal \widetilde{u}_j=F_j(u, \pa u),\quad 1\le j\le N,
\end{align}
and $\widetilde{u}(0,x)=(\pa_t \widetilde{u})(0,x)\equiv 0$.
From \eqref{CommGamma} and \eqref{DEq} we get
\begin{equation}
\label{Higheq}
\dal\left(\Gamma^\alpha \widetilde{u}_j\right)=
(\Gamma_0+2)^{\alpha_0}\Gamma_1^{\alpha_1}\cdots\Gamma_{10}^{\alpha_{10}} 
\left(F_j(u, \pa u)\right),\quad 1\le j\le N
\end{equation}
for any multi-index $\alpha=(\alpha_0, \alpha_1, \ldots, \alpha_{10})$.
We also get 
$$
\|(\Gamma^\alpha \widetilde{u})(0)\|_{L^2(\R^3)}+\|(\pa_t \Gamma^\alpha \widetilde{u})(0)
\|_{L^2(\R^3)}\le C_{\alpha}\ve^2
$$
with some positive constant $C_\alpha$.
Applying Lemma~\ref{AlinhacGhost} with $\mu=2\lambda$ to \eqref{Higheq} 
for $|\alpha|\le 2m$,
and using \eqref{Basic11} and \eqref{Basic12},
we get
\begin{align}
\left(\int_0^T\int_{\R^3}\frac{|\widetilde{u}(\tau,x)|_{Z,2m}^2}{\jb{\tau}^{4\lambda}\jb{\tau-r}^{2}}dx d\tau \right)^{1/2}\le & C\ve^2+C\int_0^T \jb{\tau}^{-2\lambda} \|F(u,\pa u)(\tau, \cdot)\|_{2m} d\tau
\nonumber\\
\le & C\ve^2\left(1+\int_0^\infty \jb{\tau}^{-\lambda-1} d\tau \right)\le C\ve^2.
\label{Concl01}
\end{align}
Applying Lemma~\ref{AlinhacGhost} with $\mu=0$ to \eqref{Higheq} 
for $1\le j\le N'$ and $|\alpha|\le 2m$, and using \eqref{Basic11}, we obtain
\begin{align*}
\sum_{j=1}^{N'}\|\pa \widetilde{v}_j (t, \cdot)\|_{2m} 
\le & C\ve^2+C\sum_{j=1}^{N'}\int_0^t \|F_j(u,\pa u)(\tau,\cdot)\|_{2m} d\tau 
\le C\ve^2 \jb{t}^{\lambda}
\nonumber
\end{align*}
for $0\le t<T$, which implies
\begin{equation}
\label{Concl02}
\sup_{0\le t<T} (1+t)^{-\lambda} \|\pa \widetilde{v}(t,\cdot)\|_{2m}\le C\ve^2.
\end{equation}
Similarly, applying Lemma~\ref{AlinhacGhost} to \eqref{Higheq} for $N'+1\le j\le N$,
and using \eqref{Basic12} instead of \eqref{Basic11}, we obtain
\begin{equation}
\label{Concl03}
\sup_{0\le t<T} \|\pa \widetilde{w}(t,\cdot)\|_{2m}\le C\ve^2.
\end{equation}

{\it Step 3: Decay estimates for generalized derivatives of higher order.}
Now we turn our attention to the decay estimates.
By \eqref{Boot01} and Lemma~\ref{KlainermanIneq}, we have
\begin{equation}
\label{Boot04}
|\pa u(t,x)|_{2m-2}\le C\ve\jb{t+r}^{\lambda-1}
\jb{t-r}^{-1/2},\quad (t,x)\in [0,T)\times \R^3.
\end{equation}
It follows from \eqref{Boot02a}, \eqref{Boot04}, and 
\eqref{Basic02} with $s=2m-2$ that
$$ 
|F_j(u, \pa u)|_{2m-2}\le C\ve^2 \jb{t+r}^{2\lambda-2}\jb{t-r}^{-3/2},\quad N'+1\le j\le N.
$$
Hence, applying Lemma~\ref{Asakura} (with $\mu=1-2\lambda$) to \eqref{Higheq} for
$N'+1\le j\le N$ and $|\alpha|\le 2m-2$, and using \eqref{FreeEst}, we obtain
\begin{equation}
\label{Boot05}
\jb{t+r}^{1-2\lambda}|w(t,x)|_{2m-2}\le C\ve,\quad (t,x)\in[0,T)\times \R^3.
\end{equation}
Therefore Lemma \ref{Frame02} yields
\begin{equation}
\label{Boot06}
\jb{t+r}^{1-2\lambda}\jb{t-r}|\pa w(t,x)|_{2m-3}\le C\ve,\quad (t,x)\in [0,T)\times \R^3.
\end{equation}
Using 
\eqref{Boot01}, \eqref{Boot02a}, 
\eqref{Boot04}, and \eqref{Boot05}
to evaluate the right-hand side of \eqref{Basic01} with $s=2m-2$, we get
\begin{align*}
|F_j(u, \pa u)|_{2m-2}\le & C\ve^2\left(\jb{t+r}^{2\lambda-2}\jb{t-r}^{-\rho-(1/2)}
{}+\jb{t+r}^{3\lambda-2}\jb{t-r}^{-1}\right)\\
\le & C\ve^2 \jb{t+r}^{4\lambda-2} \jb{t-r}^{-1-\lambda},\quad 1\le j\le N'.
\end{align*}
Similarly to the derivation of \eqref{Boot05} and \eqref{Boot06}, Lemmas~\ref{Asakura}
and \ref{Frame02} lead to
\begin{equation}
\label{Boot07}
\jb{t+r}^{1-4\lambda}\left(|v(t,x)|_{2m-2}
+\jb{t-r}|\pa v(t,x)|_{2m-3}\right) \le C\ve
\end{equation}
for $(t,x)\in [0, T)\times \R^3$.

{\it Step 4: Decay estimates for generalized derivatives of lower order.}
By \eqref{Boot05} and \eqref{Boot07} we get
\begin{equation}
|\Gamma u(t,x)|_{2m-3}\le C|u(t,x)|_{2m-2}
\le C\ve \jb{t+r}^{4\lambda-1}.
\label{Runa}
\end{equation}

Using \eqref{Boot06}, \eqref{Boot07}, and \eqref{Runa}
to estimate the right-hand side of \eqref{Basic10} with $s=2m-3$,
we obtain 
\begin{align*}
\sum_{j=N'+1}^N |F_j(u, \pa u)|_{2m-3} \le & C \ve^2 \jb{t+r}^{5\lambda-3} \jb{t-r}^{-1} 
\le 
C\ve^2\jb{t+r}^{6\lambda-3}\jb{t-r}^{-1-\lambda}.
\end{align*}
Hence applying Lemma~\ref{Asakura} with $\mu=1+\rho$, and then using Lemma~\ref{Frame02}, we get
\begin{equation}
\label{Concl04}
\jb{t+r}\jb{t-r}^{\rho}
\left(|\widetilde{w}(t,x)|_{2m-3}+\jb{t-r}|\pa\widetilde{w}(t,x)|_{2m-4}\right)
\le C\ve^2
\end{equation}
for $(t,x)\in [0, T)\times \R^3$, because $1/2<\rho\le 1-6\lambda$.
Now \eqref{FreeEst} leads to 
\begin{equation}
\label{Boot09}
\jb{t+r}\jb{t-r}^{\rho}
\left(|w(t,x)|_{2m-3}+\jb{t-r}|\pa w(t,x)|_{2m-4}\right)
\le C\ve.
\end{equation}

Using \eqref{Boot04}, \eqref{Boot09}, and \eqref{Runa}
to estimate each term on the right-hand side of \eqref{Basic09} with $s=2m-4$, 
we obtain
\begin{align*}
\sum_{j=1}^{N'}|F_j(u, \pa u)|_{2m-4}\le &
C\ve^2\left(\jb{t+r}^{\lambda-2}\jb{t-r}^{-\rho-(1/2)}+\jb{t+r}^{\lambda-2}\jb{t-r}^{-\rho-1}\right)\\
& {}+C\ve^2\jb{t+r}^{5\lambda-3}\jb{t-r}^{-1}\\
\le & C\ve^2 \jb{t+r}^{\lambda-2} \jb{t-r}^{-\rho-(1/2)},
\end{align*}
where we have used $\jb{t+r}^{5\lambda-3}\jb{t-r}^{-1}\le \jb{t+r}^{\lambda-2}\jb{t-r}^{4\lambda-2}$ and
$\rho+(1/2)<2-4\lambda$.
Since $\rho+(1/2)>1$, Lemma~\ref{Asakura} with $\mu=1-\lambda$ implies
\begin{equation}
\label{Concl05}
\jb{t+r}^{1-\lambda} |\widetilde{v}(t,x)|_{2m-4}\le C\ve^2,\quad (t,x)\in [0,T)\times \R^3.
\end{equation}

{\it Step 5: Conclusion.} Since we have $m\le 2m-5$, it follows 
from \eqref{Concl01}, \eqref{Concl02}, \eqref{Concl03},
\eqref{Concl04}, and \eqref{Concl05} that
\begin{equation}
\label{FinaleBoot}
 e_\ve[\widetilde{u}](T)\le C_1\ve^2
\end{equation}
with an appropriate positive constant $C_1=C_1(M_0,R)$, which is independent of
$A$, $\ve$, $T$, and $T_0$.
Finally, if we choose $A_0=2C_1$, \eqref{FinaleBoot} immediately implies
\eqref{AprioriEnd} for $A\ge A_0$.
This completes the proof. 
\end{proof}
%%%%%%%%%%%%%%%%%%%%%%%%%%%%%%%%%%%%%%%%%%%%%$
\section{Proof of Theorem \ref{PointwiseAsymptotics}}
\label{ProofPointwise}
%%%%%%%%%%%%%%%%%%%%%%%%%%%%%%%%%%%%%%%%%%%%%%
%
We start this section with a fundamental observation.
We set
$$
\Lambda_0:=\left\{(t,r)\in [0, \infty)\times [0, \infty); r\ge %\frac{t}{2}
t/2
\ge 1 \right\},
$$
and we put $t_0(\sigma)=\max\{2, -2\sigma\}$ for $\sigma\in \R$ so that 
$$
\bigcup_{\sigma\in \R}\{(t, t+\sigma); t\ge t_0(\sigma)\}=\Lambda_0.
$$
We take a smooth and non-increasing function $t_1=t_1(\sigma)$ such that
$t_1(\sigma)=-2\sigma$ for $\sigma\le -2$, $t_1(\sigma)=2$ for $\sigma\ge 0$,
and $t_0(\sigma)\le t_1(\sigma)\le 2-\sigma$ for $-2<\sigma<0$.
We define $r_1(\sigma):=t_1(\sigma)+\sigma$.
Note that we have
\begin{equation}
\label{Aruruw}
1+|\sigma|\le 1+t_1(\sigma)+r_1(\sigma)=1+2t_1(\sigma)+\sigma\le 5(1+|\sigma|),\quad \sigma\in \R.
\end{equation}

%%%%%%%%%%%%%%%%%%%%%%%%%%%%%%%%%%%%%
\begin{lemma}\label{Asymptotics}
Suppose that $\mu>1$ and $\kappa\ge 0$.

Let $\varphi=\varphi(t, r, \omega)$ be a $C^1$-function of
$(t,r)\in \Lambda_0$ and $\omega\in S^2$, satisfying
\begin{equation}
\label{FunEq}
\pa_+ \varphi(t,r,\omega)=h(t, r, \omega),\quad (t,r)\in \Lambda_0, \  \omega\in S^2,
\end{equation}
where $\pa_+=\pa_t+\pa_r$.
Assume that there is a positive constant $C_0$ such that
\begin{equation}
\label{FunAssump}
 |h(t, r, \omega)|\le C_0(1+t+r)^{-\mu} (1+|t-r|)^{-\kappa},\quad (t,r)\in \Lambda_0, \ \omega\in S^2.
\end{equation}
Then, putting
\begin{equation}
\label{Construction}
\Phi(\sigma, \omega)=\varphi\left(t_1(\sigma), r_1(\sigma), \omega\right)
{}+\int_{t_1(\sigma)}^\infty h(\tau, \tau+\sigma, \omega) d\tau,
\quad (\sigma, \omega)\in \R\times S^2,
\end{equation}
we have
\begin{equation}
\label{FunError}
|\varphi(t, t+\sigma, \omega)-\Phi(\sigma, \omega)|
\le \frac{C_0}{2(\mu-1)}(1+2t+\sigma)^{-(\mu-1)}(1+|\sigma|)^{-\kappa}
\end{equation}
for $t\ge t_0 (\sigma)$ and $(\sigma, \omega) \in \R\times S^2$. 
\end{lemma}
%%%%%%%%%%%%%%%%%%%%%%%%%%%%%%%%%%%%%%%%%%%%%%%%%%%%%
\begin{proof}
Observe that $\Phi$ is well-defined because of the assumption \eqref{FunAssump}.

It is easy to see that 
$$
\varphi(t, t+\sigma, \omega)=\varphi\left(t_1(\sigma), r_1(\sigma), \omega\right)
+\int_{t_1(\sigma)}^t h(\tau, \tau+\sigma, \omega) d\tau
$$
for $t\ge t_0(\sigma)$ and $(\sigma, \omega)\in \R\times S^2$.
Hence \eqref{FunAssump} and \eqref{Construction} yield
\begin{align*}
|\varphi(t, t+\sigma, \omega)-\Phi(\sigma, \omega)|
\le 
& 
\int_t^\infty |h(\tau, \tau+\sigma, \omega)|d\tau
\\
\le 
& 
C_0(1+|\sigma|)^{-\kappa}\int_t^\infty (1+2\tau+\sigma)^{-\mu} d\tau
\end{align*}
for $t\ge t_0(\sigma)$ and $(\sigma, \omega)\in \R\times S^2$;
thus the direct calculation leads to \eqref{FunError}.
This completes the proof.
\end{proof}
%%%%%%%%%%%%%%%%%%%%%%%%%%%%%%%%%%%%%%%%%%%%%

From \eqref{Construction}, it is easy to see the following:
\begin{lemma}\label{AsymptoticsRegularity}
Let $\varphi$, $h$, and $\Phi$ be as in Lemma~$\ref{Asymptotics}$.
Assume in addition that
$h$ is a $C^1$-function of $(t,r)\in \Lambda_0$ and $\omega\in S^2$.
If there exist two constants $\mu'>1$ and $C>0$ such that
\begin{equation}
\sum_{p+|\alpha|=1}|\pa_r^p\Omega^\alpha h(t,r,\omega)|\le C (1+t+r)^{-\mu'},\quad (t,r)\in \Lambda_0,\ \omega\in S^2,
\end{equation}
then we have $\Phi\in C^1(\R\times S^2)$, where $\Omega$ is regarded as a differential operator on $S^2$.
\end{lemma}
%%%%%%%%%%%%%%%%%%%%%%%%%%%%%%%%%%%%%%%%%%%%%

Let us start the proof of Theorem \ref{PointwiseAsymptotics}.
Suppose that the assumptions in Theorem \ref{PointwiseAsymptotics}
are satisfied, and let $u=(v^{\rm T},w^{\rm T})^{\rm T}$ be the global solution
to \eqref{OurSys}-\eqref{Data0} with $(f,g)\in D$.
We define $u^{0}$, $v^{0}$, $w^{0}$,
$\widetilde{u}$, $\widetilde{v}$, and $\widetilde{w}$ as
in the previous section. 
Recall that, by Theorem~\ref{GE} and Lemma~\ref{Frame02}, we have
\begin{align}
& |v(t,x)|_{2}+\ve^{-1}|\widetilde{v}(t,x)|_{2}
{}+\left(1+\bigl|t-|x|\,\bigr|\right)\left(
|\pa v(t,x)|_1+\ve^{-1}|\pa \widetilde{v}(t,x)|_1\right)
\nonumber\\ 
& \qquad \le C\ve (1+t+|x|)^{\lambda-1}, 
\label{DE01}\\
& |w(t,x)|_{3}+\ve^{-1}|\widetilde{w}(t,x)|_{3}
{}+\left(1+\bigl|t-|x|\,\bigr|\right)\left(
|\pa w(t,x)|_{2}+\ve^{-1}|\pa \widetilde{w}(t,x)|_{2}\right)
\nonumber\\
& \qquad \le C\ve (1+t+|x|)^{-1}\left(1+\bigl|t-|x|\,\bigr|\right)^{-\rho}
\label{DE02}
\end{align}
for any $(t,x)\in [0, \infty)\times \R^3$ and $\ve\in (0, \ve_0]$,
where 
$\lambda\bigl(\in (0, 1/20)\bigr)$, $\rho\bigl(\in (1/2, 1-6\lambda]\bigr)$, and $\ve_0(>0)$ are from Theorem~\ref{GE},
and $C$ is a positive constant.
Remember also that we have \eqref{SupportCond02} and \eqref{SupportCond03}
for some $R>0$.

Let $0<\ve\le \ve_0$.
We define $U_j^0\in C^\infty_0(\R\times S^2)$ by
$$
U_j^{0}(\sigma, \omega)=
 {\mathcal F}_0[f_{j}(\cdot;\ve), g_{j}(\cdot;\ve)](\sigma, \omega),\quad (\sigma, \omega)\in \R\times S^2
$$
for $1\le j\le N$, where ${\mathcal F}_0$ is given by \eqref{FriedRad}.
Recall that \eqref{SupportCond02} implies $u_j^{0}(t, r \omega)=0$
for $|r-t|\ge R$, and
\begin{equation}
\label{SupportCond05}
U_j^{0}(\sigma, \omega)
=0,\quad |\sigma|\ge R
\end{equation}
(cf.~\eqref{RadSupport}).
Therefore \eqref{FriedAsymp} implies that for any $\kappa\ge 0$,
there exists a positive constant $C_\kappa$ such that
\begin{align}
&\left|ru_j^{0}(t,r\omega)-U_j^{0}(r-t, \omega)\right|
{}+\sum_{a=0}^3
\left|r(\pa_a u_j^{0})(t, r\omega)-\omega_a( \pa_\sigma U_j^{0})(r-t, \omega)\right| \nonumber\\
& \qquad \qquad \qquad \qquad
\le C_\kappa (1+t+r)^{-1}(1+|t-r|)^{-\kappa},\quad 1\le j\le N
\label{Fried01}
\end{align}
for $(t,r)\in \Lambda_0$ and $\omega\in S^2$. 
Indeed we have $(1+|t-r|)^{-\kappa}\ge (1+R)^{-\kappa}$ in the support of
the functions on the left-hand side of \eqref{Fried01}. 
We put $V_j^{0}=U_j^{0}$ for $1\le j\le N'$, and $W_k^{
0}=U_{N'+k}^{0}$ for $1\le k\le N''$.

For any smooth function $\psi$, we have 
\begin{equation}
\label{RaCoWave}
r \dal \psi(t,x)=\pa_+\pa_-\left(r\psi (t,x)\right)-r^{-1}(\Delta_\omega \psi)(t,x),
\quad (t,x)\in [0,\infty)\times \R^3,
\end{equation}
where $r=|x|$, and 
$\Delta_\omega=\sum_{k=1}^3\Omega_k^2$.
Hence by \eqref{DEq} we have
\begin{equation}
\label{AsympBasicEq}
\pa_+\pa_- \left(
 r\widetilde{u}_j(t, r\omega)\right)=H_j(t, r, \omega), 
\quad 1\le j\le N,
\end{equation}
where $H_j$ is given by
\begin{equation}
\label{Remainder}
H_j(t, r, \omega)=rF_j\bigl(u(t, r\omega), \pa u(t, r\omega) \bigr)+r^{-1}(\Delta_\omega
\widetilde{u}_j)(t, r\omega).
\end{equation}

We always suppose $(t, r)\in \Lambda_0$ (namely $r\ge t/2\ge 1$), and $\omega\in S^2$
from now on. Then we have $r^{-1}\le 4(1+t+r)^{-1}$.
In what follows, $C$ denotes various positive constants which are 
independent of $(t,r)\in \Lambda_0$, $\omega\in S^2$,
and $\ve\in (0, \ve_0]$.

Equations \eqref{DE01} and \eqref{DE02} yield
\begin{align}
\label{DE03}
|r^{-1}(\Delta_\omega \widetilde{v}_j)(t, r\omega)|\le & C\ve^2 (1+t+r)^{\lambda-2},\quad 1\le j\le N', \\
\label{DE04}
|r^{-1}(\Delta_\omega 
\widetilde{w}_k)
 (t,r\omega)|_1\le & C\ve^2 (1+t+r)^{-2}(1+|t-r|)^{-\rho},
\quad 1\le k\le N''.
\end{align}
Because of \eqref{Tama} and \eqref{Form02} we obtain
\begin{align}
|r F_j(u, \pa u)-r G_j(\omega,u,\pa u)|_1
\le & C|u|_2|\pa u|_1
\nonumber\\
\le & 
C\ve^2(1+t+r)^{2\lambda-2}(1+|t-r|)^{-1}
\label{DE05}
\end{align}
at $(t,x)=(t, r\omega)$ for $1\le j\le N$, 
where $G_j$ for $1\le j\le N'$ is
from Condition~\ref{OurCond},
and
we have set $G_j(\omega, u, \pa u) \equiv0$ for $N'+1\le j \le N$ as before.

We first construct the standard asymptotic profile $W$ for $w$,
and then the modified asymptotic profile $V$ for $v$.

{\it Step 1: Construction of $W$.}
From \eqref{DE04} and \eqref{DE05}, we get
\begin{equation}
\label{DE06}
|H_{N'+k}(t, r, \omega)|_1\le C\ve^2 (1+t+r)^{2\lambda-2}(1+|t-r|)^{-\rho},\quad 1\le k\le N''. 
\end{equation}
We put
$$
 \widetilde{w}_{k,-}(t, r, \omega)
:=\pa_-\left(r\widetilde{w}_k(t, r\omega)\right)
\bigl(=\pa_-\left(r \widetilde{u}_{N'+k}(t, r\omega)\right)\bigr), \quad 1\le k\le N''.
$$
Equation \eqref{DE02} leads to 
\begin{align}
\label{DE07}
 \left|\widetilde{w}_{k,-}(t, r, \omega)\right|\le C\ve^2 (1+|t-r|)^{-1-\rho},\quad 1\le k\le N''.
\end{align}
Because of \eqref{AsympBasicEq} and \eqref{DE06},
Lemma~\ref{Asymptotics} implies 
\begin{equation} 
\label{DE08}
\bigl|
 \widetilde{w}_{k,-}(t, t+\sigma, \omega)-\ve\widetilde{W}_{k,-} (\sigma, \omega)
\bigr|\le
C\ve^2(1+2t+\sigma)^{2\lambda-1}(1+|\sigma|)^{-\rho}
\end{equation}
for $1\le k\le N''$, $t\ge t_0(\sigma)$, and $(\sigma, \omega)\in \R\times S^2$, where
$\widetilde{W}_{k,-}$ is defined by
\begin{equation}
\label{DefTildeWMinus}
\ve\widetilde{W}_{k,-}(\sigma, \omega)=\widetilde{w}_{k,-}\left(t_1(\sigma), r_1(\sigma), \omega\right)
{}+\int_{t_1(\sigma)}^\infty H_{N'+k}(\tau, \tau+\sigma, \omega) d\tau.
\end{equation}
%%%%%%%%%%%%%%%%%% 
Note that we have $\widetilde{W}_{k,-}\in C^1(\R\times S^2)$,
because of Lemma~\ref{AsymptoticsRegularity} and \eqref{DE06}.
By \eqref{DE07} and \eqref{DE08} we get
$$
\bigl|\widetilde{W}_{k,-}(\sigma, \omega)\bigr|\le C\ve
\bigl((1+|\sigma|)^{-1-\rho}
+(1+2t+\sigma)^{2\lambda-1}(1+|\sigma|)^{-\rho}\bigr).
$$ 
Therefore, taking the limit as $t\to\infty$, we obtain
\begin{equation}
\label{DE09}
\bigl|\widetilde{W}_{k,-} (\sigma, \omega) 
\bigr|\le
 C\ve (1+|\sigma|)^{-1-\rho}, \quad 1\le k\le N'',
\ (\sigma,\omega)\in \R\times S^2.
\end{equation}
Furthermore, since
\eqref{SupportCond03} yields $\widetilde{w}_{k,-}(t,t+\sigma, \omega)=H_{N'+k}(t,t+\sigma, \omega)=0$ when $\sigma\ge R$, 
\eqref{DefTildeWMinus} implies
\begin{equation}
\label{DE10}
 \widetilde{W}_{k,-}(\sigma, \omega)=0 \quad \text{for $1\le k\le N''$, $\sigma\ge R$, and $\omega\in S^2$.}
\end{equation}
For $1\le k\le N''$, we define
\begin{equation}
\label{DE11}
\widetilde{W}_k(\sigma, \omega)=\frac{1}{2}\int_\sigma^\infty \widetilde{W}_{k,-}(\tau, \omega) d\tau,
\quad
(\sigma, \omega)\in \R\times S^2,
\end{equation}
so that we have 
\begin{equation}
\label{Eruruw}
-2\pa_\sigma \widetilde{W}_k(\sigma, \omega)=\widetilde{W}_{k,-}(\sigma, \omega),
\quad (\sigma,\omega)\in \R\times S^2.
\end{equation}
By \eqref{DE10} and \eqref{DE11} we get
\begin{equation}
\label{OuterVanish}
\widetilde{W}_k(\sigma, \omega)=0,\quad \sigma\ge R,\ \omega\in S^2.
\end{equation}
%%%%%%%%%%%%%%%%%
Using \eqref{blue03'} and \eqref{DE02}, we obtain from \eqref{DE08} and \eqref{Eruruw} that
\begin{align}
& \sum_{a=0}^3 \bigl|
                     r \bigl(\pa_a \widetilde{w}_k\bigr)(t, r\omega)-
                        \ve\omega_a \bigl(\pa_\sigma \widetilde{W}_k\bigr)(r-t, \omega)
                  \bigr| 
\nonumber\\
& \qquad\qquad \le C\ve^2(1+t+r)^{2\lambda-1}(1+|t-r|)^{-\rho},\quad  1\le k\le N''.
\label{AsFi01}
\end{align}

We set
$$
 {\widetilde{w}}_{k}^{*}(t, r, \omega)=r\widetilde{w}_k(t,r\omega), \quad 1\le k\le N'',
$$
so that we have $\pa_-\widetilde{w}_{k}^{*}(t,r,\omega)=\widetilde{w}_{k,-}(t,r, \omega)$.
Now we are going to prove
\begin{align}
\bigl|{\widetilde{w}}_{k}^{*}(t, t+\sigma, \omega)-\ve \widetilde{W}_k(\sigma, \omega)\bigr| 
 \le C\ve^2 (1+2t+\sigma)^{2\lambda-1} (1+|\sigma|)^{1-\rho} \label{DE13}
\end{align}
for $t\ge t_0(\sigma)$ and $(\sigma,\omega) \in \R\times S^2$.
If $\sigma>R$, then \eqref{DE13} is trivial because of \eqref{SupportCond03} and \eqref{OuterVanish}.
To treat the case where $\sigma\le R$, 
we put
\begin{align*}
D_1=& \{(t,\sigma);\, 
t\ge \max\{t_0(\sigma), 2+(R-\sigma)/2\},\, \sigma\le R\},\\
D_2=& \{(t,\sigma);\, t_0(\sigma)\le t\le 2+(R-\sigma)/2,\, \sigma\le R\},
\end{align*}
so that we have
$\{(t,\sigma);\, t\ge t_0(\sigma), \sigma\le R\}=D_1\cup D_2$.
If $(t,\sigma)\in D_1$, then 
observing that we have
$$
%\breve
{\widetilde{w}}_{k}^{*}\bigl( t-(R-\sigma)/2, t+(R+\sigma)/2, \omega\bigr)=\widetilde{W}_k(R,\omega)=0,\quad 1\le k\le N''
$$
(cf.~\eqref{SupportCond03} and \eqref{OuterVanish}),
and remembering \eqref{Eruruw}, we get
\begin{align*}
& {\widetilde{w}}_{k}^{*}(t, t+\sigma, \omega)-\ve\widetilde{W}_k(\sigma, \omega)\\
& \ = \int_{t-(R-\sigma)/2}^t 
            \pa_\tau\left\{
                 {\widetilde{w}}_{k}^{*}(\tau,2t+\sigma-\tau, \omega)
                 -\ve\widetilde{W}_k (2t+\sigma-2\tau, \omega)
            \right\}d\tau    \\
& \ = \int_{t-(R-\sigma)/2}^t 
            \left\{
                   {\widetilde{w}}_{k, -}(\tau,2t+\sigma-\tau, \omega)
                 {}-\ve\widetilde{W}_{k,-}(2t+\sigma-2\tau, \omega)
            \right\} d\tau.
\end{align*}
Since we have $(\tau, 2t+\sigma-\tau)\in \Lambda_0$ 
for $t-(R-\sigma)/2\le \tau \le t$ when $(t, \sigma)\in D_1$, 
we obtain from \eqref{DE08} that
\begin{align}
& \bigl|
{\widetilde{w}}_{k}^{*}(t, t+\sigma, \omega)-\ve \widetilde{W}_k(\sigma, \omega)\bigr| 
\nonumber\\
&\quad \le C\ve^2 (1+2t+\sigma)^{2\lambda-1} \int_{t-(R-\sigma)/2}^t(1+|2t+\sigma-2\tau|)^{-\rho}d\tau \nonumber\\
& \quad \le C\ve^2 (1+2t+\sigma)^{2\lambda-1} (1+|\sigma|)^{1-\rho},\quad 
(t,\sigma)\in D_1,\ \omega\in S^2. \label{DE13-0}
\end{align}
Because $D_2$ is a bounded set,
it follows from \eqref{DE02}, \eqref{DE09}, and \eqref{DE11} that
\begin{equation}\label{DE13-1}
\bigl|{\widetilde{w}}_{k}^{*}(t, t+\sigma, \omega)|+|\ve \widetilde{W}_k(\sigma, \omega)\bigr|\le C\ve^2 \le C \ve^2 
(1+2t+\sigma)^{2\lambda-1} (1+|\sigma|)^{1-\rho}
\end{equation}
for $(t,\sigma)\in D_2$ and $\omega\in S^2$.
Now \eqref{DE13} for $\sigma\le R$ follows from \eqref{DE13-0} and \eqref{DE13-1}.

Similarly to \eqref{DE09},
we obtain from \eqref{DE02} and \eqref{DE13} that
\begin{equation}
\label{DE14}
\bigl|\widetilde{W}_k(\sigma, \omega)\bigr|
\le C\ve(1+|\sigma|)^{-\rho},\quad (\sigma, \omega)\in \R\times S^2.
\end{equation}

Finally we define
$$
W_k(\sigma, \omega)=W_k^{0}(\sigma, \omega)+\widetilde{W}_k(\sigma, \omega),\quad (\sigma, \omega)\in \R\times S^2
$$
for $1\le k\le N''$.
Then, from \eqref{Fried01}, \eqref{AsFi01}, and \eqref{DE13}, 
we find that \eqref{Piyoko} and \eqref{Chada} are true.
\eqref{As02} follows from \eqref{DE09} and \eqref{DE14}.

Note that we have \eqref{As02a} by \eqref{As02} and \eqref{SupportCond05}.
Notice also that we have $\pa_\sigma^p W_k\in C^1(\R\times S^2)$ for $1\le k\le N''$ and $p=0,1$.

{\it Step 2: Construction of $V$.}
We define
\begin{align}
v_-^\dagger(t,r,\omega)= & 
\bigl(v_{1,-}^\dagger(t,r,\omega), \ldots, v_{N',-}^\dagger(t,r,\omega)\bigr)^{\rm T}
\nonumber\\
:= & e^{-{\tens{\Theta}}_\ve(t,r-t,\omega)}\pa_-\left(rv(t,r\omega)\right),
\end{align}
where 
\begin{equation}
\label{PhaseDef}
\tens{\Theta}_\ve(t, \sigma, \omega):=
\left(\ve \log t\right)
\tens{A}[W](\sigma ,\omega),
\end{equation}
and $\tens{A}[W](\sigma, \omega)$ is given by \eqref{Weiss}-\eqref{AssB} with
$W=W(\sigma, \omega)=\left(W_k(\sigma, \omega)\right)_{1\le k\le N''}^{\rm T}$
just having been constructed.
Since 
$\pa_+\left(\tens{\Theta}_\ve(t,r-t,\omega)\right)
\left(=\ve t^{-1}\tens{A}[W](r-t,\omega)\right)$ commutes with $\tens{\Theta}_\ve(t,r-t, \omega)$,
we have
\begin{equation}
\pa_+\bigl(e^{-\tens{\Theta}_\ve(t,r-t, \omega)}\bigr)
=-\frac{\ve}{t}e^{-\tens{\Theta}_\ve(t,r-t,\omega)}\tens{A}[W](r-t, \omega).
\end{equation}
Hence we get
\begin{align*}
\pa_+ 
v_-^\dagger
(t, r, \omega)=e^{-\tens{\Theta}_\ve(t, r-t, \omega)}
\left\{\pa_+\pa_-\bigl(rv(t, r\omega)\bigr)-\frac{\ve}{t}\tens{A}[W](r-t, \omega) \pa_-\bigl( r v (t, r\omega)\bigr) \right\},
\end{align*}
and from \eqref{RaCoWave} and \eqref{AsympBasicEq} we obtain
\begin{equation}
\label{ODEv}
\pa_+ 
v_{j,-}^\dagger
(t,r,\omega)=\Psi_j(t,r,\omega),\quad 1\le j\le N',
\end{equation}
where
\begin{align*}
{\Psi}(t,r,\omega) =&\left(
  \Psi_1(t,r,\omega), \ldots, \Psi_{N'}(t,r,\omega)\right)^{\rm T}\\
 =& \, e^{-\tens{\Theta}_\ve(t, r-t, \omega)}
  \left({H}^\#(t, r, \omega) 
{}+\ve r^{-1} (\Delta_\omega 
v^{0})(t,r\omega)\right),\\
{H}^\#(t, r, \omega)=& \bigl(
H_1^\#(t,r,\omega), \ldots, H_{N'}^\#(t,r,\omega)\bigr)^{\rm T}\\
=& 
\left(H_1(t,r,\omega), \ldots, H_{N'}(t,r,\omega)\right)^{\rm T}
{}-\frac{\ve}{t}\tens{A}[W](r-t, \omega) \pa_-\bigl( r v (t, r\omega)\bigr) .
\end{align*}
Here 
${H}_j(t,r,\omega)$ for $1\le j\le N'$ is given by \eqref{Remainder}.

By \eqref{Weiss}, \eqref{AssB}, and \eqref{As02a} we obtain
\begin{equation}
\label{PhaseEst00}
\|\tens{A}[W](\sigma, \omega)\|\le C \left(|W(\sigma, \omega)|+|\pa_\sigma W(\sigma, \omega)|\right)\le C (1+|\sigma|)^{-\rho},
\end{equation}
where $\|\tens{B}\|$ denotes the operator norm for a matrix $\tens{B}$.
Hence we get $\|\tens{\Theta}_\ve(t,\sigma,\omega)\|\le C\ve \log t$,
which leads to
\begin{equation}
\label{PhaseEst01}
\bigl\|e^{\pm\tens{\Theta}_\ve(t,\sigma,\omega)}\bigr\|
\le e^{\|\tens{\Theta}_\ve(t,\sigma, \omega)\|}\le e^{C\ve\log t}\le t^{C\ve},
\end{equation}
and
\begin{equation}
\label{PhaseEst02}
\bigl\|e^{-\tens{\Theta}_\ve(t,\sigma,\omega)}-\tens{I}\bigr\|
\le e^{\|\tens{\Theta}_\ve(t,\sigma, \omega)\|}-1\le e^{C\ve\log t}-1
\le C\ve t^{C\ve+\lambda}
\end{equation}
for $t>0$ and $(\sigma, \omega)\in \R\times S^2$.
Here, in order to obtain 
\eqref{PhaseEst02}, we have used the inequality
$e^\tau-1\le \tau e^\tau$ that is valid for $\tau\ge 0$,
and the inequality $\log t\le C_\lambda t^{\lambda}$ that is valid for $t\ge 1$
with some positive constant $C_\lambda$ depending only on $\lambda>0$. 

From \eqref{FreeEst} we get
\begin{equation}
\label{DE21}
r^{-1} |(\Delta_\omega
v^{0})(t,r\omega)|
\le C (1+t+r)^{-2}(1+|t-r|)^{-1}.
\end{equation}
Since $t^{-1}-r^{-1}=(r-t)(rt)^{-1}$, we obtain from \eqref{DE01} and \eqref{PhaseEst00} that
\begin{equation}
\label{DE21a}
\left|\left(\frac{\ve}{t}-\frac{\ve}{r}\right)\tens{A}[W](r-t,\omega)\pa_-\bigl(rv(t,r\omega)\bigr)\right|\le C\ve^2(1+t+r)^{\lambda-2}(1+|t-r|)^{-\rho}.
\end{equation}
Note that we can replace $t^{-1}$ by $(1+t+r)^{-1}$ 
to derive of \eqref{DE21a} because we have $A[W](r-t,\omega)=0$ for $r\ge t+R$.
We define
\begin{align*}
G^\#(t,r,\omega)=& \left(G_1^\#(t,r,\omega), \ldots, G_{N'}^\#(t,r,\omega)\right)^{\rm T}\\
:=& 
r\left(G_j\bigl(\omega, u(t, r\omega), \pa u(t, r\omega)\bigr)\right)^{\rm T}_{1\le j\le N'}
{}-\frac{\ve}{r}\tens{A}[W](r-t, \omega) \pa_-\bigl(r v (t, r\omega) \bigr),
\end{align*}
where each $G_j$ is given by \eqref{AssA}.
By \eqref{Piyoko}, \eqref{As02a}, \eqref{blue03'}, and \eqref{DE01}, 
we get
\begin{align}
& \left|r w_l(t, r\omega)(\pa_a v_k)(t, r\omega)
{}-\left(
     -\frac{\ve\omega_a}{2r}
    \right)
  W_l(r-t, \omega) 
      \pa_- \left(
                rv_k(t,r\omega)
             \right)
\right| \nonumber\\
& \qquad \le \left|\left(r w_l(t, r\omega)
 {}-\ve W_l(r-t,\omega)
             \right)(\pa_a v_k)(t, r\omega)\right| \nonumber\\
 & \qquad\quad{}+\left|
     \frac{\ve}{r} W_l(r-t, \omega)\left(r(\pa_a v_k)(t, r\omega)-\left(-\frac{\omega_a}{2}\right) \pa_-\bigl( r v_k (t,r\omega)\bigr)\right)
    \right| \nonumber\\
& \qquad \le C\ve^2 (1+t+r)^{3\lambda-2}(1+|t-r|)^{-\rho} \nonumber
\end{align}
for $1\le k\le N'$ and $1\le l \le N''$.
Similarly, using \eqref{Chada} instead of \eqref{Piyoko}, we obtain
\begin{align}
&\left|r (\pa_b w_l)(t, r\omega)(\pa_a v_k)(t, r\omega)
{}-\left(
     -\frac{\ve \omega_a\omega_b}{2r}
    \right)
  (\pa_\sigma W_l)(r-t, \omega) 
      \pa_- \left(
                rv_k (t,r\omega)
             \right)
\right| \nonumber\\
& \qquad \le C\ve^2 (1+t+r)^{3\lambda-2}(1+|t-r|)^{-1-\rho}. \nonumber
\end{align}
Thus, recalling \eqref{AssA} and \eqref{AssB}, we get
$$
\left| G^\#(t, r, \omega) \right| \le C\ve^2 (1+t+r)^{3\lambda-2}(1+|t-r|)^{-\rho},
$$
which, together with \eqref{DE03},
\eqref{DE05}, and \eqref{DE21a}, yields
\begin{equation}
\label{EarStar}
\left|H_j^\#(t, r, \omega) \right|\le C\ve^2 (1+t+r)^{3\lambda-2},\quad 1\le j\le N'.
\end{equation}
Hence by \eqref{PhaseEst01} and \eqref{DE21}, we get
\begin{equation}
\label{SixThreeFour}
\left|\Psi_j(t, r, \omega)\right|\le C\ve (1+t+r)^{3\lambda+C\ve-2},\quad 1\le j\le N'.
\end{equation}
Now, we define $V_{j,-}=V_{j,-}(\sigma, \omega)$ by
\begin{align}
\label{DefPhi}
\ve V_{j,-}(\sigma, \omega)=& 
v_{j,-}^\dagger
\left(t_1(\sigma), r_1(\sigma), \omega\right)
{}+\int_{t_1(\sigma)}^\infty \Psi_j(\tau, \tau+\sigma ,\omega) d\tau
\end{align}
for $1\le j\le N'$.
Then, from Lemma~\ref{Asymptotics}, \eqref{ODEv}, and \eqref{SixThreeFour}, we obtain
\begin{equation}
\label{SixThirtyThree}
\bigl| 
 v_{j,-}^\dagger 
(t, t+\sigma, \omega)-\ve V_{j,-}(\sigma, \omega)
\bigr|\le C\ve (1+2t+\sigma)^{3\lambda+C\ve-1}
\end{equation}
for $1\le j\le N'$, $t\ge t_0(\sigma)$, and $(\sigma, \omega)\in \R\times S^2$.
%%%%%%%%%%%%%%%%%%%%%%%%%%%%%%%%%%%%%%%%%%%%%%%%%%%%%%%%%%%%%%%%%%%%
From \eqref{SupportCond03} and \eqref{DefPhi} we get
\begin{equation}
\label{OuterVanish00}
V_{j,-}(\sigma, \omega)=0, \quad \sigma\ge R,\ \omega\in S^2.
\end{equation}
Hence
\begin{equation}
 \label{DefV}
 V_j(\sigma, \omega)=\frac{1}{2}\int_\sigma^\infty V_{j,-}(\tau, \omega) d\tau, \quad 1\le j\le N'
\end{equation}
is well-defined. We put $V=(V_1,\dots, V_{N'})^{\rm T}$.
Since $-2 \pa_\sigma V_j=V_{j, -}$, \eqref{PhaseEst01} and \eqref{SixThirtyThree} lead to
\begin{align}
& |\pa_-\left(rv(t, r\omega)\right)-(-2\ve)e^{\tens{\Theta}_\ve(t,r-t,\omega)}(\pa_\sigma V)(r-t, \omega)|\nonumber\\
& \quad \le \|e^{\tens{\Theta}_\ve(t,r-t,\omega)}\|\,|e^{-\tens{\Theta}_\ve(t,r-t,\omega)}\pa_-\left(rv(t, r\omega)\right)-(-2\ve)(\pa_\sigma V)(r-t, \omega)|
\nonumber\\ 
& \quad \le C\ve (1+t+r)^{3\lambda+C\ve-1}.
\label{TabiMexe}
\end{align}
Finally, observing that \eqref{blue03'} and \eqref{DE01} yield 
$$
\left|r(\pa_a v)(t, r\omega)-\left(-\frac{\omega_a}{2}\right)\pa_-\bigl(r v(t, r\omega)\bigr)\right|
\le C\ve (1+t+r)^{\lambda-1},
$$
we obtain \eqref{Ruh} from \eqref{TabiMexe}.

%%%%%%%%%%%%%%%%%%%%%%%%%%%%%%%%%%%%%%%%%%%%%%%%%%%%%%%%%%%%%%%%%%%%
Now it remains to prove
\eqref{As01}.
We set $v^{0,*}=(v_{j}^{0, *})_{1\le j\le N'}^{\rm T}$,
where ${v}_{j}^{0,*}(t,r,\omega)=r v_j^{0} (t, r\omega)$.
From \eqref{FreeEst} and \eqref{Fried01} we find that
\begin{equation}
\label{Prim01}
\left|(\pa_-{v}_{j}^{0, *})(t, t+\sigma, \omega)-(-2)(\pa_\sigma V_j^{0})(\sigma, \omega)
\right|\le C(1+2t+\sigma)^{-1}(1+|\sigma|)^{-1}
\end{equation}
for $t\ge t_0(\sigma)$ and $(\sigma,\omega)\in \R\times S^2$.
If we put
\begin{align}
V_{j,-}^{0}(\sigma, \omega):=& (\pa_- {v}_{j}^{0, *})\left(t_1(\sigma), r_1(\sigma), \omega\right)
\nonumber\\
&{}+\int_{t_1(\sigma)}^\infty 
\left. \left(r^{-1} (\Delta_\omega v_j^{0})
(\tau, r\omega)\right)\right|_{r=\tau+\sigma} d\tau,
\label{Karura01}
\end{align}
then it follows from Lemma~\ref{Asymptotics}, \eqref{RaCoWave} with $\psi=v_j^{0}$, and \eqref{DE21} that
\begin{equation}
\label{Prim02}
|(\pa_- v_{j}^{0, *})(t,t+\sigma,\omega)-V_{j,-}^{0}(\sigma, \omega)|\le C (1+2t+\sigma)^{-1}(1+|\sigma|)^{-1}
\end{equation}
for $t\ge t_0(\sigma)$ and $(\sigma, \omega)\in \R\times S^2$.
Now from \eqref{Prim01} and \eqref{Prim02} we get
$$
\left|-2(\pa_\sigma V_j^{0})(\sigma,\omega)-V_{j,-}^{0}(\sigma, \omega)\right|
\le C\lim_{t\to\infty}(1+2t+\sigma)^{-1}(1+|\sigma|)^{-1}=0,
$$
which shows
\begin{equation}
\label{Karura02}
-2(\pa_\sigma V_j^{0})(\sigma,\omega)=V_{j,-}^{0}(\sigma, \omega),\quad
(\sigma, \omega)\in \R\times S^2.
\end{equation}
%%%%%%%%%%%%%%%%%%%%%%%%%%%%%%%%%%%%%%%%%%%%%%%%%%%%%%%%%%%%%%%%%%%%

We put ${\widetilde{v}}_{j}^{\,*}(t,r,\omega)=r\widetilde{v}_j(t, r\omega)$, and
$\widetilde{v}^{\,*}=(\widetilde{v}_{j}^{\,*})_{1\le j\le N'}^{\rm T}$.
By \eqref{DefPhi}, \eqref{DefV}, \eqref{Karura01}, and \eqref{Karura02}, we obtain
\begin{align*}
&2\left((\pa_\sigma V)(\sigma, \omega)-(\pa_\sigma V^{0})(\sigma, \omega)\right)
= -\left(V_{j,-}(\sigma,\omega)-V_{j,-}^0(\sigma,\omega)\right)_{1\le j\le N'}^{\rm T}
\\
&= -\left(e^{-\tens{\Theta}_\ve(t_1(\sigma), \sigma, \omega)}-\tens{I}\right)
(\pa_- {v}^{0,*})\left(t_1(\sigma), r_1(\sigma), \omega\right)\\
&\quad {}-\ve^{-1}e^{-\tens{\Theta}_\ve (t_1(\sigma), \sigma, \omega)}
 (\pa_- {\widetilde{v}}^{\, *})\left(t_1(\sigma), r_1(\sigma), \omega \right)\\
&\quad {}-\ve^{-1}\int_{t_1(\sigma)}^\infty 
e^{-\tens{\Theta}_\ve(\tau, \sigma, \omega)} H^\#(\tau, \tau+\sigma, \omega) 
d\tau\\
&\quad {}-\int_{t_1(\sigma)}^\infty 
\left. \left(r^{-1} \left(e^{-\tens{\Theta}_\ve(\tau, \sigma, \omega)}-\tens{I}\right)
(\Delta_\omega
v^{0})(\tau, r \omega)\right)\right|_{r=\tau+\sigma} d\tau.
\end{align*}
Now using
\eqref{FreeEst}, \eqref{Aruruw},
\eqref{DE01},   
\eqref{PhaseEst01}, \eqref{PhaseEst02}, \eqref{DE21},
and \eqref{EarStar},
we get
\begin{align*}
\left|\pa_\sigma V(\sigma, \omega)-\pa_\sigma V^{0}(\sigma, \omega)\right|
\le & C\ve (1+t_1(\sigma)+r_1(\sigma))^{\lambda+C\ve}(1+|\sigma|)^{-1}\\
&{}+C\ve\int_{t_1(\sigma)}^\infty (1+2\tau+\sigma)^{3\lambda+C\ve-2}d\tau
\\
\le & 
C\ve (1+|\sigma|)^{3\lambda+C\ve-1},
\end{align*}
which leads to \eqref{As01}. 
This completes the proof. \qed
%%%%%%%%%%%%%%%%%%%%%
%%%%%%%%%%%%%%%%%%%%%%%%%%%%%%%%%%%%%%%%%%
\section{Asymptotic Behavior in the Energy Sense}
\label{EneBehavior}
%%%%%%%%%%%%%%%%%%%%%%%%%%%%%%%%%%%%
The purpose of this section is to prove 
Corollary~\ref{EnergyAsymptotics}.
We also revisit three examples stated in Section~\ref{SAPB},
and we will prove \eqref{EneGrow01}, \eqref{EneGrow02},
and \eqref{EneGrow03}.

To begin with, we observe that \eqref{Ruh}, \eqref{Piyoko},
and \eqref{Chada} in Theorem~\ref{PointwiseAsymptotics} are actually valid
for any $(t, r)\in [2,\infty)\times (0,\infty)$ and $\omega\in S^2$.
Indeed what is left to prove is their validity for $(t,r)\in \Lambda_1$
and $\omega\in S^2$, where $\Lambda_1=\{(t,r)\in [2,\infty)\times (0,\infty); r<t/2\}$. 
We obtain from
\eqref{DE01} that
\begin{equation}
\sup_{\omega\in S^2} |r\pa_a v(t,r\omega)| \le C\ve\jb{t+r}^{\lambda-1}
\label{Tea01}
\end{equation}
for $(t,r)\in \Lambda_1$,
because we have $\jb{t+r}\le C\jb{t-r}$ for $(t,r)\in \Lambda_1$.
Similarly \eqref{As01a} and \eqref{PhaseEst01} imply
\begin{equation}
\sup_{\omega\in S^2} \left|\ve 
e^{(\ve\log t){\tens A}[W](r-t,\omega)}(\pa_\sigma V)(r-t,\omega)\right|
\le C\ve 
\jb{t+r}^{3\lambda+C\ve-1}
\label{Tea02}
\end{equation}
for $(t,r)\in \Lambda_1$.
Equations \eqref{Tea01} and \eqref{Tea02} show that \eqref{Ruh} holds
also for $(t,r)\in \Lambda_1$ and $\omega\in S^2$.
In the same fashion, it follows from \eqref{DE02} and \eqref{As02a}
that
\begin{align}
& \sup_{\omega\in S^2}
\sum_{l=0}^1 \jb{t-r}^l
\left(\sum_{|\alpha|=l} |r (\pa^\alpha w)(t,r\omega)|
+|\ve(\pa_\sigma^l W)(r-t,\omega)|\right)
\nonumber\\
& \qquad\qquad\qquad\qquad \le
C\ve \jb{t+r}^{-\rho}\le C\ve \jb{t+r}^{-1}\jb{t-r}^{1-\rho},
\quad (t,r)\in \Lambda_1,
\label{Tea03}
\end{align}
which shows that \eqref{Piyoko} and \eqref{Chada} are true also for $(t,r)\in \Lambda_1$ and $\omega\in S^2$. 

The following is an immediate consequence of 
this observation.
%%%%%%%%%%%%%%%%%%%%%%%%%%%%
\begin{corollary}\label{GenEnergyEst}
Assume that the assumptions of Theorem~$\ref{PointwiseAsymptotics}$ 
are fulfilled, and
$(f,g)\in D$. Let $v$, $w$, $V$, $W$, 
and $\lambda$ be as in Theorem~$\ref{PointwiseAsymptotics}$.
We define $\tens{\Theta}_\ve^+$ as in Corollary~$\ref{EnergyAsymptotics}$.
We also set
\begin{align}
\label{DefExtend01}
{V}^{a}_{\sigma,*}(t,x):=& \left. \left(\omega_a r^{-1}(\pa_\sigma V)(r-t, \omega) \right) \right|_{r=|x|, \omega=x/|x|},\\
W^{a}_{\sigma,*}(t,x):=& \left. \left( \omega_a r^{-1}(\pa_\sigma W)(r-t, \omega)\right) \right|_{r=|x|, \omega=x/|x|}
\end{align}
for $0\le a\le 3$.
Then there exists a positive constant $C$ such that we have
\begin{align}
& \left(\sum_{a=0}^3 \|\pa_a v(t,\cdot)-\ve
e^{\tens{\Theta}_\ve^+(t,\cdot)}V^{a}_{\sigma, *}(t,\cdot) \|_{L^2(\R^3)}^2\right)^{1/2}\le C\ve 
(1+t)^{3\lambda+C\ve-(1/2)},
\label{GEE03}
\\
\label{GEE04}
& \left(\sum_{a=0}^3 \|\pa_a w(t,\cdot)-\ve W^{a}_{\sigma, *}(t,\cdot)\|_{L^2(\R^3)}^2\right)^{1/2}\le  C\ve (1+t)^{2\lambda-1}
\end{align}
for $t\ge 2$ and sufficiently small $\ve(>0)$.
\end{corollary}
%%%%%%%%%%%%%%%%%%%%%%%%%%%%
%%%%%%%%%%%%%%%%%%%%%%%%%%%%
\begin{proof}
Let $t\ge 2$, and let $\ve(>0)$ be sufficiently small.
Switching to the polar coordinates, we find from \eqref{Ruh}
for $(t,r)\in [2,\infty)\times (0,\infty)$ and $\omega\in S^2$
that
the left-hand side of \eqref{GEE03} is bounded by
$$
C\ve \left(\int_0^\infty (1+t+r)^{6\lambda+2C\ve-2} dr\right)^{1/2}
\le C\ve (1+t)^{3\lambda+C\ve-(1/2)},
$$
which is the desired bound.
Similarly \eqref{GEE04} follows from \eqref{Chada} for $(t,r)\in [2,\infty)\times (0,\infty)$ and $\omega\in S^2$, because we have
$$
C\ve \left(\int_0^\infty (1+t+r)^{4\lambda-2}(1+|t-r|)^{-2\rho} dr\right)^{1/2}
\le C\ve (1+t)^{2\lambda-1}
$$
for $\rho>1/2$.
\end{proof}
%%%%%%%%%%%%%%%%%%%%%%%%%%%%%%%%%%%%%%%%
%%%%%%%%%%%%%%%%%%%%%%%%%%%%%%%%%%%%%%%%%

Corollary~\ref{EnergyAsymptotics} follows from Lemma~\ref{Trans01} and Corollary~\ref{GenEnergyEst}.
\begin{proof}[Proof of Corollary~\ref{EnergyAsymptotics}]
Suppose that all the assumptions of Theorem~\ref{PointwiseAsymptotics} are
fulfilled.
Note that \eqref{As01a} and \eqref{As02a} imply
$\pa_\sigma V_j, \pa_\sigma W_k \in L^2(\R\times S^2)$
for small $\ve>0$.
Remember the definition of the isometric isomorphism ${\mathcal T}:H_0\to L^2(\R\times S^2)$ 
given in Section~\ref{translation}.
We set $(f_j^{+}, g_j^{+})={\mathcal T}^{-1}[\pa_\sigma V_j](\in H_0)$
for $1\le j\le N'$, and 
$$
v^+=\left(\ve \, {\mathcal U}_0[f_1^{+}, g^{+}_1], \ldots, \ve \, {\mathcal U}_0[f_{N'}^{+}, g^{+}_{N'}]\right)^{\rm T}.
$$
Let $V^{a}_{\sigma, *}$ 
be as in Corollary~\ref{GenEnergyEst}.
Then, since we have
$$
\ve V^{a}_{\sigma, *}(t, r\omega)=
\bigl(\ve \omega_a r^{-1} {\mathcal T}[f_j^+, g_j^+](r-t, \omega)
\bigr)_{1\le j\le N'}^{\rm T},
$$
Lemma~\ref{Trans01} implies 
\begin{align*}
\lim_{t\to\infty} 
\left(
\sum_{a=0}^3 \|\pa_a v^+(t,\cdot)-\ve V^{a}_{\sigma, *}(t,\cdot)\|_{L^2(\R^3)}^2
\right)^{1/2}
=& 0.
\end{align*}
Now \eqref{PiyokoE} follows from \eqref{GEE03} in Corollary~\ref{GenEnergyEst}
with the help of \eqref{PhaseEst01}.

By setting $(f_{N'+k}^{+}, g_{N'+k}^{+})={\mathcal T}^{-1}[\pa_\sigma W_k]$ for $1\le k\le N''$, 
and using \eqref{GEE04},
we can show \eqref{ChadaE} similarly. 
\end{proof}

Another consequence of Lemma~\ref{Trans01} and Corollary~\ref{GenEnergyEst}
is the following.
%%%%%%%%%%%%%%%%%%%%%%%%%%%%%%%%%%%%%
\begin{lemma}\label{EnergyPointwise}
Suppose that all the assumptions of Theorem~$\ref{PointwiseAsymptotics}$ are fulfilled.
Let $v$, $V$, and $W$ be as in Theorem~$\ref{PointwiseAsymptotics}$.
Assume that $\ve$ is positive and small enough.

If $v$ is asymptotically free in the energy sense, then
there exists a function 
${T}^+={T}^+(\sigma,\omega)\in L^2(\R\times S^2; \R^{N'})$ such that we have
$$
\lim_{t\to \infty} e^{(\ve\log t)\tens{A}[W]}
\pa_\sigma V={T}^+
\quad \text{in $L^2(\R\times S^2; \R^{N'})$.}
$$
\end{lemma}
%%%%%%%%%%%%%%%%%%%%%%%%%%%%%%%%%%%%%%%%%%%%%%%%%%%%
\begin{proof}
 Choose $(f_j^{+}, g_j^{+})\in \dot{H}^1(\R^3)\times L^2(\R^3)$
 for $1\le j\le N'$ satisfying
\begin{equation}
\lim_{t\to\infty}
\|v_j(t,\cdot)-{\mathcal U}_0[\ve f_j^{+}, \ve g_j^{+}](t,\cdot)\|_E=0.
\label{Mimi}
\end{equation}
We set ${T}^+_j:={\mathcal T}[f_j^+, g_j^+]$ for $1\le j\le N'$, and
${T}^+:=({T}_1^+,\ldots, {T}_{N'}^+)^{\rm T}$.
If we put 
$$
{T}^{+,a}_*(t,x):=\left. \left(\omega_a r^{-1}{T}^+(r-t,\omega)\right)\right|_{r=|x|, \omega=x/|x|},
$$ 
then \eqref{Mimi} and Lemma~\ref{Trans01} imply
\begin{equation}
\label{Chime01}
\lim_{t\to\infty}
\sum_{a=0}^3 \|\pa_a v(t,\cdot)-\ve {T}^{+,a}_*(t,\cdot)\|_{L^2(\R^3)}^2
=0.
\end{equation}
Let $V^{a}_{\sigma, *}$ and $\tens{\Theta}_\ve^+$ 
be as in Corollary~\ref{GenEnergyEst}. Then \eqref{GEE03} and
\eqref{Chime01} yield
$$
\lim_{t\to\infty}
%\left(
\sum_{a=0}^3 \|e^{\tens{\Theta}_\ve^+(t,\cdot)}V^{a}_{\sigma, *}(t,\cdot)-{T}^{+,a}_*(t,\cdot)\|_{L^2(\R^3)}^2
%\right)^{1/2}
=0,
$$
which leads to
$$
%\begin{equation}
%\label{Chime02}
\lim_{t\to\infty}
\int_{\omega\in S^2}
\left(\int_{-t}^\infty 
|e^{(\ve \log t){\tens A}[W](\sigma, \omega)}
%e^{\tens{\Theta}_\ve(t, \sigma, \omega)}
(\pa_\sigma V)(\sigma, \omega)
{}-T^+(\sigma, \omega)|^2d\sigma \right) dS_\omega
=0.
%\end{equation}
$$
By \eqref{As01a} and \eqref{PhaseEst01} we get
$$
\bigl\|e^{(\ve\log t){\tens A}[W]}(\pa_\sigma V)\|_{L^2((-\infty, -t)\times S^2)}
\le C(1+t)^{3\lambda+C\ve-(1/2)}\to 0,\quad t\to\infty.
$$
Since $T^+\in L^2(\R\times S^2)$, it is easy to see that
$\|T^+\|_{L^2((-\infty, -t)\times S^2)}$ converges to $0$ as $t\to\infty$.
Now we obtain the desired result immediately.
\end{proof}
%%%%%%%%%%%%%%%%%%%%%%%%%%%%%%%%%%

Next we will observe how freely we can choose the values of the radiation fields.
Let $\psi\in C^\infty_0(\R^3)$ be radially symmetric, namely there exists
a function $\psi^{\rm R}$ such that
$\psi(x)=\psi^{\rm R}(|x|)$ for $x\in \R^3$.
Then, the Radon transform of $\psi$ can be written as
\begin{align}
\label{RadConst01}
{\mathcal R}[\psi](\sigma, \omega)
=\int_{y'\in \R^2} \psi^{\rm R}\left(\sqrt{|y'|^2+\sigma^2}\right) dy'
=
2\pi \int_{|\sigma|}^\infty s \psi^{\rm R}(s) ds.
\end{align}
We also get
\begin{equation}
\label{RadConst02}
\pa_\sigma {\mathcal R}[\psi](\sigma, \omega)=-2\pi\sigma \psi^{\rm R}(|\sigma|).
\end{equation}
\begin{lemma}\label{RadConst}
Given $\sigma_0\ne 0$, 
$\alpha\in \R$, and $\beta\in \R$,
there exists a pair of functions $\left(\varphi, \psi\right)\in C^\infty_0(\R^3)\times C^\infty_0(\R^3)$
such that we have
\begin{equation}
{\mathcal F}_0\left[\varphi, \psi\right](\sigma_0, \omega)=\alpha,
\ \left(\pa_\sigma {\mathcal F}_0\left[\varphi, \psi\right]\right)(\sigma_0, \omega)=\beta 
\label{SevenEleven}
\end{equation}
for any $\omega\in S^2$. 
\end{lemma}
\begin{proof}
We can choose a smooth and even function $\zeta=\zeta(s)$,
vanishing in some neighborhood of $s=0$ and vanishing also for large $|s|$,  
such that
$\zeta(|\sigma_0|)=2 \alpha$ and $\zeta'(|\sigma_0|)=2(\sgn \sigma_0) \beta$,
where $\zeta'(s)=(d\zeta/ds)(s)$.
For $x\in \R^3$, we define 
$\varphi(x)=0$  and  $\psi(x)=-|x|^{-1} \zeta'(|x|)$.
Then we have $\psi\in C^\infty_0(\R^3)$, and we can easily obtain \eqref{SevenEleven} from \eqref{RadConst01} and \eqref{RadConst02}.
\end{proof}
%%%%%%%%%%%%%%%%%%%%%%%%%%%%%%%%%%%

Let us recall the systems \eqref{simplestEx}, \eqref{LogEx}, and \eqref{RotEx} in
Examples~\ref{FirstE}, \ref{SecondE}, and \ref{ThirdE}.
We are going to show \eqref{EneGrow01}, \eqref{EneGrow02}, and \eqref{EneGrow03}.
For simplicity we only consider the initial condition
of the form \eqref{Data}, instead of \eqref{Data0}.
More precisely we prove the following.
%%%%%%%%%%%%%%%%%%%%%%%%%%%%%%%%%%%%%%%%%%%%%
\begin{proposition}\label{EneGrow11}
We put $\widetilde{\mathcal X}_N=C^\infty_0(\R^3;\R^N)\times C_0^\infty(\R^3;\R^N)$.
\smallskip\\
{$(1)$} There exist  
$(f, g)\in \widetilde{\mathcal X}_2$ and three positive constants $\ve_1$, $C_1$, and $C_2$
such that
\begin{equation}
C_1\ve (1+t)^{C_1\ve} \le \|u(t)\|_E\le C_2\ve (1+t)^{C_2\ve},\quad
t\ge 0, \ 0<\ve\le \ve_1,
 \label{EG111}
\end{equation}
where $u=(u_1, u_2)^{\rm T}$ 
is the global solution to \eqref{simplestEx} with \eqref{Data}.
\smallskip\\
%%%%%%%%%%%%%%%%%%% 
{$(2)$}
There exist
$(f,g)\in \widetilde{\mathcal X}_3$ and 
three positive constants $\ve_1$, $C_1$, and $C_2$ such that
\begin{equation}
C_1\bigl(\ve+\ve^2\log(1+t)\bigr)\le \|u(t)\|_E\le C_2\bigl(\ve+\ve^2 \log(1+t)\bigr),\quad 
t\ge 0, \ 0<\ve \le \ve_1,
\end{equation}
where $u=(v^{\rm T}, w)^{\rm T}=\bigl((v_1, v_2), w\bigr)^{\rm T}$ 
is the global solution to \eqref{LogEx} with \eqref{Data}.
\smallskip\\
%%%%%%%%%%%%%%%%%%%
{$(3)$}
For any $(f,g)\in \widetilde{\mathcal X}_3$ with $(f,g)\not\equiv (0,0)$, 
there exist three positive constants $C_1$, $C_2$, and $\ve_1$ such that
\begin{equation}
C_1\ve \le \|u(t)\|_E\le C_2\ve, \quad t\ge 0, \ 0<\ve \le \ve_1,
\end{equation}
where $u=(v^{\rm T}, w)^{\rm T}=\bigl((v_1, v_2), w\bigr)^{\rm T}$ 
is the global solution to \eqref{RotEx} with \eqref{Data}.
However there exists some $(f,g)\in \widetilde{\mathcal X}_3$ such that $u$ is not asymptotically free in the energy sense
for $0<\ve\le \ve_1$.
\end{proposition}
\begin{proof}
In what follows, $V$ and $W$ are the modified and the standard asymptotic profiles for the system 
under consideration in each assertion.
$V^{a}_{\sigma, *}$, $W^{a}_{\sigma, *}$ and $\tens{\Theta}_\ve^+$ 
are defined as in Corollary~\ref{GenEnergyEst}.

First we prove (1). We put $w=\pa_1u_2-\pa_2 u_1$. 
Then, as we have shown, $(u_1, u_2, w)^{\rm T}$ satisfies \eqref{simplestExR}
with \eqref{Data1} and \eqref{Data2}.
By \eqref{Chato-t}, we see that for any $(f,g)\not\equiv(0,0)$ and $T>0$, 
there exist two positive constants ${C}_T$ and $\widetilde{\ve}_T$, depending on $f$, $g$, and $T$, such that
\begin{equation}
{C}_T^{-1}\ve(1+T)\le \|u(t)\|_E\le {C}_T\ve,\quad 0\le t\le T,\ 0<\ve\le \widetilde{\ve}_T.
\label{SmallTimeEst}
\end{equation}
We are going to prove the following: If we choose appropriate $(f,g)\in \widetilde{\mathcal X}_2$, then
there exist three positive constants $C_1'$, $C_2'$, and $\ve_2$ such that 
\begin{equation}
\label{EneAim01}
C_1' (1+t)^{C_1'\ve}\le 
  \left(
   \sum_{a=0}^3\left(
        \bigl\|
            e^{\tens{\Theta}_\ve^+(t)}V^{a}_{\sigma, *}(t)\bigr\|_{L^2}^2
         {}+\bigl\|
            W^{a}_{\sigma, *}(t)\bigr\|_{L^2}^2\right)
   \right)^{1/2}
\le C_2' (1+t)^{C_2'\ve}
\end{equation}
for $t\ge 2$ and $0< \ve \le \ve_2$.
Once \eqref{EneAim01} is established, we obtain the desired result from
\eqref{GEE03}, \eqref{GEE04}, and \eqref{SmallTimeEst}. 
Indeed, if we choose sufficiently large $T(>0)$, then \eqref{GEE03}, \eqref{GEE04}, and \eqref{EneAim01} yield the desired estimate
for $t\ge T$,
while \eqref{SmallTimeEst} implies the desired estimate for $0\le t\le T$, provided that $\ve$ is sufficiently small.

We start the proof of \eqref{EneAim01}.
We take
$$
{\mathcal O}=\{\omega=(\omega_1, \omega_2, \omega_3)\in S^2;\, \omega_1\le -1/4,\ \omega_2\le -1/4\}.
$$
Remembering \eqref{Example0101}, we get
\begin{align}
\label{EneBehav01}
\sum_{a=0}^3 \|e^{\tens{\Theta}_\ve^+(t)}V^{a}_{\sigma, *}\|_{L^2}^2 
\ge & 2 \int_{\omega \in {\mathcal O}}
\int_{-t}^\infty t^{-\ve \omega_1 W(\sigma,\omega)}
|\pa_\sigma V_1(\sigma, \omega)|^2 d\sigma dS_\omega
\end{align}
for $t\ge 2$.
Fix some $\sigma_0>0$.
From Lemma \ref{RadConst} we can find $(f_j, g_j)\in 
%C^\infty_0(\R^3)\times C_0^\infty(\R^3)
\widetilde{\mathcal X}_2$ ($j=1,2$)
such that
$$
\left(\pa_\sigma \mathcal{F}_0[f_1, g_1]\right)(\sigma_0, \omega)
= -\left(\pa_\sigma \mathcal{F}_0[f_2, g_2]\right)(\sigma_0,\omega)
=8, \quad \omega\in S^2.
$$
From the continuity of the radiation field, we can choose some $\delta>0$ such that
\begin{equation}
\label{Tabi-On}
\left(\pa_\sigma \mathcal{F}_0[f_1, g_1]\right)(\sigma, \omega),
 -\left(\pa_\sigma \mathcal{F}_0[f_2, g_2]\right)(\sigma,\omega)
\ge 4, \quad |\sigma-\sigma_0|<\delta,\ \omega\in S^2.
\end{equation}
Then \eqref{Data2} leads to 
\begin{align*}
\mathcal{F}_0[f_3, g_3](\sigma,\omega)=&\omega_1\left(\pa_\sigma \mathcal{F}_0[f_2, g_2]\right)(\sigma,\omega)
-\omega_2\left(\pa_\sigma \mathcal{F}_0[f_1, g_1]\right)(\sigma,\omega)
\ge 2
\end{align*}
for $|\sigma-\sigma_0|<\delta$ and $\omega\in {\mathcal O}$.
Now, with this choice of $(f,g)$, \eqref{As01} and \eqref{As02} imply
\begin{align}
\label{EneBehav02}
\pa_\sigma V_1(\sigma, \omega)\ge 2,\ W(\sigma, \omega)\ge 1, \quad |\sigma-\sigma_0|<\delta,\ \omega\in {\mathcal O}
\end{align}
for sufficiently small $\ve$. Since we may assume $\sigma_0-\delta>0> -t$,
by \eqref{EneBehav01} and \eqref{EneBehav02} we obtain
$$
\sum_{a=0}^3\|e^{\tens{\Theta}_\ve^+(t)}V^{a}_{\sigma, *}(t)\|_{L^2}^2\ge C (1+t)^{\ve/4}
$$ for $t\ge 2$,
and the first half of \eqref{EneAim01} is proved.

By \eqref{As01a} and \eqref{As02a},
we have $\|\pa_\sigma V\|_{L^2_{\sigma, \omega}}+\|\pa_\sigma W\|_{L^2_{\sigma, \omega}}\le C$  for small $\ve>0$,
where $L_{\sigma, \omega}^2$ denotes $L^2(\R\times S^2)$.
We also have 
$
t^{-\ve \omega_1 W/2}+|\omega_1^{-1}(t^{-\ve \omega_1 W/2}-1)|\le C(1+t)^{C\ve}
$
for $t\ge 2$.
Hence we get 
\begin{align}
\label{EneConcl01}
\left(\sum_{a=0}^3 \|e^{\tens{\Theta}_\ve^+(t)} V^{a}_{\sigma, *}(t)\|_{L^2}^2\right)^{1/2}
\le &  C(1+t)^{C\ve} \|\pa_\sigma V\|_{L^2_{\sigma, \omega}}
\le C (1+t)^{C\ve}, \\
\label{EneConcl02}
\left(\sum_{a=0}^3 \|W^{a}_{\sigma, *}(t)\|_{L^2}^2\right)^{1/2}\le & C \|\pa_\sigma W\|_{L^2_{\sigma, \omega}}\le C,
\end{align}
which imply the last half of \eqref{EneAim01}.

Next we prove (2). 
Since \eqref{EneConcl02} is valid also for this case, our task is to estimate
 $\sum_{a=0}^3\|e^{\tens{\Theta}_\ve^+(t)}V^{a}_{\sigma, *}(t)\|_{L^2}^2$.
Fix some $\sigma_0>0$. 
Similarly to \eqref{Tabi-On},
we can take some $(f,g)\in %C^\infty_0(\R^3; \R^3)\times C^\infty_0(\R^3; \R^3)
\widetilde{\mathcal X}_3$ and
some $\delta>0$ such that
$$
(\pa_\sigma {\mathcal F}_0[f_j, g_j])(\sigma, \omega)\ge 2,
\quad |\sigma-\sigma_0|<\delta,\ \omega \in S^2,\ 1\le j\le 3.
$$ 
From \eqref{As01} and \eqref{As02}, we get
$$
\pa_\sigma V_1(\sigma, \omega)\ge 1,\ \pa_\sigma V_2(\sigma, \omega)\ge 1,\ \pa_\sigma W(\sigma, \omega)\ge 1
$$
for $|\sigma-\sigma_0|<\delta$ and $\omega\in S^2$, provided that $\ve$ is small enough.
Since we may assume $\sigma_0-\delta>0$,
from \eqref{Example0201} we get
\begin{align*}
\sum_{a=0}^3\|e^{\tens{\Theta}_\ve^+(t)}V^{a}_{\sigma, *}\|_{L^2}^2
\ge 8\pi \int_{\sigma_0-\delta}^{\sigma_0+\delta} |1+\ve \log t|^2 d\sigma
\ge C(1+\ve \log(1+t) )^2
\end{align*}
for $t\ge 2$ and small $\ve$.
It follows from \eqref{As01a}, \eqref{As02a}, \eqref{Example0201}, 
and \eqref{Example0202} that
\begin{align*}
\sum_{a=0}^3\|e^{\tens{\Theta}_\ve^+(t)}V^{a}_{\sigma, *}\|_{L^2}^2
\le & C \left(\|\pa_\sigma V\|_{L^2_{\sigma, \omega}}+\bigl(\ve \log t\bigr)
\|\pa_\sigma W\|_{L^\infty(\R\times S^2)}\|\pa_\sigma V_2\|_{L^2_{\sigma, \omega}}
\right)^2\\
\le & C\bigl(1+\ve \log (1+t)\bigr)^2.
\end{align*}
Now, using Corollary~\ref{GenEnergyEst}, we can easily reach at the desired result
as before.

Finally we prove (3).
From \eqref{RotEx} we obtain
$$
\frac{d}{dt}(\|u(t)\|_E^2)=\sum_{j=1}^3 \int_{\R^3} F_j^0(\pa u) \pa_t u_j dx.
$$
It follows from Lemma~\ref{NullNull} (with $s=0$), \eqref{Fukkie},
and \eqref{Chato} that
\begin{align*}
\sum_{j=1}^3 \int_{\R^3} \left|F_j^0(\pa u) \pa_t u_j\right| dx
\le & 
C\left(\sup_{x\in \R^3}\jb{t+|x|}^{-1}|u(t,x)|_1\right)\|\pa u(t)\|_{L^2}^2
\\
\le & C\ve^3(1+t)^{3\lambda-2}.
\end{align*}
Hence we get
$$
\left|\|u(t)\|_E^2-\|u(0)\|_E^2\right|\le C\ve^3\int_0^\infty(1+\tau)^{3\lambda-2}d\tau\le C\ve^3,
$$
which leads to \eqref{EneGrow03} for small $\ve>0$, because we have
$$
 2\|u(0)\|_E^2=\ve^2(\|\nabla_x f\|_{L^2}^2+\|g\|_{L^2}^2).
$$

Now we are going to prove that $u$ is not asymptotically free in the energy sense
for some initial profile. 
We fix some $\sigma_0> 0$. Then, similarly to the above, we can take $(f,g)\in %C^\infty_0(\R^3;\R^3)\times C^\infty_0(\R^3;\R^3)
\widetilde{\mathcal X}_3$ and $\delta>0$ such that
\begin{equation}
\pa_\sigma V(\sigma, \omega)\ne 0,\ 1\le \pa_\sigma W(\sigma, \omega)
\le 2, \quad (\sigma,\omega)\in {\mathcal I}\times S^2
\label{NonAF01}
\end{equation}
for sufficiently small $\ve>0$,
where $\mathcal{I}=(\sigma_0-\delta, \sigma_0+\delta)$. 
Suppose that $u$ is asymptotically free in the energy sense. 
From Lemma~\ref{EnergyPointwise} we see that
$e^{(\ve \log t)\tens{A}[W]}(\pa_\sigma V)$
converges to some function in $L^2(\R\times S^2)$ 
as $t$ goes to $\infty$.
Therefore we get
\begin{equation}
\lim_{t\to \infty}
\|(e^{(\ve\log(et))\tens{A}[W]}-e^{(\ve\log t)\tens{A}[W]})(\pa_\sigma V)%(\cdot,\cdot)
\|_{L^2(\R\times S^2)}=0.
\label{NonAF02}
\end{equation}
On the other hand, we get
\begin{align}
& \bigl\|\bigr(e^{(\ve\log(et))\tens{A}[W]}-e^{(\ve\log t)\tens{A}[W]}\bigr)(\pa_\sigma V)\bigr\|_{L^2(\R\times S^2)}
\nonumber\\
&\qquad\qquad\qquad\qquad
 = \|(e^{\ve\tens{A}[W]}-\tens{I})(\pa_\sigma V)\|_{L^2(\R
\times S^2)}
\label{NonAF03}
\end{align}
for $t\ge 2$,
because we have $|e^{(\ve \log t)\tens{A}[W](\sigma,\omega)}Y|=|Y|$
for any $Y\in \R^{2}$
by \eqref{Example0302}. 
By \eqref{NonAF02} and \eqref{NonAF03} we conclude that
$(e^{\ve\tens{A}[W](\sigma,\omega)}-\tens{I})(\pa_\sigma V)(\sigma,\omega)=0$
for almost every $(\sigma,\omega)\in \R\times S^2$.
This contradicts \eqref{NonAF01} if $\ve(>0)$ is small enough, because $e^{\ve\tens{A}[W](\sigma,\omega)}-\tens{I}$ is invertible 
when $0<\ve (\pa_\sigma W)(\sigma,\omega)<2\pi$.
Hence $u$ is not asymptotically free in the energy sense for small $\ve$.
\end{proof}
%%%%%%%%%%%%%%%%%%%%%%%%%%%%%%%%%%%%%%%%%%%%%
%
%
%
%%%%%%%%%%%%%%%%%%%%%%%%%%%%%%%%%%%%%%%%%%%%%%
\section{Proof of Theorems~\ref{Necessity} and \ref{Necessity02}}
\label{ProofNecessity}
%%%%%%%%%%%%%%%%%%%%%%%%%%%%%%%%%%%%%%%%%%%%%%%%
\subsection{Proof of Theorem~$\ref{Necessity}$}
Before we proceed to the proof of Theorem~\ref{Necessity},
we recall some elementary facts from the linear algebra.

Let $m$ be a positive integer,
and $\tens{I}$ be the $m\times m$ identity matrix. 
Suppose that $\tens{B}$ is an $m\times m$ complex matrix.  
For a complex number $\mu$ and a positive integer $n$, we define
${\mathcal K}_\mu^n(\tens{B})= \ker (\tens{B}-\mu \tens{I})^n=\{y
\in \C^m; (\tens{B}-\mu \tens{I})^n 
y=0
\}$. 
We also define ${\mathcal K}_\mu(\tens{B})=\bigcup_{n=1}^\infty {\mathcal K}_\mu^n(\tens{B})$.
It is known that for any $\mu\in \C$, there exists a positive integer
$l (\le m)$ such that
${\mathcal K}_\mu^l (\tens{B})={\mathcal K}_\mu(\tens{B})$.
Let $\mu_1, \dots, \mu_J$ (with some $J\le m$) be all of the distinct eigenvalues of $\tens{B}$.
Then it is well known that we have
\begin{equation}
\label{Decomp}
 \C^m=\bigoplus_{j=1}^J {\mathcal K}_{\mu_j}(\tens{B}).
\end{equation}
For $y %\vec{z}
\in {\mathcal K}_{\mu_j}(\tens{B})$, we get
\begin{equation}
\label{ExpExp}
e^{\tau \tens{B}} y 
=e^{\tau \mu_j \tens{I}}e^{\tau (\tens{B}-\mu_j \tens{I})} y 
=e^{\mu_j \tau}\sum_{n=0}^{l_j-1} \frac{\tau^n}{n!} (\tens{B}-\mu_j \tens{I})^n y, 
\end{equation}
where $l_j$ is the smallest positive integer to satisfy ${\mathcal K}_{\mu_j}^{l_j}(\tens{B})
={\mathcal K}_{\mu_j}(\tens{B})$.
We define
$$
{\mathcal Z}(\tens{B})={\mathcal K}_0^1(\tens{B})\oplus %\left(
\bigoplus_{1\le j\le J; \real \mu_j<0
%j\in \{1, \ldots, J; \real \mu_j<0\}
}
{\mathcal K}_{\mu_j}(\tens{B}).
$$ 
Note that $y %\vec{z}
\in {\mathcal Z}(\tens{B})$ implies
\begin{equation}
\label{Eigen00}
\biggl(\tens{B}\prod_{1\le j\le J; \real \mu_j<0} (\tens{B}-\mu_j \tens{I})^{m}\biggr) y=0.
\end{equation}
\eqref{Decomp} and \eqref{ExpExp} lead to the following property:
\begin{lemma}\label{Eigen01}
$e^{\tau \tens{B}} y$ converges to some vector in $\C^m$ as 
$\tau\to\infty$, if and only if
$$
 y \in {\mathcal Z}(\tens{B}).
$$
\end{lemma}
%%%%%%%%%%%%%%%%%%%%%%%
The following two lemmas play important roles in the proof of Theorem~\ref{Necessity}.
%%%%%%%%%%%%%%%%%%%%%%%
\begin{lemma}\label{Eigen02}
Let $\tens{\Phi}_0$ be an $m\times m$ matrix.
If $\tens{\Phi}_0$
has an eigenvalue $\lambda$ whose real part is positive, 
then there exists a vector 
$z^0=z^0(\tens{\Phi}_0)\in \C^m$ such that the following holds:
Suppose that $\tens{\Phi}=\tens{\Phi}(\ve)$ be an  $m\times m$ matrix-valued function of 
$\ve\in [0,1]$, say;
if $\tens{\Phi}=\tens{\Phi}(\ve)$ is continuous at $\ve=0$,
and $\tens{\Phi}(0)=\tens{\Phi}_0$, then
there exist two positive constants $\ve_1$ and $\delta$
such that
\begin{equation}
\label{key03}
 {\mathcal Z}(\tens{\Phi}(\ve))\cap B_\delta (z^0) =\emptyset
\end{equation}
for any $\ve\in [0, \ve_1]$, where $B_\delta(z^0 %\vec{y}
 )=\{ y %\vec{z}
\in \C^m; %|\vec{z}-\vec{y}|
|y-z^0|\le \delta\}$.

Furthermore, if $\tens{\Phi}_0$ is a real matrix, 
then we can choose $z^0$ in the above from $\R^m$. 
\end{lemma}
%%%%%%%%%%%%%%%%%%%%%%%
\begin{proof}
We denote the eigenvalues of $\tens{\Phi}(\ve)$ by $\lambda_1(\ve), \dots, \lambda_m(\ve)$, 
with each eigenvalue being counted up to its algebraic multiplicity.
Without loss of generality, we may assume that each $\lambda_j(\ve)$ 
is continuous at $\ve=0$.
Note that $\lambda_1(0), \ldots, \lambda_m(0)$ are the eigenvalues of $\tens{\Phi}_0$.

If all the real parts of the eigenvalues of $\tens{\Phi}_0$ are positive,
then \eqref{key03} is trivial, since $\real \lambda_j(\ve)>0$ for $1\le j\le m$, provided that $\ve$ is small.
 
Suppose that the real part of some eigenvalue of
$\tens{\Phi}_0$ is non-positive. 
Then, without loss of generality, we may assume that
$\real \lambda_j(0)>0$ for $1\le j\le m'$ and $\real \lambda_j(0)\le 0$
for $m'+1\le j\le m$ with some positive integer $m'(<m)$.
We define
$$
\tens{\Psi}(\ve)=\tens{\Phi}(\ve) (\tens{\Phi}(\ve)-\lambda_{m'+1}(\ve)\tens{I})^m\cdots (\tens{\Phi}(\ve)-\lambda_{m}(\ve)\tens{I})^m,\quad \ve\ge 0.
$$
We are going to show that there exists ${z^0}\in \C^m$,
depending only on $\tens{\Phi}_0$,
such that
\begin{equation}
\label{key02}
\tens{\Psi}(0) z^0\ne 0.
\end{equation}
Let ${z}^1\in \C^m$
be an eigenvector of $\tens{\Phi}_0$ 
associated with the eigenvalue $\lambda_1(0)(>0)$. Then we have
$$
\tens{\Psi}(0) %\vec{z}_1
z^1=\lambda_1(0)
\bigl(\lambda_1(0)-\lambda_{m'+1}(0)\bigr)^m\cdots \bigl(\lambda_1(0)-\lambda_m(0)\bigr)^m 
z^1\ne 0,
$$
and we obtain \eqref{key02} with $z^0=z^1$. 
If $\tens{\Phi}_0$ is a real matrix, 
then the same is true for $\tens{\Psi}(0)$ (observe that if $\lambda$ is an eigenvalue of $\tens{\Phi}_0$, so is its complex conjugate $\overline{\lambda}$), 
and we have either $\tens{\Psi}(0)(\real {z}^1)\ne {0}$
or 
$\tens{\Psi}(0)(\ima {z}^1)\ne 0$.
Thus we can choose ${z}^0\in \R^m$ satisfying \eqref{key02}.

Suppose that $z^0$ satisfies \eqref{key02}. From the continuity of $\tens{\Psi}(\ve)$ at $\ve=0$,
the mapping $(\ve, y)\mapsto \tens{\Psi}(\ve)y$ is continuous at $(\ve, y)=(0, z^0)$.
Hence there exists $\ve_2>0$ and $\delta>0$ such that
$\tens{\Psi}(\ve)y\ne 0$ for any $y\in B_\delta(z^0)$ and $0\le \ve \le \ve_2$.
This leads to \eqref{key03}, because $y
\in {\mathcal Z}\bigl(\tens{\Phi}(\ve)\bigr)$ implies $\tens{\Psi}(\ve)y=0$ if
$\ve(>0)$ is sufficiently small (cf.~\eqref{Eigen00}; observe that $\real \lambda_j(\ve)\le 0$
implies $m'+1\le j\le m$ for small $\ve$).
\end{proof}
%%%%%%%%%%%%%%%%%%%%%%%%%%%%%%%%%%%%%%%%%%%%
%%%%%%%%%%%%%%%%%%%%%%%%%%%%%%%%%%%%%%%%%%%%
\begin{lemma}\label{Eigen03}
Let $\tens{\Phi}=\tens{\Phi}(\ve)$ be an
$m\times m$ real matrix-valued function,
and $\ve_1$ be a positive constant. 
Assume that $\tens{\Phi}=\tens{\Phi}(\ve)$ is continuous at $\ve=0$,
and that $\tens{\Phi}(0)\ne \tens{O}$,
where $\tens{O}$ denotes the zero matrix.

If all the real parts of the eigenvalues of $\tens{\Phi}(\ve)$ are zero for $0\le \ve \le \ve_1$, then for any 
${z}\in \R^m$
satisfying $\tens{\Phi}(0) z \ne 0$,
there exist two positive constant $\ve_2(\le \ve_1)$ and $\delta$ 
such that
$
{\mathcal Z}(\tens{\Phi}(\ve))\cap B_\delta (z) 
 =\emptyset
$
holds for any $\ve\in [0, \ve_2]$.
\end{lemma}
%%%%%%%%%%%%%%%%%%%%%%%%%%%%%%%%%%%%%%%
\begin{proof}
Suppose $\tens{\Phi}(0) z\ne 0$.
Then, from the continuity of
the mapping $(\ve, y)\mapsto \tens{\Phi}(\ve)y$ at $(\ve, y)=(0, z)$,
there exist two positive constant $\ve_2(\le \ve_1)$ and $\delta$ 
such that $\tens{\Phi}(\ve)y\ne 0$ for $0\le \ve\le \ve_2$ and $y\in B_\delta(z)$.
This completes the proof because ${\mathcal Z}(\tens{\Phi}(\ve))={\mathcal K}_0^1(\tens{\Phi}(\ve))$ from the assumption.
\end{proof}

Now we are in a position to prove Theorem~\ref{Necessity}.
{(AFP)} under the null condition follows immediately from Theorem \ref{PointwiseAsymptotics} by regarding $w=u$ and neglecting $v$.
Hence we only have to prove that (2) implies (1).

Suppose that Condition~\ref{OurCond} is satisfied with $D={\mathcal X}_N$, and {(AFP)} holds.
We define 
$\tens{B}(\omega, \xi, \eta)=\left(B_{jk}(\omega, \xi, \eta)\right)_{1\le j, k\le N'}$ by
$$
{B}_{jk}(\omega, \xi, \eta)=-\frac{1}{2}\sum_{a=0}^3 \omega_a
\sum_{l=1}^{N''} \left(c_{jk}^{a, l}(\omega) \xi_l+\sum_{b=0}^3 d_{jk}^{\, a, lb}(\omega) 
\omega_b \eta_l\right)
$$
for $\omega=(\omega_1, \omega_2, \omega_3)\in S^2$, and $\xi=(\xi_l)^{\rm T}_{1\le l\le N''}$, 
$\eta=(\eta_l)^{\rm T}_{1\le l\le N''}
\in \R^{N''}$, where $c_{jk}^{a, l}$ and $d_{jk}^{\, a, lb}$ are from \eqref{AssA}.
Observe that we have
$$
\tens{A}[\zeta](\sigma, \omega)=\tens{B}\bigl(\omega, \zeta(\sigma, \omega), (\pa_\sigma \zeta)(\sigma, \omega)\bigr),
$$
where $\tens{A}[\zeta](\sigma, \omega)$ is given by \eqref{Weiss}-\eqref{AssB}.

Suppose that there exists some $(\omega', \xi', \eta')\in S^2\times
\R^{N''}\times \R^{N''}$
such that 
$\tens{A}_0:=\tens{B}({\omega}', {\xi}', {\eta}')$ 
has an eigenvalue $\lambda$ with $\real \lambda\ne 0$.
We may assume $\real \lambda>0$, because unless so,
$-\lambda$ is an eigenvalue of $\tens{B}({\omega}', -{\xi}', -{\eta}')$ 
with $\real (-\lambda)>0$.
We put
$z'={z}^0(\tens{A}_0)$, where $z^0(\tens{A}_0)$ is from Lemma~\ref{Eigen02} with $m=N'$
and $\tens{\Phi}_0=\tens{A}_0$.
Note that we may assume $z'\in \R^{N'}$.  
We fix some $\sigma'\ne 0$.
Writing $\xi'=(\xi'_{k})^{\rm T}_{1\le k\le N''}$, $\eta'=(\eta'_{k})^{\rm T}_{1\le k\le N''}$,
and $z'=(z'_{j})^{\rm T}_{1\le j\le N'}$,
by Lemma~\ref{RadConst} we can find $(f, g)\in C^\infty_0(\R^3; \R^N)\times C^\infty_0(\R^3; \R^N)$
such that we have
\begin{align}
\label{CD01}
& (\pa_\sigma {\mathcal F}_0[f_j, g_j])(\sigma', \omega)= z'_{j},
\\
\label{CD02}
& {\mathcal F}_0[f_{N'+k}, g_{N'+k}](\sigma', \omega)= \xi'_{k}, 
\ (\pa_\sigma {\mathcal F}_0[f_{N'+k}, g_{N'+k}])(\sigma', \omega)=\eta'_{k}
\end{align}
for $1\le j\le N'$, $1\le k\le N''$, and $\omega\in S^2$.
Let $V$ and $W$ be the modified and the standard
asymptotic profiles corresponding to $(f,g)$ above. 
We write $V=V(\sigma, \omega; \ve)$ and $W=W(\sigma, \omega; \ve)$ to indicate
the dependence on the parameter $\ve$ explicitly.
We put $\tens{A}_\ve=\tens{A}[W(\cdot,\cdot;\ve)](\sigma', \omega')$ for $\ve>0$.
Then, by \eqref{As02}, we get
$$
\tens{A}_\ve=\tens{B}\bigl(\omega', W(\sigma', \omega'; \ve), (\pa_\sigma W)(\sigma', \omega'; \ve)\bigr)
\to \tens{B}(\omega', \xi', \eta')=\tens{A}_0
$$
as $\ve\to +0$. 
Hence, from Lemma~\ref{Eigen02}, there exist $\ve_1>0$ and $\delta>0$ such that
$$
{\mathcal Z}(\tens{A}_\ve)\cap B_\delta(z')=\emptyset
$$
holds for $\ve\in [0,\ve_1]$. 
Since \eqref{As01} and \eqref{CD01} imply $(\pa_\sigma V)(\sigma', \omega'; \ve)\in B_\delta(z')$ for sufficiently small $\ve(>0)$, we obtain
\begin{equation} 
\label{CD04}
(\pa_\sigma V)(\sigma', \omega'; \ve) \not\in {\mathcal Z}(\tens{A}_\ve)
\end{equation}
for sufficiently small $\ve(>0)$.
From {(AFP)} and \eqref{Ruh}, 
$\ve e^{\tau \tens{A}_\ve} (\pa_\sigma V)(\sigma', \omega'; \ve)$
must converge to some vector in $\R^{N'}$ as $\tau\to \infty$, provided 
that $\ve(>0)$ is sufficiently small. 
However this never occurs because of
\eqref{CD04} and Lemma~\ref{Eigen01}.  
%%%%%%%%%%%%%%%%%%%%%%%%%%%%%%%%%%%%%%%%%%%%%%%%%%%%%%%%%%%
%%%%%%%%%%%%%%%%%%%%%%%%%%%%%%%%%%%%%%%%%%%%%%%%%%%%%%%%%%%
Hence we conclude that all the real parts of the eigenvalues of
$\tens{B}(\omega, \xi, \eta)$ must vanish for any $(\omega, \xi, \eta)\in S^2\times \R^{N''}\times \R^{N''}$.

We fix arbitrary $(\omega', \xi', \eta')\in S^2\times \R^{N''}\times \R^{N''}$, and
put $\tens{A}_0=\tens{B}(\omega', \xi', \eta')$ as before.
Suppose that $\tens{A}_0\ne \tens{O}$. Then we can find a vector $z'\in \R^{N'}$
such that $\tens{A}_0z' \ne {0}$. 
Now, for this new choice of
$\omega'$, $\xi'$, $\eta'$, and $z'$,
we choose $(f,g)$ satisfying \eqref{CD01} and \eqref{CD02}.
Then, following the similar lines to the above, but using Lemma~\ref{Eigen03}
instead of Lemma~\ref{Eigen02}, we reach at \eqref{CD04} again, which is a contradiction.
Hence we conclude $\tens{A}_0=\tens{O}$. Because $(\omega', \xi', \eta')$ can be chosen arbitrarily, this means that 
\begin{equation}
\label{TabiZero}
\tens{B}(\omega, \xi, \eta)=\tens{O}, \quad (\omega, \xi, \eta)\in S^2\times \R^{N'}\times \R^{N'}.
\end{equation}
Therefore $\sum_{a=0}^3\omega_a c_{jk}^{a, l}(\omega)=\sum_{a,b=0}^3 \omega_a \omega_b d_{jk}^{\, a, lb}(\omega)=0$ for any $\omega\in S^2$. 
Remembering that we have $\pa_a=Z_a-\omega_a \pa_t$, and 
looking at $\eqref{AssA}$, we get
\begin{align*}
 G_j(\omega, u, \pa u)=& \sum_{a=0}^3
 \sum_{k=1}^{N'}\sum_{l=1}^{N''} 
  \left(c_{jk}^{a, l}(\omega) w_{l}+\sum_{b=0}^3d_{jk}^{\, a, lb}(\omega) \pa_b w_{l}\right) 
(Z_a v_k) \\
&{}-\sum_{a=0}^3\sum_{k=1}^{N'}\sum_{l=1}^{N''} 
  \sum_{b=0}^3d_{jk}^{\, a, lb}(\omega) (Z_b w_{l})(\omega_a \pa_t v_k)
\end{align*}
for a smooth solution $u=(v^{\rm T}, w^{\rm T})^{\rm T}$ 
to \eqref{OurSys}-\eqref{Data0} on $[0,T)\times \R^3$,
and we obtain 
$|G_j(\omega, u, \pa u)|\le C\left(|u|+|\pa u|\right)|Zu|$.
Now, from \eqref{Form02}, \eqref{Tama}, and Lemma~\ref{CommZG} we get 
\begin{equation}
|F\left(u(t,r\omega), \pa u(t, r\omega)\right)| 
\le C \jb{t+r}^{-1}\left(|u(t, r\omega)|_{1}+|\pa u(t, r\omega)|\right)|u(t, r\omega)|_1
\label{Tama02}
\end{equation}
for any $\omega\in S^2$ and $(t,r)\in [0, T)\times [0,\infty)$ satisfying $r \ge t/2\ge 1$.

We fix $\sigma_0\ne 0$.
Let $(X, Y)\in \R^N\times \R^N$ be given. By Lemma \ref{RadConst}, 
we can choose $(f, g)\in C^\infty_0(\R^3;\R^N)\times C^\infty_0(\R^3; \R^N)$
such that
\begin{equation}
\label{TabiData}
\bigl({\mathcal F}_0[f_j, g_j](\sigma_0, \omega), (\pa_\sigma {\mathcal F}_0[f_j, g_j])(\sigma_0, \omega)\bigr)
=(X_j, Y_j), \quad \omega \in S^2,\ 1\le j\le N.
\end{equation}
Let $u=(v^{\rm T}, w^{\rm T})^{\rm T}$ be the solution to \eqref{OurSys}-\eqref{Data0} with this choice of $(f,g)$ from now on.
We write $u=u(t,x;\ve)$ in order to indicate the dependence on $\ve$ explicitly.
Let $V$ and $W$ be the modified and the standard asymptotic profiles given by Theorem~\ref{PointwiseAsymptotics}. 
Note that, from \eqref{TabiZero} and \eqref{Ruh}, we have
\begin{equation}
\label{Ruh07}
\sum_{a=0}^3\left|r(\pa_a v)(t, r\omega)-\ve \omega_a (\pa_\sigma V)(r-t, \omega)\right|\le C\ve 
\jb{t+r}^{3\lambda+C\ve-1}
\end{equation} 
for $r\ge t/2\ge 1$ and $\omega \in S^2$.
We put $U=(V^{\rm T}, W^{\rm T})^{\rm T}$.
From \eqref{Form01}, \eqref{Ruh07}, \eqref{Piyoko}, and \eqref{Chada}, we get
\begin{align}
& 
\lim_{t\to \infty} \left.\left\{\ve^{-2}{r^2} F\bigl(u(t, r\omega; \ve), \pa u(t, r\omega; \ve)\bigr)\right\}\right|_{r=t+\sigma}
\nonumber\\
& \qquad\qquad 
=F^{\rm red}\bigl(\omega, U(\sigma, \omega;\ve), \pa_\sigma U(\sigma, \omega;\ve)\bigr)
\label{NullRad03}
\end{align}
for any fixed $(\sigma, \omega)\in \R\times S^2$.
On the other hand,
using \eqref{DE01} and \eqref{DE02} to evaluate the right-hand side of
\eqref{Tama02}, we obtain
\begin{equation}
\label{Tama-han}
\left|\left.\left\{\ve^{-2}{r^2} F\bigl(u(t,r\omega;\ve), \pa u(t, r\omega;\ve)\bigr)
\right\}\right|_{r=t+\sigma}\right|\le C\jb{2t+\sigma}^{2\lambda-1}\to 0,\ \, t\to\infty.
\end{equation}
It follows from \eqref{NullRad03} and \eqref{Tama-han} that
\begin{equation}
\label{NullRad05}
F^{\rm red}\bigl(\omega, U(\sigma_0, \omega; \ve), \pa_\sigma U(\sigma_0, \omega; \ve)\bigr)=0,
\quad \omega\in S^2.
\end{equation}
Taking the limit in \eqref{NullRad05} as $\ve\to +0$,
we obtain from \eqref{As01}, \eqref{As02}, and \eqref{TabiData}
that $F^{\rm red}(\omega, X, Y)=0$
for $\omega\in S^2$ and $1\le j\le N$
(note that $F^{\rm red}$ does not depend on $X_j$ with $1\le j\le N'$ because of \eqref{Form01}).
Since $(X,Y)\in \R^N\times \R^N$ can be chosen arbitrarily, the null condition 
\eqref{NC00} is satisfied.
This completes the proof.
\qed 
%%%%%%%%%%%%%%%%%%%%%%%%%%%%%%%%%%%%%%%%%%%%%%%%%%%%
%%%%%%%%%%%%%%%%%%%%%%%%%%%%%%%%%%%%%%%%%%%%%%%%%%%%
\subsection{Proof of Theorem~$\ref{Necessity02}$}
Before we start the proof of Theorem~\ref{Necessity02},
we investigate $e^{\tau \tens{B}}$ for a real matrix $\tens{B}$ whose rank is at most one.
Let $p=(p_1, \ldots, p_N)^{\rm T}$, ${q}=(q_1, \ldots, q_N)^{\rm T}\in \R^N$.
We put 
$$
\tens{B}= q p^{\rm T}
=\left(
   \begin{matrix} p_1 q_1 & \dots & p_N q_1\\
                     \vdots &  \ddots & \vdots\\
                     p_1 q_N & \dots & p_N q_N\\
   \end{matrix}
 \right).
$$
For any $y=(y_1, \dots, y_N)^{\rm T}\in \R^N$, we have
$\tens{B} y=q p^{\rm T} y=  \ip{p}{y} q$,
where $\ip{\, \cdot\, }{\, \cdot\, }$ denotes the inner product in $\R^N$.
Hence, for a positive integer $k$, we have
$$
\tens{B}^k y=\ip{p}{y}{\tens{B}}^{k-1} q=\ip{p}{y} \ip{p}{q}^{k-1} q,
\quad y\in \R^N,
$$
and we get
\begin{align}
\label{Nec02}
e^{\tau \tens{B}} y
=& \begin{cases}
y+\tau \ip{p}{y} q, & 
 \text{if $\ip{p}{q}=0$},\\
y+\ip{p}{q}^{-1}\ip{p}{y}\left(e^{\ip{p}{q}\tau}-1
                  \right) q,
& \text{if $\ip{p}{q}\ne 0$}
\end{cases}
\end{align}
for any $y\in \R^N$.
We need the following lemma for the proof of Theorem~\ref{Necessity02}:
\begin{lemma}\label{AiHatsushiba}
Let $\tens{C}$ be an $N\times N$ real matrix, and %$\vec{b}\in \R^m$.
${b}\in \R^N$.
Suppose that there exist two vectors 
 ${y}$, ${y}'\in \R^N$
such that
$$
\ip{\tens{C}{y}}{{y}}\ne 0, \quad \ip{\tens{C}{y}'}{{b}}\ne 0.
$$
Then there exists $z\in \R^N$
such that we have
\begin{equation}
\ip{\tens{C} z}{z}\ne 0, \text{ and } \ip{\tens{C}z}{{b}}> 0.
\label{HanakoShinohara}
\end{equation}
\end{lemma}
\begin{proof}
It suffices to prove the existence of $z$ satisfying
\eqref{HanakoShinohara}
with 
$\ip{\tens{C}z}{{b}}> 0$
being replaced by 
$\ip{\tens{C}z}{{b}}\ne 0$,
because 
either $z$ or $-z$
has the desired property, 
depending on the sign of 
$\ip{\tens{C}z}{{b}}$. 
The case where we have either 
$\ip{\tens{C}{y}}{{b}}\ne 0$ or $\ip{\tens{C}{y'}}{{y'}}\ne 0$
is triviality. 
Hence we suppose
$\ip{\tens{C}{y}}{{b}}=\ip{\tens{C}{y'}}{{y'}}=0$,
and we put $z={y}+\lambda {y}'$
for some $\lambda\ne 0$.
Then we have 
$\ip{\tens{C}z}{{b}}=\lambda\ip{\tens{C}{y'}}{{b}}\ne 0$. 
Since $\ip{\tens{C}z}{{z}}=\ip{\tens{C} {y}}{{y}}+\lambda \left(\ip{\tens{C}{y}}{{y'}}+\ip{\tens{C}y'}{y}\right)$,
we find that $\ip{\tens{C}z}{z}\ne 0$ if $|\lambda|$ is sufficiently small.
This completes the proof.
\end{proof}

Now we start the proof of Theorem~\ref{Necessity02}.
(AFP) and (AFE) under the null condition follow from
Theorem~\ref{PointwiseAsymptotics} and Corollary~\ref{EnergyAsymptotics},
respectively.
Hence it suffices to prove that
if (1) does not hold, then neither (2) nor (3) holds.

Suppose that Condition~\ref{AlinhacCond}
is fulfilled, but 
the null condition is not satisfied.
As we have seen in Section~\ref{Ex3},
we can apply Theorem \ref{GE} to the extended system \eqref{HolTheGarasshi},
and we can obtain a global solution $u=(u_j)_{1\le j\le N}$ to 
the original Cauchy problem \eqref{OurSys}-\eqref{Data0} with $F=F(\pa u)$
for sufficiently small $\ve(>0)$.

Since the null condition is not satisfied,
there exist some $\omega'\in S^2$ and $Y' %=(Y'_{j})_{1\le j\le N}^{\rm T}
\in \R^N$ 
such that $F^{\rm red}(\omega', Y')\ne 0$.
Then we may assume that $\beta(\omega)$ in the assumption \eqref{AA-1} is smooth
and satisfies $|\beta(\omega)|=1$ in some neighborhood ${\mathcal O}(\subset S^2)$ of $\omega'$. 
Indeed, there exists a neighborhood ${\mathcal O}$ of $\omega'$
such that we have $F^{\rm red}(\omega, Y')\ne 0$ for $\omega\in {\mathcal O}$.
We define $\widetilde{\beta}(\omega)=|F^{\rm red}(\omega, Y')|^{-1} F^{\rm red}(\omega, Y')$.
Then $\widetilde{\beta}(\omega)$ is a smooth
function of $\omega\in {\mathcal O}$,
and apparently satisfies $|\widetilde{\beta}(\omega)|=1$. Moreover, from \eqref{AA-1} 
we see that $\widetilde{\beta}(\omega)$ is proportional to $\beta(\omega)$,
and \eqref{AAB-2} as well as \eqref{AA-1}
(with modified $M(\omega, Y)$) remains valid if we replace 
$\beta(\omega)$ by $\widetilde{\beta}(\omega)$ for $\omega\in {\mathcal O}$.

Since $|\beta(\omega)|=1$ for $\omega\in{\mathcal O}$, it follows from \eqref{AA-1}, \eqref{AAB-1}, and \eqref{Yone}
that
$$
M(\omega,Y)=\bigl\langle \beta(\omega), F^{\rm red}(\omega, Y)\bigr\rangle=
\sum_{m=1}^N \sum_{k=1}^N \sum_{l=1}^{N_0} \beta_m(\omega)g_{ml, k}(\omega)Y_k h_l(\omega, Y)
$$
for $(\omega,Y)\in{\mathcal O}\times \R^N$. 
Thus, again by \eqref{AA-1}, we obtain
\begin{equation}
F_j^{\rm red}(\omega, Y)=\beta_j(\omega)
\sum_{m=1}^N \sum_{k=1}^{N} \sum_{l=1}^{N_0} 
 \beta_m(\omega)g_{ml, k}(\omega)Y_k h_l(\omega, Y)
\label{Lily01}
\end{equation}
for $(\omega,Y)\in {\mathcal O}\times \R^N$ and $1\le j\le N$.
We fix a neighborhood ${\mathcal O}'$ of $\omega'$,  which is strictly contained in ${\mathcal O}$,
and choose a smooth cut-off function $\chi=\chi(\omega)$ satisfying
$\chi\equiv 1$ on ${\mathcal O}'$ and $\chi\equiv 0$ on $S^2\setminus {\mathcal O}$.
From \eqref{AAB-1} and \eqref{Lily01}, we have
$$
F_j^{\rm red}(\omega, Y)=\sum_{l=1}^{N_0} \widetilde{g}_{jl}(\omega, Y)h_l(\omega, Y),
\quad 1\le j\le N,\ (\omega,Y)\in S^2\times \R^N,
$$
where
$$
\widetilde{g}_{jl}(\omega, Y)=\sum_{k=1}^N \widetilde{g}_{jl, k}(\omega) Y_k, \quad 1\le j\le N,\ 1\le l\le N_0
$$
with
\begin{equation}
\label{Kabie02}
\widetilde{g}_{jl, k}(\omega)=\left(1-\chi(\omega)\right)g_{jl, k}(\omega)+\chi(\omega)
\beta_j(\omega) \sum_{m=1}^N \beta_m(\omega) g_{ml, k}(\omega)
\end{equation}
for $1\le j, k \le N$ and $1\le l \le N_0$. Thus \eqref{AAB-1} holds true
if we replace $g_{jl, k}$ by $\widetilde{g}_{jl, k}$.
%%%%%%%%%%%%%%%%%%%%%%%%%
We consider the extended system \eqref{HolTheGarasshi}, 
and apply Theorem~\ref{PointwiseAsymptotics} to it
in the following way:
Let $G$ be defined by \eqref{Nisesshi01}-\eqref{Nisesshi02} with
$g_{jl, k}$ being replaced by $\widetilde{g}_{jl, k}$, and let
$\tens{A}[\zeta]=(A_{jk}[\zeta])_{1\le j,k\le 5N}$ be determined by \eqref{Weiss}-\eqref{AssB} (with $N'=5N$ and $N''=N_0$) 
correspondingly to this replaced $G$, where $\zeta=(\zeta_l)_{1\le l\le N_0}^{\rm T}$ is a smooth function of 
$(\sigma,\omega)\in \R\times S^2$. 
Then, for $1\le j\le N$ we get
$$
A_{jk}[\zeta](\sigma, \omega)=
\begin{cases}
\displaystyle
-\frac{1}{2}\sum_{l=1}^{N_0} \widetilde{g}_{jl, k}(\omega) \zeta_l(\sigma, \omega), & 1\le k\le N,\\
0, & N+1\le k\le 5N.
\end{cases}
$$
We define $\tens{B}(\omega, \eta)=\bigl(B_{jk}(\omega, \eta) \bigr)_{1\le j, k\le N}$ by
\begin{equation}
\label{Kabie}
\tens{B}(\omega, \eta)=
-\frac{1}{2}\beta(\omega) \eta^{\rm T} \sum_{m=1}^N
\beta_m(\omega)\tens{G}_m(\omega), 
\quad (\omega, \eta)\in {\mathcal O}\times \R^{N_0},
\end{equation}
where $\tens{G}_1,\ldots, \tens{G}_N$ are $N_0\times N$ matrix-valued
functions defined by
$$
\tens{G}_j(\omega)=\bigl(g_{jl, k}(\omega)\bigr)_{1\le l\le N_0, 1\le k \le N}, \quad 1\le j \le N.
$$
Then \eqref{Kabie02} yields
$$
A_{jk}[\zeta](\sigma, \omega)=B_{jk}\bigl(\omega, \zeta(\sigma, \omega)\bigr),\quad
1\le j, k\le N,\ (\sigma, \omega) \in \R\times {\mathcal O'}.
$$
Now we apply Theorem~\ref{PointwiseAsymptotics} to the extended 
system \eqref{HolTheGarasshi} of $u^*$ with $5N+N_0$ components,
and we look only at the first $N$ components corresponding to the original $u$, 
and the last $N_0$ components corresponding to $w$ which is defined by \eqref{Ofuku}. Then we see that
there exist $U=\bigl(U_j(\sigma, \omega)\bigr)_{1\le j\le N}^{\rm T}$ and
$W=\bigl(W_k(\sigma, \omega)\bigr)_{1\le k\le N_0}^{\rm T}$ such that
\begin{align}
& \sum_{a=0}^3 \left|r (\pa_a u)(t,r\omega)-
 \ve \omega_a e^{(\ve \log t) \tens{B}(\omega, W(r-t, \omega))}
(\pa_\sigma U)(r-t, \omega)\right|
\nonumber \\
& \qquad\qquad\qquad\qquad\qquad\qquad\quad\,    
\le C\ve \jb{t+r}^{3\lambda+C\ve-1}, \quad \omega\in {\mathcal O}',
\label{Ruh-A01}\\
\label{Piyoko-A01}
& \left|r w(t, r\omega)-\ve W(r-t, \omega)\right|\le C\ve \jb{t+r}^{2\lambda-1}\jb{t-r}^{1-\rho},
\quad \omega\in S^2
\end{align}
for 
$(t,r)$ satisfying $r\ge t/2\ge 1$,
where $\lambda$ and $\rho$ are positive constants as in Theorem~\ref{GE}.
Moreover, correspondingly to \eqref{As01} 
we also have
\begin{align}
\label{As11}
& \sum_{j=1}^N |\pa_\sigma U_j(\sigma, \omega)-\pa_\sigma {\mathcal F}_0[f_j, g_j](\sigma, \omega)|\le 
C\ve (1+|\sigma|)^{3\lambda+C\ve-1}.
\end{align}

Let $\tens{H}$ be a matrix-valued function 
given by 
$$
\tens{H}(\omega)=\left(\sum_{a=0}^3 h_{l, k a}\omega_a
\right)_{1\le l\le N_0, 1\le k \le N},
$$ 
where $h_{l, ka}$ is from \eqref{Mil}.
Then the condition \eqref{AAB-2} leads to
\begin{equation}
\label{Lily02}
\tens{H}(\omega)\beta(\omega)={0},\quad \omega\in S^2.
\end{equation}
We claim that
\begin{equation}
\label{ExpW01}
W(\sigma, \omega)=\tens{H}(\omega)(\pa_\sigma 
{U})(\sigma, \omega)
=\tens{H}(\omega) 
{\Upsilon}(\tau, \sigma, \omega)
\end{equation}
for $(\sigma, \omega)\in \R\times {\mathcal O}'$ and $\tau\in \R$,
where 
${\Upsilon}(\tau, \sigma, \omega)=e^{\tau \tens{B}(\omega, W(\sigma, \omega))}(\pa_\sigma {U})(\sigma, \omega)$.
From \eqref{Lily02} and \eqref{Kabie}, we get
$$
\tens{H}(\omega) 
\tens{B}(\omega, \eta)=0,\quad (\omega, \eta)\in {\mathcal O}' \times \R^{N_0},
$$
and we obtain 
$$
\tens{H}(\omega)e^{\tau 
\tens{B}(\omega, \eta)}=\tens{H}(\omega),
\quad (\tau, \omega, \eta)\in \R\times 
{\mathcal O}' \times \R^{N_0},
$$
which leads to
\begin{equation}
\label{ExpW04}
\tens{H}(\omega) 
{\Upsilon}(\tau, \sigma, \omega)=\tens{H}(\omega)(\pa_\sigma 
{U})(\sigma, \omega), \quad (\tau, \sigma, \omega)\in \R\times\R\times {\mathcal O}'.
\end{equation}
From \eqref{ExpW04}, \eqref{Ofuku}, and \eqref{Ruh-A01}, we get
\begin{align}
|r w(t,r\omega)-\ve \tens{H}(\omega)(\pa_\sigma 
{U})(r-t, \omega)|=&
|r w(t,r\omega)-\ve \tens{H}(\omega) 
{\Upsilon}(\ve\log t, r-t, \omega)| 
\nonumber\\
\le & C\ve \jb{t+r}^{3\lambda+C\ve-1}.
\label{ExpW02}
\end{align}
By \eqref{Piyoko-A01} and \eqref{ExpW02}, we get
$$
|W(\sigma, \omega)-\tens{H}(\omega) (\pa_\sigma %\widetilde
{U})(\sigma,\omega)|
\le \lim_{t\to\infty} C\jb{2t+\sigma}^{3\lambda+C\ve-1}=0, \quad (\sigma, \omega)\in \R\times {\mathcal O}',
$$
which, together with \eqref{ExpW04}, shows \eqref{ExpW01}.

We put
$$
\tens{C}(\omega) =
-\frac{1}{2} \sum_{m=1}^N {\beta_m(\omega)}
\tens{G}_m(\omega)^{\rm T} \tens{H}(\omega),
\quad \omega\in {\mathcal O}',
$$
so that we have 
$$
\tens{B}(\omega, \tens{H}(\omega) Y)=\beta(\omega)\bigl(\tens{C}(\omega)Y\bigr)^{\rm T},\quad \omega\in{\mathcal O}'.
$$ 
Observing that \eqref{Mil} can be written as $\bigl(h_l(\omega,Y)\bigr)_{1\le l\le N_0}^{\rm T}=\tens{H}(\omega)Y$,
we obtain from \eqref{Lily01} and \eqref{Kabie} 
that
$$
F^{\rm red}(\omega, Y)=-2\tens{B}(\omega, \tens{H}(\omega)Y)Y
=-2\ip{\tens{C}(\omega)Y}{Y}\beta(\omega),\quad (\omega, Y)\in{\mathcal O}'\times \R^N.
$$
Hence, recalling $F^{\rm red}(\omega', Y')\ne 0$, and
taking smaller ${\mathcal O}'$ if necessary, 
we may assume that
\begin{equation}
\label{ayumu035}
\bigl\langle{\tens{C}(\omega)Y'},{Y'}\bigr\rangle\ne 0 \text{ for any $\omega\in {\mathcal O}'$}.
\end{equation}

First we assume that there exists some $(\omega^*, Y^*)\in {\mathcal O}' \times \R^N$ such that
$$
\ip{\tens{C}(\omega^*)Y^*}{\beta(\omega^*)}\ne 0.
$$
Then in combination with \eqref{ayumu035} for $\omega=\omega^*$, Lemma~\ref{AiHatsushiba} implies that
there exists $Y^0\in \R^N$ satisfying
\begin{equation}
\label{ayumu05}
\ip{\tens{C}(\omega^*) Y^0}{Y^0}
\ne 0 \text{ and } \ip{\tens{C}(\omega^*) Y^0}{\beta(\omega^*)}>0.
\end{equation}
We put
\begin{align*}
\lambda_0=\ip{\tens{C}(\omega^*)Y^0}{\beta(\omega^*)},\ \mu_0=\ip{\tens{C}(\omega^*)Y^0}{Y^0}.
\end{align*}
Note that \eqref{ayumu05} implies $\lambda_0>0$ and $\mu_0\ne 0$.
Now fix some $\sigma_0>0$.
By Lemma~\ref{RadConst}, we can choose 
some $C^\infty_0$-data such that we have
\begin{equation}
\label{ayumu06}
\bigl((\pa_\sigma \mathcal{F}_0[f_j, g_j])(\sigma_0, \omega) \bigr)_{1\le j\le N}^{\rm T}
=Y^{0}, \quad \omega\in S^2.
\end{equation}
Let $U=U(\sigma, \omega; \ve)$ be the modified asymptotic profile corresponding to this choice of data.
We set 
\begin{align*}
\tilde{\tens{\Theta}}_\ve(t, \sigma, \omega)= & (\ve \log t)\tens{B} \bigl(
             \omega, \tens{H} (\omega) (\pa_\sigma U) (\sigma, \omega; \ve) \bigr), \\
\lambda_\ve(\sigma, \omega)= &
      \ip{\tens{C}(\omega)(\pa_\sigma U)(\sigma, \omega; \ve)}{\beta(\omega)}, \\
\mu_\ve(\sigma, \omega) =& 
\langle \tens{C}(\omega)
(\pa_\sigma U)(\sigma, \omega;\ve), (\pa_\sigma U)(\sigma, \omega; \ve)\rangle.          
\end{align*}
Since $\tens{C}$ and $\beta$ are smooth in ${\mathcal O}'$,
in view of \eqref{As11} and \eqref{ayumu06} we get 
$$
\lim_{(\sigma, \omega,\ve)\to (\sigma_0,\omega^*,0)}\lambda_\ve(\sigma,\omega)=\lambda_0,
\quad \lim_{(\sigma, \omega, \ve)\to (\sigma_0,\omega^*,0)}\mu_\ve(\sigma,\omega)=\mu_0.
$$
Hence there exist an open set ${\mathcal O}''\subset {\mathcal O}'$
(with $\omega^*\in {\mathcal O}''$) and
some open interval ${\mathcal I}(\ni \sigma_0)$
such that we have
\begin{align}
&\lambda_0/2 \le \lambda_\ve(\sigma, \omega)\le 3\lambda_0/2, 
\qquad\,
(\sigma, \omega) \in {\mathcal I} \times {\mathcal O}'', \nonumber\\
\label{ayumu42}
&|\mu_\ve(\sigma, \omega)| \ge |\mu_0|/2>0,
\qquad\quad
(\sigma, \omega) \in {\mathcal I} \times {\mathcal O}''
\end{align}
for small and positive $\ve$.
We may assume $\sigma>0$ for $\sigma \in {\mathcal I}$.
Now \eqref{Nec02} (with $p=\tens{C}(\omega)(\pa_\sigma U)$, $q=\beta(\omega)$, and $y=\pa_\sigma U$) yields
\begin{align}
\left|e^{\tilde{\tens{\Theta}}_\ve(t, \sigma, \omega)} (\pa_\sigma U)(\sigma, \omega)\right|
=& \left|(\pa_\sigma U) (\sigma, \omega)+
\frac{\mu_\ve(\sigma, \omega)}{\lambda_\ve(\sigma, \omega)} (e^{\lambda_\ve(\sigma, \omega) \tau}-1) \beta(\omega)\right|
\nonumber\\
\ge & \frac{|\mu_0|}{3\lambda_0}(e^{\lambda_0\tau/2}-1)-C
\ge \frac{|\mu_0|}{6\lambda_0} t^{\lambda_0\ve/2}
\label{ayumu43}
\end{align}
for 
$(\sigma, \omega)\in {\mathcal I}\times {\mathcal O}''$ and
$t\ge t_\ve^*$, where we have put $\tau=\ve\log t$, and 
$t_\ve^*$ is some positive constant depending on $\ve$. 
In view of \eqref{Ruh-A01}, this shows that {(AFP)} does not hold.
We define 
$$
\widetilde{E}(t)^2:=
\int_{t/2}^\infty\left(\int_{S^2}|e^{\tilde{\tens{\Theta}}_\ve(t,r-t,\omega)}(\pa_\sigma U)(r-t,\omega)|^2dS_\omega\right)dr.
$$
By \eqref{ayumu43}, we obtain
\begin{align}
\label{ayumu48}
\widetilde{E}(t)
\ge \left(
         \frac{\mu_0^2}{36\lambda_0^2} \int_{\mathcal I} \int_{{\mathcal O}''} t^{\lambda_0\ve} dS_\omega 
d\sigma \right)^{1/2}\ge  C t^{\lambda_0 \ve/2} 
\end{align}
for $t\ge t_\ve^*$. 
From Corollary~\ref{GenEnergyEst} and \eqref{ayumu48}, we find that
$\|u(t)\|_E\ge C\ve (1+t)^{C\ve}$ for $t\ge t_\ve^{**}$ with some $t_\ve^{**}>0$, 
and {(AFE)} never holds. 

Next we assume that 
$$
\ip{\tens{C}(\omega)Y}{\beta(\omega)}=0
$$
holds for all $(\omega, Y)\in {\mathcal O}'\times \R^N$.
Fix some $\omega^*\in {\mathcal O}'$, and some $\sigma_0>0$.
Because of \eqref{ayumu035}, if we put $Y^0=Y'$, we have $\ip{\tens{C}(\omega^*)Y^0}{Y^0}\ne 0$. 
Now we choose some data such that we have \eqref{ayumu06}.
As above, we can also choose some ${\mathcal O}'' (\subset {\mathcal O}')$ and
some interval ${\mathcal I}(\ni \sigma_0)$ such that \eqref{ayumu42} holds for small $\ve$.
Then, from \eqref{Nec02} we obtain
\begin{align}
\left|e^{\tilde{\tens{\Theta}}_\ve(t, \sigma, \omega)} (\pa_\sigma U)(\sigma, \omega; \ve)\right|
=& \left|(\pa_\sigma U)(\sigma, \omega;\ve)+(\ve \log t) \mu_\ve(\sigma, \omega) \beta(\omega)\right|
\nonumber\\
\ge &  \frac{|\mu_0|}{2} \ve \log t-C \ge \frac{|\mu_0|}{4}\ve \log t
\label{ayumu44}
\end{align}
for $(\sigma, \omega)\in {\mathcal I}\times {\mathcal O}'' $
and $t\ge t_\ve^*$ with some $t_\ve^*>0$.
From \eqref{ayumu44} we conclude that {(AFP)} fails to hold. 
Finally \eqref{ayumu44} leads to
$\widetilde{E}(t)\ge  C \ve \log t$
for $t\ge t_\ve^*$, and 
we find that (AFE) does not hold because we get 
$\|u(t)\|_E \ge C\ve^2 \log t$ for $t\ge t_\ve^{**}$ with some $t_\ve^{**}>0$.
This completes the proof. \qed

\section*{Acknowledgments}
The author would like to express 
his gratitude to Professor Yoshio Tsutsumi, 
because the present work is motivated by his question during the author's talk.
The author is grateful to Professor Serge Alinhac 
for his helpful suggestions, and to Professor Hideaki Sunagawa for fruitful conversations. 

This research is partially supported by Grant-in-Aid for Scientific Research (C)
(No.~20540211), Japan Society for the Promotion of Science.

%%%%%%%%%%%%%%%%%%%%%%%%%%%%%%%%%%%%%%%%%%%%%%%%

%%%%%%%%%%%%
\begin{flushleft}
{\sc Department of Mathematics, Wakayama University\\
930 Sakaedani, Wakayama 640-8510, Japan} \\
e-mail: katayama@center.wakayama-u.ac.jp
\end{flushleft}
\end{document}